\documentclass[12pt,a4paper,reqno]{amsart}
\usepackage[utf8]{inputenc}    
\usepackage[T1]{fontenc}       

\usepackage[margin=3cm]{geometry}
\usepackage[colorlinks=true]{hyperref}
\usepackage[dvipsnames]{xcolor} % colors
\usepackage{amsmath,amssymb,amsthm,enumerate,mathtools,stmaryrd}
\usepackage{mathbbol}
\usepackage{tikz}
\usetikzlibrary{decorations.pathmorphing}
\usetikzlibrary{arrows.meta, positioning}
\usepackage{float}
\usepackage{booktabs}

\usepackage{cjhebrew}
\newcommand{\hej}{\text{{\<h>}}}

\newcommand{\oM}{\overline{\mathcal{M}}}

\newcommand{\suma}{\mathbf{a}}

\newcommand{\bOm}{{\mathbb{\Pi}}}
\newcommand{\bPsi}{{\mathbb{\Psi}}}
\newcommand{\bA}{{\mathbb{A}}}

\newcommand{\Q}{\mathbb{Q}}
\newcommand{\C}{\mathbb{C}}
\newcommand{\Z}{\mathbb{Z}}
\newcommand{\Coeff}[1]{\mathop{\mathrm{Coeff}}\nolimits_{\left[#1\right]}}

\newcommand{\SRT}{\mathrm{SRT}}
\newcommand{\bD}{{\mathbb{D}}}

\DeclareMathOperator{\DR}{DR}
\def\d{{\partial}}

\usepackage[textsize=tiny]{todonotes}	% todos
\presetkeys{todonotes}{color=Apricot}{} 
\numberwithin{equation}{section}

\definecolor{airforceblue}{rgb}{0.36, 0.54, 0.66}
\definecolor{azure}{rgb}{0.0, 0.5, 1.0}

\hypersetup{
 colorlinks=true,
 citecolor=azure,
 linkcolor=airforceblue,
 urlcolor=Blue}

\newtheorem{theorem}{Theorem}

\newtheorem{lemma}{Lemma}[section]
\newtheorem{proposition}[lemma]{Proposition}
\newtheorem{corollary}[lemma]{Corollary}

\theoremstyle{remark}

\newtheorem{remark}[lemma]{Remark}

\theoremstyle{definition}

\newtheorem{definition}[lemma]{Definition}

\usepackage[
    backend=biber,
    isbn=false,
    url=false,
    giveninits=true,
    style=alphabetic,
    sorting=nyt,
    maxalphanames=99,
    minalphanames=1,
    maxbibnames=99
]{biblatex}

\AtEveryBibitem{%
  \clearfield{issn}%
  \clearlist{language}%
}

\begin{document}

%\title[New integrable systems for CohFTs]{New integrable systems associated to cohomological field theories}

\title[Integrable observables for cohomological field theories]{Beyond descendants: integrable observables for cohomological field theories}

\author[X.~Blot]{Xavier Blot}
\address{X.~B.: Korteweg-de Vriesinstituut voor Wiskunde, Universiteit van Amsterdam, Postbus 94248, 1090GE Amsterdam, Nederland}
\email{x.j.c.v.blot@uva.nl}	

\author[D.~Lewa\'nski]{Danilo Lewa\'nski}
\address{D.~L.: Sezione di Matematica, Dipartimento MIGe, Universit\`a degli studi di Trieste, via Valerio 12/1, 34127	Trieste (IT) \& Istituto Nazionale di Fisica Nucleare (INFN), Sezione di Trieste (IT)}
\email{danilo.lewanski@units.it} 

\author[S.~Shadrin]{Sergey Shadrin}
\address{S.~S.: Korteweg-de Vriesinstituut voor Wiskunde, Universiteit van Amsterdam, Postbus 94248, 1090GE Amsterdam, Nederland}
\email{s.shadrin@uva.nl}	

\begin{abstract}
We introduce the concept of integrable observables and propose them as alternatives to the standard Witten's psi classes (a.k.a.~descendants in $2D$ quantum gravity) to be coupled with cohomological field theories and their generalisations. The main property of integrable observables is that they retain the integrability properties. 

We present three examples of integrable observables. The first two recover the Dub\-ro\-vin–Zhang and double ramification hierarchies, while revealing new structural features in this framework. The third, a new example, builds on recently established properties of the so-called $\mathbb{\Pi}$-class, extending them and placing this class naturally within the theory of integrable systems. 

Notably, our integrable observables framework yields a proof that the new $\mathbb{\Pi}$-hierarchies are Miura equivalent both to the Dub\-ro\-vin–Zhang hierarchies and to the double ramification hierarchies. A new very short proof of Witten's conjecture is also provided.
\end{abstract}

\maketitle
\tableofcontents

\setcounter{section}{-1}
\section{Preface: history and context} 

In 1991 Witten~\cite{Witten-KdV} conjectured that the generating function of the intersection numbers of $\psi$-classes on the moduli spaces of curves $\oM_{g,n}$ solves the Korteweg-de Vries integrable hierarchy of evolutionary PDEs (for history and a new proof, see Sec.~\ref{sec: new proof WK}). This started an intensive study of the interaction of different mathematics avatars of topological string theory, such as quantum singularity theory and Gromov-Witten theory (formalized as cohomological field theories, or CohFTs, introduced in~\cite{KontsevichManin}), with integrable systems. 

\subsection{Dubrovin-Zhang theory} Dubrovin and Zhang came up in~\cite{Dubrovin-Zhang-g1} with a classification project based on the theory of Frobenius manifolds, introduced by Dubrovin in~\cite{Dub-TFT}. Their approach was fully developed it in~\cite{dubrovin2001normalformshierarchiesintegrable} proposing a formal system of axioms that should identify a unique bi-Hamiltonian tau-symmetric dispersive integrable system associated to an input Frobenius manifold, whose dispersionless limit is the principal hierarchy of Dubrovin~\cite{Dub-TFT}. They proved uniqueness, but not existence, which was fully proved much later in~\cite{LiuWangZhang}. An excellent concise survey of their approach and important intermediate steps is available in~\cite[Sec.~1]{lorenzoni2026generalisedbihamiltonianstructureshydrodynamic}.

Dubrovin and Zhang also identified in~\cite{dubrovin2001normalformshierarchiesintegrable} a particular tau function of their (back then conjectural) integrable systems with the Givental formula~\cite{Giv} canonically associated to a calibrated Frobenius manifold. By Teleman's classification~\cite{Tel}, there is a unique semi-simple homogeneous cohomological field theory associated to a given Frobenius manifold, with Givental's formula giving this partition function coupled to $\psi$-classes. This linked Dubrovin-Zhang theory to the realm of cohomological field theories.

A natural question back then (besides the existence of the Dubrovin-Zhang hierarchies, which was reduced to the polynomiality property in~\cite{dubrovin2001normalformshierarchiesintegrable}, although left as an open problem) was whether the Dubrovin-Zhang hierarchies can also be defined for a wider class of CohFTs. Indeed, according to~\cite{Tel} the homogeneous semi-simple CohFTs in the realm of all CohFTs form just isolated points, or, in some cases, finite dimensional subvarieties inside an infinite dimensional space. 

To resolve this issue, the Hamiltonian tau-symmetric integrable systems for the partition function for general semi-simple CohFTs was constructed in~\cite{BPS1,BPS2}. By the uniqueness result of Dubrovin and Zhang, this also settled the existence problem for all structures conjectured in~\cite{dubrovin2001normalformshierarchiesintegrable}, with the exception of existence of the second Hamiltonian structure. Indeed, the second Hamiltonian structure does demand homogeneity and is out of reach of any method that doesn't employ it.

It is important to stress the difference between the axiomatic approach of Dubrovin-Zhang and all subsequent developments, that we discuss here and below. The axiomatic approach of~\cite{dubrovin2001normalformshierarchiesintegrable,LiuWangZhang} works in a much more restrictive setting: it in fact requires homogeneity. On the other hand, it only requires \emph{genus zero} data, providing a remarkable full scale alternative to the Givental-Teleman theory, with an extra advantage that it remains entirely in the realm of integrable systems. Moreover, it offers a large variety of extra tools to control the higher genus data, just in terms of the genus zero input (see e.g.~\cite{ShaWang}). All other approaches, in particular the one of~\cite{BPS1,BPS2}, demand to take as input the CohFT (or some of its generalised versions, as we discuss below) in \emph{all genera}.

The construction of~\cite{BPS1,BPS2} is based on the so-called tautological relations proposed in Eguchi-Xiong~\cite{EguchiXiong}. Thus, it can be applied to non-semi-simple CohFTs, and more generally to the so-called partial CohFTs~\cite{LiuRuanZhang}, although in these cases the methods of \emph{op.~cit.} do not allow to establish the polynomiality of the equations, of the Hamiltonians, and of the Poisson bracket. A different but equivalent construction was proposed in~\cite{BS22}, where these properties were reduced to a system of different tautological relations, later proved in~\cite{BSS25} building up on the key insights in~\cite{BLS-DRDZ}.

In short, one can summarise Dubrovin-Zhang theory as follows: 
\begin{itemize}
    \item The approach of~\cite{dubrovin2001normalformshierarchiesintegrable,LiuWangZhang} works only for semi-simple homogenenous CohFTs, but uses only genus zero data as input. Because of the homogeneity assumption, it does provide bi-Hamiltonian structures.
    \item The approach of~\cite{BPS1,BPS2} works for semi-simple CohFTs without assumption of homogeneity, but it requires as input all genera data, lacks bi-Hamiltonian structure, and the methods of proofs in these papers are not working outside the realm of semi-simple CohFTs (although the construction itself can be extended to general partial CohFTs).
    \item The approach of~\cite{BS22,BLS-DRDZ,BSS25} works for general partial CohFTs. It does use as input all genera data, it lacks bi-Hamiltonian structure, but it ultimately gives the most general result in this context. 
\end{itemize}
In all these cases we call the constructed integrable systems the \emph{Dubrovin-Zhang hierarchies}.

\subsection{Buryak's double ramification hierarchies} Buryak proposed in~\cite{Bur} a drastically different construction of a Hamiltonian integrable system associated to a CohFT, based on the so-called double ramification (DR) cycles, which are explicit cohomology classes on the moduli space of stable curves. These integrable systems are usually called the \emph{DR hierarchies} in the literature. The construction of Buryak might be seen as inspired by symplectic field theory context~\cite{Eliashberg-SFT}, and it works for any CohFT, without any restrictions. It was noticed in~\cite{BR-spin} that it can be applied to partial CohFTs as well, also without any further assumption. 

Buryak conjectured that DR hierarchies are Miura equivalent to the Dubrovin-Zhang hierarchies. This conjecture was known under the name of DR/DZ equivalence. Initially the conjecture was posed in~\cite{Bur} in the semi-simple case, but it was further developed to a much more precise statement that also identified tau structures and gave a concrete path towards a proof in~\cite{BDGR1,BDGR20,BGR19,BS22}. In particular, the latter reference formulated it as a conjecture for partial CohFTs (alongside the conjecture of the existence of the Dubrovin-Zhang hierarchies for partial CohFTs, as we mentioned above).

The DR/DZ equivalence was proved for semi-simple CohFTs in~\cite{BLS-DRDZ}, and for a general partial CohFT it was reduced in \emph{op. cit.} to the so-called master relation, which was then proved in~\cite{BSS25}.

\subsection{Beyond Hamiltonian structure} A natural question, first studied in the dispersionless limit, addresses which axioms of the Frobenius manifold (designed to classify Dubrovin's principal hierarchies, which are bi-Hamiltonian and tau symmetric) can be dropped without affecting integrability. In~\cite{LorenzoniPedtoniRaimondo} Dubrovin's construction was generalized to a correspondence between the so-called \emph{flat F-manifolds} (known in the literature outside integrable systems under the name \emph{hypercommutative algebras with a unit}~\cite{Getzler}) and integrable systems of conservation laws. This correspondence delivered a rich variety of new structural results, examples, and conjectures, see e.g.~\cite{ArsLor1,ArsLor2,lorenzoni2026generalisedbihamiltonianstructureshydrodynamic}.

It was noticed in~\cite{BR-spin} that Buryak's DR construction can be applied to produce integrable systems of conservation laws from the so-called F-CohFTs, the higher genera analog of flat F-manifolds. This approach was studied in detail in~\cite{ABLR}. In~\cite{BS22} the Dubrovin-Zhang theory for general partial CohFTs was also proposed for F-CohFTs and the existence (polynomiality) of the corresponding system of conservation laws and equivalence to the DR construction of conservation laws were reduced to a system of tautological relations proved subsequently in~\cite{BLS-DRDZ,BSS25}.

Note that the approach of~\cite{BS22} uses \emph{all genera} data input. A proposal on how to restore the original Dubrovin-Zhang axiomatic approach in the context of flat F-manifolds, that would employ purely \emph{genus zero} data, was recently developed in~\cite{lorenzoni2026generalisedbihamiltonianstructureshydrodynamic}.

\subsection{The \texorpdfstring{$\bOm$}{hhej} class} \label{sub:sec:hej} The study of the tau function of the DR hierarchy, and in particular its quantum analog~\cite{BDGR20,Blot,BLS-Quantum} led the first two authors of the present paper to realize that many correlators can be alternatively described by the intersection numbers with the so-called $\hej$-classes, which is a particular family of specializations of the $\Omega$-classes, or Chiodo classes, whose Chern characters were introduced and computed in~\cite{Chiodo08}, and it was considered as and proved to be a CohFT in~\cite{LPSZ17}. See e.g. \cite{GLN23} for a summary on their properties. In the case of the \emph{classical} (as opposed to quantum, which we do not discuss in this paper) correlators one has to multiply the $\hej$ classes by the top Chern class of the Hodge bundle, and the resulting classes are called $\bOm$-classes in~\cite{BLS-Omega}.

In Dubrovin-Zhang theory (both for partial CohFTs and F-CohFTs) the partition functions are always described in terms of intersections with $\psi$-classes. In the DR approach to integrability it was shown that the partition functions can be described via the so-called $\bA$-classes~\cite{BDGR20,BGR19,BS22}. It was noticed in~\cite{BLRS,BLS-Omega} that the $\bA$-classes interact with $\bOm$ classes in exactly the same way as it was conjectured to interact with $\psi$-classes, and in fact this similarity was the main source of ideas that led to the proofs of existence of the Dubrovin-Zhang hierarchies and their equivalence to the DR hierarchies in the most general situation in~\cite{BLS-DRDZ,BSS25}.

A natural question is whether the $\bOm$ classes carry integrability intrinsically. In the present paper we answer this question in the affirmative by revisiting the notion of coupling to gravity, traditionally understood as intersection with $\psi$-classes, and introducing the concept of \emph{integrable observables}. This framework places $\psi$-, $\bA$-, and $\bOm$-classes on equal footing as integrable observables, and provides a uniform mechanism that recovers the Dubrovin–Zhang and double ramification hierarchies, while also yielding a new class of integrable systems, that deserves further investigation per sé. It would be also interesting to know whether integrable observables different from $\psi$-classes, as e.g. $\bOm$-classes, have some physics meaning arising from $2D$ quantum gravity.

%%%%%%%%%%%%%%%%%%%%%%%%%%%%%%%%%%%%%%%%%%%%%%%%%%%%%%%%%%%%%%%%%%%
\vspace{1cm}
\section{Introduction}
\label{sec:intro}
%%%%%%%%%%%%%%%%%%%%%%%%%%%%%%%%%%%%%%%%%%%%%%%%%%%%%%%%%%%%%%%%%%%

\subsection{Cohomological field theories and their generalisations} 
There are several natural algebraic structures on the collection of the moduli spaces of curves $\{\oM_{g,n}\}$, $g\geq 0$, $n\geq 0$, $2g-2+n>0$, that are relevant for the connection to integrable systems: a modular operad, a cyclic operad graded by genus, or, once we distinguish one of the marked point and consider the spaces $\{\oM_{g,n+1}\}$, an operad graded by genus. It is not the exhaustive list, one can also consider the most natural structure of properad or wheeled prop, but they are not connected to fundamental types of integrability.

Let $V$ be a finite dimensional vector space over $\C$ equipped with a non-degenerate bilinear symmetric map $\eta\colon V^{\otimes 2} \to \C$. Then $\{\mathrm{Hom}(V^{\otimes n}, \C)\}$ can be equipped with a structure of a cyclic operad, or a modular operad, and the representations of $\{H_\bullet (\oM_{g,n}), \C\}$ in $\{\mathrm{Hom}(V^{\otimes n}, \C)\}$ are called \emph{partial cohomological field theories (P-CohFTs)} in the case of genus graded cyclic operad structure or \emph{cohomological field theories (CohFTs)} in the case of modular operad structure. 

The latter one was proposed by Kontsevich and Manin in~\cite{KontsevichManin} as an axiomatic way to capture the key properties of Gromov-Witten theories. It is the most studied object in this context, but it turns out that for connections to integrability the less restrictive and more general concept of partial cohomological field theory is already sufficient. 

Let $V$ be a vector space over $\C$, not necessarily finite dimensional. $\{\mathrm{Hom}(V^{\otimes n}, V)\}$ is canonically equipped with a structure of an operad and the representations of $\{H_\bullet (\oM_{g,n+1}), \C\}$ in $\{\mathrm{Hom}(V^{\otimes n}, V)\}$ are called \emph{F cohomological field theories (F-CohFTs)}. 

This paper focuses on universal integrability properties of P-CohFTs and F-CohFTs, and the main goal of this paper is to break the tradition that goes back to Witten~\cite{Witten-KdV} to couple cohomological field theories to gravity using, as the fundamental system of observables, the so-called $\psi$-classes (gravitational descendants).

Before we proceed, a few remarks and conventions are in order. First of all, in the context of operadic structures it is natural to consider $\Z_2$-graded P-CohFTs and F-CohFTs (represented in $\Z_2$ vector spaces), and it does lead to non-trivial integrable systems as well. But we omit this more general framework for simplicity, in particular, we assume that we only take even classes on moduli spaces.
%, that is, $\bullet\in 2\Z_{\geq 0}$ in the notation above. 
Second, we solely use P-CohFTs and F-CohFTs in the rest of the text, but each time we state anything for P-CohFTs, it does work also for usual CohFTs. Third, we focus on P-CohFTs and F-CohFTs equipped with a \emph{unit} that we introduce below.

\subsection{Precise definitions of P-CohFT and F-CohFT with a unit}

Let $V$ be a finite dimentional vector space with a chosen vector $e\in V$ and a non-degenerate bilinear symmetric map  $\eta\in \mathrm{Hom}(V^{\otimes 2},\C)$.
A ($V$-valued) \emph{P-CohFT} is a system of linear maps of vector spaces
\begin{align}
    \mathfrak{pc}_{g,n}\colon V^{\otimes n} \to H^\bullet (\oM_{g,n},\C)
\end{align}
satisfying the following conditions:
\begin{description}
    \item[commutativity]  $\mathfrak{pc}_{g,n}$ is $\mathfrak{S}_n$-equivariant with respect to the action on $V^{\otimes n}$ by permutation of the factors and the action on $H^\bullet (\oM_{g,n},\C)$ induced by the automorphisms of $\oM_{g,n}$ caused by permutation of labels of the marked points. 
    \item[flat unit] Let $\pi\colon \oM_{g,n+1}\to \oM_{g,n}$ be the forgetful morphisms, that is, the map that forgets the last marked point and stabilizes the curve. Then $\pi^*\mathfrak{pc}_{g,n} = \mathfrak{pc}_{g,n+1}(e)$, where $e$ is fed as the last argument. Moreover, $\mathfrak{pc}_{0,3}(e) = \eta$. 
    \item[factorization] Let $\mathrm{gl}\colon \oM_{g_1,n_1+1}\times\oM_{g_2,n_2+1} \to \oM_{g_1+g_2,n_1+n_2}$ be the gluing morphisms, that is, the map that creates a nodal curve out of two curves by identifying into a node their last marked points and identifying the labels on the first and the second component with the ordered subsets $I_1,I_2\subseteq \{1,\dots,n_1+n_2\}$, $|I_1|=n_1$, $|I_2|=n_2$. Then 
    \begin{align}
        & \mathrm{gl}^*\mathfrak{pc}_{g_1+g_2,n_1+n_2}\Big(\Vec{\bigotimes}_{i=1}^{n_1+n_2} v_i\Big) = \\ \notag & \eta^{\alpha\beta} \mathfrak{pc}_{g_1,n_1+1}\Big(\Vec{\bigotimes}_{i\in I_1} v_i \otimes e_\alpha \Big) \otimes \mathfrak{pc}_{g_2,n_2+1}\Big(\Vec{\bigotimes}_{i\in I_2} v_i \otimes e_\beta \Big),
    \end{align}
    where by $\Vec{\bigotimes}$ we denote the ordered tensor product with an index that runs over an ordered set, and $\eta^{\alpha\beta} e_\alpha\otimes e_\beta \in V^{\otimes 2}$ is a bivector $\eta^{-1}$ dual to $\eta$ expressed in the basis of $V = \langle e_1=e, e_2,\dots,e_N\rangle$. Here and below we always assume the summation over repeated Greek indices. 
\end{description}

\begin{remark} In order to define a more traditional concept of CohFT one has to enhance the factorization property with one more condition. Let $\mathrm{gl'}\colon \oM_{g-1,n+2}\to\oM_{g,n}$ that creates a nodal curve of the arithmetic genus one higher by identifying the $(n+1)$-st and $(n+2)$-nd marked point into a node. Then a P-CohFT is a CohFT if we additionally require that 
  \begin{align}
        & \mathrm{gl'}^*\mathfrak{pc}_{g,n}\Big(\Vec{\bigotimes}_{i=1}^{n} v_i\Big) =  \eta^{\alpha\beta} \mathfrak{pc}_{g-1,n+2}\Big(\Vec{\bigotimes}_{i=1}^{n} v_i \otimes e_\alpha \otimes e_\beta\Big).
    \end{align}
\end{remark}

Examples of P-CohFTs include Gromov-Witten theories of target varieties, Witten's $r$-spin classes and more general classes coming from quantum singularity theory of Fan-Jarvis-Ruan, etc. (which all are also examples of CohFTs). The authentic examples of P-CohFTs that are not CohFTs come, for instance, from the study of quantum singularity theory with symmetries, see~\cite{LiuRuanZhang} where this concept was first introduced.

Let $V$ be a vector space with a chosen vector $e\in V$. It is now convenient to assume that the marked points of moduli spaces $\oM_{g,n+1}$ are labeled from $0$ to $n$. A ($V$-valued) \emph{F-CohFT} \cite{ABLR,ABLR-2} is a system of linear maps of vector spaces 
\begin{align}
    \mathfrak{fc}_{g,n}\colon V^{\otimes n} \to V\otimes H^\bullet (\oM_{g,n+1},\C)
\end{align}
satisfying the following conditions:
\begin{description}
    \item[commutativity] $\mathfrak{fc}_{g,n}$ is $\mathfrak{S}_n$-equivariant with respect to the action on $V^{\otimes n}$ by permutation of the factors and the action on $H^\bullet (\oM_{g,n+1},\C)$ induced by the automorphisms of $\oM_{g,n+1}$ caused by permutation of the labels $\{1,\dots,n\}$ of the marked points (that is, the zeroth marked point is left intact). 
    \item[flat unit] Let $\pi\colon \oM_{g,n+2}\to \oM_{g,n+1}$ be the forgetful morphisms, that is, the map that forgets the last marked point and stabilizes the curve. Then $\pi^*\mathfrak{fc}_{g,n} = \mathfrak{fc}_{g,n+1}(e)$, where $e$ is fed as the last argument. Moreover, $\mathfrak{fc}_{0,2}(e) = \mathrm{Id}_V$. 
    \item[factorization] Let $\mathrm{gl}\colon \oM_{g_1,n_1+2}\times\oM_{g_2,n_2+1} \to \oM_{g_1+g_2,n_1+n_2+1}$ be the gluing morphisms, that is, the map that creates a nodal curve out of two curves by identifying into a node the last marked point on the first curve with the zeroth marked point on the second curve and identifying the remaining non-zero labels on the first component and all non-zero labels on the second component with the ordered subsets $I_1,I_2\subseteq \{1,\dots,n_1+n_2\}$, $|I_1|=n_1$, $|I_2|=n_2$. Then 
    \begin{align}
        & \mathrm{gl}^*\mathfrak{fc}_{g_1+g_2,n_1+n_2}\Big(\Vec{\bigotimes}_{i=1}^{n_1+n_2} v_i\Big) = \\ \notag & \mathfrak{fc}_{g_1,n_1+1}\Big(\Vec{\bigotimes}_{i\in I_1} v_i \otimes \mathfrak{fc}_{g_2,n_2}\Big(\Vec{\bigotimes}_{i\in I_2} v_i \Big) \Big),
    \end{align}
    where we assume that the substitution of $\mathfrak{fc}_{g_2,n_2}\Big(\Vec{\bigotimes}_{i\in I_2} v_i \otimes e_\beta \Big)$ as an argument in $\mathfrak{fc}_{g_1,n_1+1}$ acts as a right module with respect to the tensor multiplication by classes in $H^\bullet (\oM_{g_2,1+n_2},\C)$, that is, the right hand side takes values in $V\otimes H^\bullet (\oM_{g_1,1+n_1+1},\C) \otimes H^\bullet (\oM_{g_2,1+n_2},\C)$.
\end{description}

In genus $0$ this concept is reduced to the so-called flat F manifolds, or, in the formal setup and without unit, to hypercommutative algebras~\cite{LorenzoniPedtoniRaimondo,Getzler}.
Examples of F-CohFTs can be obtained, for instance, by forgetting part of the structure of given P-CohFTs. It is a canonical construction, and it reads
\begin{align} \label{eq:From-PCohFT-to-FCohFT}
    \eta\bigg(v_0,\mathfrak{fc}_{g,n}\Big(\Vec{\bigotimes}_{i=1}^{n} v_i\Big)\bigg)
    =
    \bigg[\mathfrak{pc}_{g,n+1}(\Big(\Vec{\bigotimes}_{i=1}^{n} v_i \otimes v_0\Big)\bigg]^{\circlearrowright},
\end{align}
where $[\cdots]^\circlearrowright$ is the operation of relabeling marked points 
%of ``$\mod (n+1)$ relabeling'' of the marked points 
that makes the 
%$(n+1)$-st 
last marked point the zeroth one, which is needed to match the convention we used in the definition of F-CohFT.

\begin{remark} \label{rem:infinite-dim} Both P-CohFTs and F-CohFTs can be considered with an infinite-dimensional target space $V$ (for P-CohFTs one might need to replace the bilinear symmetric map $\eta$ with a symmetric $2$-tensor in the basic setup), and there are meaningful examples, cf.~\cite{BR-partial}. However, to simplify the presentation of the basic constructions below one might always assume that $V$ is finite dimensional. Moreover, we shall always work with a fixed basis $V=\langle e_1,\dots,e_N\rangle$, $N=\dim V$, and assume that $e=e_1$. 
\end{remark}

\subsection{Tautological classes and coupling to gravity} The P-CohFTs and F-CohFTs can provide physically relevant information once one integrates the corresponding cohomology classes over arbitrary given cycles on the moduli spaces of curves. The structure of the cohomology of $\oM_{g,n}$ over $\C$ is quite involved (even in the case one is interested just in the even degree part of it), and one typically restricts attention to the cycles that are Poincar\'e dual to the so-called \emph{tautological classes}.

The cohomological tautological rings $RH^\bullet(\oM_{g,n})$ are defined as the minimal system of subalgebras of $H^\bullet(\oM_{g,n},\Q)$ that is closed under the push-forwards with respect to all natural maps between the moduli spaces of curves, namely, the projections $\pi\colon \oM_{g,n+1}\to \oM_{g,n}$, the gluing maps $\mathrm{gl}\colon \oM_{g_1,n_1+1}\times\oM_{g_2,n_2+1} \to \oM_{g_1+g_2,n_1+n_2}$, and the gluing maps $\mathrm{gl'}\colon \oM_{g-1,n+2}\to\oM_{g,n}$. 

 Of particular importance are the so-called $\psi$-classes, $\psi_i \in RH^2(\oM_{g,n})$, $i=1,\dots,n$, defined as follows: $\psi_i$ is the first Chern classes of the line bundle over $\oM_{g,n}$ whose fiber at a point in $\oM_{g,n}$ represented by a stable curve $C$ with the marked points $x_1,\dots,x_n\in C$ is given by $T^*_{x_i}C$. 

\begin{remark} It is often convenient to consider the Chow version of the tautological rings denoted by $R^\bullet(\oM_{g,n})$. In these cases we use the Chow grading which is half of the cohomological grading, e.g., $\psi_i \in R^1(\oM_{g,n})$. 
\end{remark}

The partition function that stores the information about the correlators of a given $(V,\eta)$-valued P-CohFT is defined as
\begin{align}
    Z & = \exp \big(\epsilon^{-2} F \big) = \exp\bigg( \sum_{g=0}^\infty \epsilon^{2g-2} F_g(\{t^{\alpha,d}\}) \bigg)
    \\ \notag &\coloneqq \exp\bigg( \sum_{g,n} \frac{\epsilon^{2g-2}}{n!} \sum_{d_1,\dots,d_n=0}^\infty \int_{\oM_{g,n}} \mathfrak{pc}\Big( \Vec{\bigotimes}_{i=1}^n e_{\alpha_i}\Big) \prod_{i=1}^n \psi_i^{d_i} \prod_{i=1}^n t^{\alpha_i,d_i} \bigg)
\end{align}
where $e_\alpha$, $\alpha =1,\dots,N$, form a basis of $V = \langle e_1,\dots,e_N\rangle$, $\epsilon$ and $t^{\alpha,d}$, $d\geq 0$, are formal variables. 

The particular choice of $\psi$-classes as the key ``observables'' to form a partition function $Z$ has, beyond its motivation from theoretical physics, at least three deep mathematical reasons:
\begin{itemize}
    \item If $\{\mathfrak{pc}_{g,n}\}$ is a CohFT (to be a P-CohFT is not enough for this property), then the full collection of integrals $\int_{\oM_{g,n}} \mathfrak{pc}\big( \Vec{\bigotimes}_{i=1}^n e_{\alpha_i}\big) \prod_{i=1}^n \psi_i^{d_i}$ determines uniquely and with explicit formulas all integrals of the shape $\int_{\oM_{g,n}} \mathfrak{pc}\big( \Vec{\bigotimes}_{i=1}^n e_{\alpha_i}\big) \cdot\mathfrak{a}$, where $\mathfrak{a}\in R^\bullet(\oM_{g,n})$ is an arbitrary tautological class. 
    \item If $\{\mathfrak{pc}_{g,n}\}$ is a CohFT with extra properties (it must be homogeneous), then it is expected that $Z$ satisfies the so-called Virasoro constraints, that is, $Z$ is annihilated by a system of differental operators of a particular shape that form a half of the Virasoro algebra. It is proved in some situations, but remains an open conjecture in general. 
    %for semi-simpleConjecturally, this property holds in more general situation(s) as well. 
    \item \emph{With no extra assumptions} on a P-CohFT $\{\mathfrak{pc}_{g,n}\}$, $Z$ appears to be a tau function of a Hamiltonian integrable systems whose dispersionless limit is the so-called principal hierarchy, which is a Hamiltonian integrable system of hydrodynamic type~\cite{BS22,BLS-DRDZ,BSS25}, see Section \ref{sec:integrable} for more details.
\end{itemize}

The third property (\emph{integrability}) is especially remarkable since it does not require any extra assumptions on a P-CohFT and it allows to fully embed the study of P-CohFTs into the realm of Hamiltonian dispersive deformations of hydrodynamic integrable systems. In fact, it is a manifestation that virtually all problems in enumerative geometry possess strong integrability properties that can be used to universally resolve them. 

It is also the only property that survives when one considers a $V$-valued F-CohFTs. More precisely, in this case one can assemble the generating series for the correlators into the so-called vector potential of an F-CohFT, which is indeed a system of formal vector fields on the loop space of $V^*$. For clarity, let us assume that $V$ is finite dimensional with a fixed basis $\{ e_1,\dots,e_N\}$, and let $\{e^{1},\dots,e^N\}$ be the dual basis. Then define 
\begin{align}
    X & = \sum_{g,n} \frac{\epsilon^{2g}}{n!} \sum_{d_1,\dots,d_n=0}^\infty \int_{\oM_{g,1+n}} e^{\alpha_0} \bigg(\mathfrak{fc}\Big( \Vec{\bigotimes}_{i=1}^n e_{\alpha_i}\Big) \bigg) %\psi_0^{d_0} 
    \prod_{i=1}^n \psi_i^{d_i} \prod_{i=1}^n t^{\alpha_i,d_i} \frac{\partial}{\partial t^{\alpha_0,0}}
\end{align}
This vector field %(for $d_0=0$) 
serves as a vector potential related in a particular universal way (mimicking the properties of a tau function) to a topological solution of an integrable system of conservation laws. 

We focus on the integrability property and the main message of this paper is that in this regard the $\psi$-classes are not as exceptional as it might seem from the Witten's theory of topological gravity, and we formulate general conditions and provide examples of two more systems of classes that turn partition functions of P-CohFTs into tau functions of Hamiltonian integrable systems and vector potentials of F-CohFTs into integrable systems of conservation laws. 

\subsection{Integrable observables} The goal of this section is to formulate a list of formal requirements for a collection of classes in $R^\bullet(\oM_{g,n})$ that allows to replace the monomials of $\psi$-classes in the discussion above retaining the integrability properties of P-CohFTs and F-CohFTs. The definition below is motivated by a practical guide to construction of integrable systems from P-CohFTs and F-CohFTs that we explain in the next sections. 

\begin{definition} [Integrable observables] \label{def:int-obs}
    We call a collection of polynomial-valued classes 
    \begin{align}
        O_{g,n}(a_1,\dots,a_n) \in R^\bullet (\oM_{g,n})\otimes_\Q \Q[a_1,\dots,a_n],        
    \end{align}
 $g\geq 0$, $n\geq 1$, $2g-2+n>0$, \emph{integrable observables} if it satifies the following list of properties:
\begin{enumerate}
    \item \emph{String and dilaton equations:} Let $\pi\colon \oM_{g,n+1}\to \oM_{g,n}$ be the map that forgets the last marked point. Then
    \begin{align} \label{eq:push-forward-O}
        \pi_* O_{g,n}(a_1,\dots,a_n,b) = \left(\sum_i a_i + b\left(2g-2+n\right)\right) O_{g,n}(a_1,\dots,a_n) + O(b^2).
    \end{align}
    Here the string equation is obtained by taking the constant term in $b$ of this relation and the dilaton equation is obtained by extracting the coefficient of $b^1$. We also call this property the \emph{push-forward property}. 
    \item \emph{Genus 0 normalization:} 
    \begin{align}
        O_{0,n}(a_1,\dots,a_n) = \prod_{i=1}^n \frac{1}{1-a_i\psi_i}.
    \end{align}
    \item \emph{Homogeneity:} 
        \begin{equation}
        \Coeff{\prod_{i=1}^n a_i^{d_i}} O_{g,n}(a_1, \dots, a_n) \in R^{\sum d_i} (\oM_{g,n}),
        \end{equation} 
        where $\Coeff{\prod_{i=1}^n a_i^{d_i}}$ denotes the operator that extracts the coefficient of the monomial $\prod_{i=1}^n a_i^{d_i}$ from the polynomial next to it.

    \item \emph{Leveled rooted tree relations LRT-$1$ and LRT-$2$:} %for $g,n,m\geq0$ such that $2g-2+n+m>0$
    %\begin{align}
    %    \deg_a \left({}^{O}\!B_{g,n}^{m} - \bA_{g,n}^{m}\right) \leq 2g-2+m.
    %\end{align}
    these are some complicated properties of the classes that involve the sum over stable rooted trees with a level structure. We introduce them in Sec.~\ref{subsec: stable rooted tree relations}, see Def.~\ref{def: stable rooted tree relations}. 
\end{enumerate}

Both leveled rooted tree relations LRT-$1$ and LRT-$2$ define integrable observables, but the properties of the associated integrable systems are going to be different. When we need to specify what type of leveled rooted tree relations is used, we refer to \emph{type $1$ integrable observables} or \emph{type $2$ integrable observables}, respectively.  
\end{definition}

The definition might look complicated at first glance, especially the leveled rooted tree relations that are still to be defined, but it is motivated by the fact that it properly stores the universal integrability properties.  In particular, we have the following. 

\begin{theorem}[Precise version: Proposition~\ref{prop: integrability of integrable observables for F-CohFT}] \label{thm: integrable observables vector potential}
(1) For any integrable observables $\{O_{g,n}\}$, and any F-CohFT $\{\mathfrak{fc}_{g,n}\}$, the vector field 
\begin{align}
    \label{vector potential integrable observable}
    {}^O\!X & \coloneqq \sum_{g,n} \frac{\epsilon^{2g}}{n!} \sum_{d_1,\dots,d_n=0}^\infty \int_{\oM_{g,n+1}} e^{\alpha_0} \bigg(\mathfrak{fc}_{g,n}\Big( \Vec{\bigotimes}_{i=1}^n e_{\alpha_i}\Big) \bigg)
    \times \\ \notag & \qquad \qquad \qquad 
    \Coeff{\prod_{i=1}^n a_i^{d_i}} O_{g,n+1}(0,a_1,\dots,a_n) \prod_{i=1}^n t^{\alpha_i,d_i} \frac{\partial}{\partial t^{\alpha_0,0}}
\end{align}
is the vector potential of the so-called topological solution of an integrable system of conservation laws. Moreover, in the dispersionless limit $\epsilon\to 0$ the integrable system is independent of the choice the integrable observable. 

(2) If the integrable observables are of type $1$ then this integrable system is Miura equivalent to the DR hierarchy of conservation laws associated to $\{\mathfrak{fc}_{g,n}\}$. 

(3) If the integrable observable are of type $2$, than the fluxes of this integrable system can be described by explicit formulas. 
\end{theorem}

\begin{theorem}[Precise version: Proposition~\ref{prop: integrability of integrable observables for P-CohFT}] \label{thm: integrable observables tau function} (1) For any integrable observables $\{O_{g,n}\}$, and any P-CohFT $\{\mathfrak{pc}_{g,n}\}$, the partition function
\begin{align}
    \label{tau funtion integrable observable}
    {}^{O}\!Z = \exp\bigg(\frac{{}^O\!F}{\epsilon^2}\bigg)& \coloneqq \exp\bigg( \sum_{g,n} \frac{\epsilon^{2g-2}}{n!} \sum_{d_1,\dots,d_n=0}^\infty \int_{\oM_{g,n}} \mathfrak{pc}_{g,n}\Big( \Vec{\bigotimes}_{i=1}^n e_{\alpha_i}\Big) \times \\ \notag & \qquad \qquad \qquad \Coeff{\prod_{i=1}^n a_i^{d_i}} O_{g,n}(a_1,\dots,a_n) \prod_{i=1}^n t^{\alpha_i,d_i} \bigg)
\end{align}
is the tau function of the so-called topological solution of a Hamiltonian integrable system.  Moreover, in the dispersionless limit $\epsilon\to 0$ the integrable system is independent of the choice the integrable observable.

(2) If the integrable observables are of type $1$ then this Hamiltonian integrable system is normal Miura equivalent to the Hamiltonian DR hierarchy associated to $\{\mathfrak{pc}_{g,n}\}$ considered in its normal coordinates.

(3) If the integrable observable are of type $2$, then the Hamiltonians and the fluxes of this integrable system can be described by explicit formulas.

(4) If the integrable observables are of type $1$ and type $2$ simultaneously, then the normal Miura transformation to the DR hierarchy is also given by explicit formulas.
\end{theorem}

In order to be explained, both theorems need an extra piece of theory. In particular, we have to explain how we pass from tau function and vector potential to integrable systems, what we mean by explicit formulas in the case of type $2$ integrable observables, and we also have to introduce the concepts of Hamiltonian DR hierarchy~\cite{Bur,BDGR1,BR-spin} and DR hierarchy of conservation laws~\cite{BR-spin,ABLR} that 
%are strongly motivated by the constructions in the symplectic field theory~\cite{Eliashberg-SFT} and that 
we use here as canonical reference points. We provide all necessary constructions in Section~\ref{sec:integrable}. 

\begin{remark}
    Towards integrability, one can relax some hypothesis of integrable observables. To get Theorems~\ref{thm: integrable observables vector potential} and~\ref{thm: integrable observables tau function}, we do not require the integrable observables to lie in the tautological rings, it holds in the cohomology or the Chow rings of $\oM_{g,n}$. Second, if we drop the genus $0$ normalization, both theorems remain valid, though the resulting hierarchy need not to share the same dispersionless limit  $\epsilon \to 0$. However, no such example is known.
\end{remark}

\subsection{The leveled rooted tree relations} \label{subsec: stable rooted tree relations} The goal of this Section is to complete Definition~\ref{def:int-obs} by defining the \emph{leveled rooted tree relations} for the classes $O_{g,n}(a_1,\dots,a_n)$. This requires to introduce a notation and a variety of structures for rooted stable trees. 

\subsubsection{Basic notation for trees}
Let $\SRT_{g,n,m}$ be the set of stable rooted trees of total genus $g$, with $n$ regular legs $\sigma_1,\dots,\sigma_n$ and $m$ extra legs $\sigma_{n+1},\dots,\sigma_{n+m}$, which we refer to as ``frozen'' legs and must always be attached to the root vertex. For a $T\in \SRT_{g,n,m}$ we use the following notation:
\begin{itemize}
	\item $H(T)$ is the set of half-edges of $T$.
	\item $L(T),L_r(T),L_f(T)\subset H(T)$ are the sets of all, regular, and frozen legs of $T$, respectively. $L(T) = L_r(T)\sqcup L_f(T)$.
	\item $H_e(T)\coloneqq H(T)\setminus L(T)$.
	\item $\iota\colon H_e(T)\to H_e(T)$ is the involution that interchanges the half-edges that form an edge.
	\item $E(T)$ is the set of edges of $T$, $E\cong H_e(T)/\iota$.
	\item $H_+(T)\subset H(T)$ is the set of the so-called ``positive'' half-edges that consists of all regular legs of $T$ and of half-edges in $H(T)\setminus L(T)$ directed away from the root at the vertices where they are attached,
	$H_+(T)\cong E(T)\cup L_{r}(T)$; 
	\item $H_-(T)\subset H(T)$ is the set of the so-called ``negative'' half-edges that consists of all frozen legs of $T$ and of half-edges in $H(T)\setminus L(T)$ directed towards the root at the vertices where they are attached, $H_-(T)\cong E(T)\cup L_{f}(T)$;
	\item $V(T),V_{nr}(T)$ are the sets of vertices and non-root vertices of $T$. 
	\item $v_r\in V(T)$ is the root vertex of $T$; $V(T)=\{v_r(T)\}\sqcup V_{nr}(T)$.
	\item For a $v\in V(T)$, $H(v),H_+(v),H_-(v)$ are all, positive, and negative half-edges attached to $v$, respectively. Obviously, $|H_-(v_r)|=m$ and for any $v\in V_{nr}(T)$ we have $|H_-(v)|=1$.
	\item For a $v\in V(T)$ let $g(v)\in \Z_{\geq 0}$ be the genus assigned to $v$. The stability condition means that $2g(v)-2+|H(v)|>0$.
	%\[\chi(v)\coloneqq 2g(v)-2+|H(v)|>0.\] 
	The genus condition reads $\sum_{v\in V(T)} g(v) = g.$
	%\[
	%\sum_{v\in V(T)} g(v) = g.
	%\]
	\item We say that a vertex or a (half-)edge $x$ is a descendant of a vertex or a (half-)edge $y$ if $y$ is on the unique path connecting $x$ to $v_r$. 
	\item For an $h\in H_+(T)$ let $DL(h)$ be the set of all legs that are descendants to $h$, including $h$ itself. Note that $DL(h)\subseteq L_r(T)$ for any $h\in H_+(T)$ and $DL(l)=\{l\}$ for $l\in L_r(T)$. 
	%\item For an $h\in H_+(T)$ let $DH(h)$ be the set of all positive half-edges that are descendants to $h$, \emph{excluding} $h$. For instance, for $l\in L_r(T)$ we have $DH(l) = \emptyset$, and for $h\in H_+(T)\setminus L_r(T)$ we have $DH(h) \supseteq DL(h)$. 
	\item For an $e\in E(T)$ let $DL(e)$ be the set of all legs that are descendants to $e$. Note that $DL(e)\subseteq L_r(T)$ for any $e\in E(T)$. 
	\item For an $v\in V(T)$ let $DL(v)$ be the set of all regular legs that are descendants to $v$. In particular, $DL(v_r) = L_r(T)$. 
	\item For a $v\in V(T)$ let $DV(v)\subset V(T)$ be the subset of all vertices that are descendants of $v$, including $v$ itself. For instance, $DV(v_r) = V(T)$. %Let 
	%\[
	%D\chi(v)\coloneqq \sum_{v'\in DV(v)} \chi(v').
	%\]
\end{itemize}

\subsubsection{Ring of coefficients}
Consider the polynomial ring
\begin{align}
	Q\coloneqq \Q[a_1,\dots,a_n,b_1,\dots,b_m]
\end{align} 
and define $a\colon H_+(T) \to Q$, $a\colon E(T)\to Q$, and $a\colon V(T)\to Q$ (abusing notation we use the same symbol $a$ for all of these maps) by 
\begin{align}
%	a(\sigma_{n+j})& \coloneqq b_j, & j=1,\dots,m; 
%    \\ \notag
%	a(\sigma_i)& \coloneqq a_i, & i=1,\dots,n; 
%    \\ \notag 
	a(h)& \coloneqq \textstyle\sum_{l\in DL(h)} a(l), & h\in H_+(T); 
    \\ \notag 
	a(e)& \coloneqq \textstyle\sum_{l\in DL(e)} a(l),& e\in E(T); 
    \\ \notag 
	a(v)& \coloneqq \textstyle\sum_{l\in DL(v)} a(l), & v\in V(T).
\end{align} 
So, in particular, $a(\sigma_i)= a_i$, $i=1,\dots,n$. The variables $b_i$ are associated to the frozen legs $\sigma_{n+i}$, $i=1,\dots,m$.

\begin{figure}[h!]
\label{fig:SRT}
\begin{center}
    \begin{tikzpicture}[
        level/.style={font=\small, inner sep=1pt, fill=white}
    ]
    
    % Nodes with custom styles
    \node[draw, circle, minimum size=1cm, inner sep=0pt] (g0) at (-0.4,0) {$g_0$};
    \node[draw, circle, minimum width=1cm, minimum height=1cm] (g1) at (2.7,1) {$g_1$};
    \node[draw, circle, minimum width=1cm, minimum height=1cm] (g2) at (5.5,1) {$g_2$};
    \node[draw, circle, minimum size=1cm, inner sep=0pt] (g3) at (3,-1.5) {$g_3$};
    
    % Edges
    \draw (g0) -- (g1);
    \draw (g1) -- (g2);
    \draw (g0) -- (g3);
    
    % Wavy arrows into q0
    \draw[decorate, decoration={snake, segment length=4pt}] (-1.2,0.5) -- (g0);
    \draw[decorate, decoration={snake, segment length=4pt}] (-1.2,-0.5) -- (g0);
    
    % Extra stubs
    % \draw (g0) -- ++(0.8,0.8);
    % \draw (g1) -- ++(0.8,0.8);
    % \draw (g3) -- ++(1,0.5);
    % \draw (g3) -- ++(1,-0.5);
    % \draw (g2) -- ++(1,0);
    % \draw (g2) -- ++(0.8,0.8);
    \draw (g0) -- ++(0.74,0.7);
    \draw (g1) -- ++(0.8,0.7);
    \draw (g3) -- ++(0.8,0.4);
    \draw (g3) -- ++(0.8,-0.4);
    \draw (g2) -- ++(0.8,-0.4);
    \draw (g2) -- ++(0.8,0.4);
    
    % Math expression
    \node at (-1.4,0.5) {\scriptsize ${b_1}$};
    \node at (-1.4,-0.5) {\scriptsize ${b_2}$};
    
    \node at (0.5,0.85) {\scriptsize ${a_1}$};
    \node at (1.25,0.1) {\scriptsize ${a_2 + a_3 + a_6}$};
    \node at (0.5,-0.7) {\scriptsize ${a_4 + a_5}$};
    
    \node at (3.7,1.8) {\scriptsize ${a_2}$};
    \node at (3.8,1.2) {\scriptsize ${a_3 + a_6}$};
    
    \node at (4,-1) {\scriptsize ${a_4}$};
    \node at (4,-1.9) {\scriptsize ${a_5}$};
    
    \node at (6.5,0.6) {\scriptsize ${a_3}$};
    \node at (6.5,1.4) {\scriptsize ${a_6}$};
    
    \end{tikzpicture}    
\end{center}
    \caption{A stable rooted graph with root vertex $v_0$ of genus $g_0$. The vertex $v_2$ of genus $g_2$ is descendant of the vertices $v_1$ of genus $g_1$ and of course of the root $v_0$. Half-edges (both leaves and half-edges of edges) decorated with polynomials in the $a_i$ are positive; those decorated with $b_j$ (frozen legs, represented by wavy lines) and the remaining half-edges of the edges are negative.}
\end{figure}
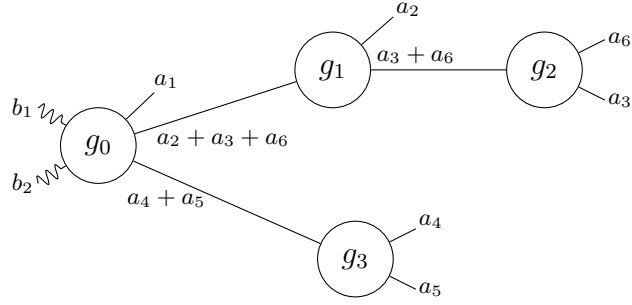

\subsubsection{Levels}
We enhance the structure of a stable rooted tree to the so-called leveled stable rooted tree (of genus $g$, with $n$ regular and $m$ frozen legs). Let $T\in\SRT_{g,n,m}$. A function $\ell\colon V(T)\to\Z_{\geq 0}$ is called a \textit{level function} if the following conditions are satisfied:
\begin{itemize}
	\item The value of $\ell$ on the root vertex is zero ($\ell(v_r) = 0$).
	\item If $v'\in DV(v)$ and $v'\not=v$, then $\ell(v')>\ell(v)$. 
	\item There are no empty levels, that is, for any $0\leq i \leq \max \ell(V(T))$ the set $\ell^{-1}(i)$ is non-empty. 
\end{itemize}
 A leveled stable rooted tree is a pair $(T,\ell)$. The height of a leveled stable rooted tree $(T,\ell)$ is $\ell(T)\coloneqq \max \ell(V(T))$. Let $\mathcal{L}(T)$ denote the set of level functions on $T$. 

Once we fix a level function $\ell\in \mathcal{L}(T)$, we extend the definition of a genus function $g\colon V(T) \to \Z_{\geq 0}$ to the level genus function $g^{\mathrm{lvl}}\colon \{0,\dots,\ell(T)\}\to \Z_{\geq 0}$ defined as 
\begin{align}
	g^{\mathrm{lvl}} (l) \coloneqq \sum_{v\in V(T), \ell(v) \leq l} g(v). 
\end{align}
Of course, 	$g^{\mathrm{lvl}} (\ell(T))=g$. 

\subsubsection{Degree}
We enhance the structure of a stable rooted tree $T\in SRT_{g,n,m}$ with a so-called \textit{degree function} $d\colon V(T) \to \Z_{\geq 0}$. 
Let $\mathcal{D}(T)$ denote the set of degree labels on $T$. 

In the presence of a level function $\ell\in \mathcal{L}(T)$, we associate with a degree function $d$  a \textit{level degree function} $d^{\mathrm{lvl}}\colon \{0,\dots,\ell(T)\}\to \Z_{\geq 0}$, defined by 
\begin{align}
	d^{\mathrm{lvl}} (l) \coloneqq \sum_{v\in V(T), \ell(v) \leq l} d(v) + |\{v\in V(T), \ell(v) \leq l\}|-1. 
\end{align}
Define also 
$
d(T)\coloneqq d^{\mathrm{lvl}} (\ell((T)).
$

\subsubsection{The leveled rooted tree classes}
Let $T\in\SRT_{g,n,m}$. Assign to each $v\in V(T)$ the moduli space of curves $\overline{\mathcal{M}}_{g(v),|H(v)|}$, where the first $|H_{+}(v)|$ marked points correspond to the positive half-edges attached to $v$ and ordered in an arbitrary but fixed way and the the last $|H_{-}(v)|$ marked points correspond to the negative half-edges attached to $v$ (it is just one half edge for a non-root vertex, and for the root vertex we order them following the order of frozen legs $\sigma_{n+1},\dots,\sigma_{n+m}$). Consider the class associated to each vertex:
\begin{equation}
\label{eq:def-O(v)}
O(v)\coloneqq\begin{cases}
O_{g(v),|H(v)|}\left(a(h_1),\dots,a(h_{|H_+(v)|}),b_{1},\dots,b_{m}\right), & v=v_{r};\\
O_{g(v),|H(v)|}\left(a(h_1),\dots,a(h_{|H_+(v)|}),0\right). & v\neq v_{r}.
\end{cases}
\end{equation}

For each $d\geq 0$, define
\begin{align}
 O(v)_{\deg_{R^\bullet} =d}
\end{align} to be the  homogeneous component of degree $d$ of the class $O(v)$ with respect to the Chow grading. By the homogeneity property in Definition~\ref{def:int-obs}, this coincides with the  homogeneous component of degree $d$, viewed as a polynomial in the variables $a_1,\dots,a_n,b_1\dots,b_m$.
We also define 
\begin{align}
O(v)_{\deg_a=d}
\end{align}
to be the homogeneous component of degree $d$ of the class $O(v)$ as a polynomial in  $a_1,\dots,a_n$ exclusively, treating $b_1,\dots,b_m$ as constants.
Then $O(v)_{\deg_a=d} =  O(v)_{\deg_{R^\bullet}=d}$ for $v\not= v_r$, and for the root vertex
\begin{align}
\Coeff{\prod_{i=1}^m b_i^{d_i}} O(v)_{\deg_a=d} = \Coeff{\prod_{i=1}^m b_i^{d_i}} (O(v))_{\deg_{R^\bullet}=d+\sum_{i=1}^m d_i}.
\end{align}
% or, in other words,
% \begin{align}
%     (O(v_r))_{\deg_a=d} & = 
%     \sum_{d_1,\dots,d_m=0}^\infty \prod_{i=1}^m b_i^{d_i} \Coeff{\prod_{i=1}^m b_i^{d_i}} O_{g(v_r),|H(v_r)|}\big(a(h_1),\dots,
%     \\ \notag & \qquad \qquad \qquad \qquad \qquad \qquad a(h_{|H_+(v_r)|}),b_{1},\dots,b_{m}\big)_{\deg_{R^*}=d-\sum_{i=1}^n d_i}.
% \end{align}
%
Only the notion of degree in the variables $a_1,\dots,a_n$ is involved in the following definition. 

% Previously: For each $d\geq 0$ define $(O(v))_{d}$ to be the homogeneous component of degree $d$ of the class $O(v)$ as a polynomial in the variables $a_1,\dots,a_n$. Note that it is not exactly the same as the homogenenous component of $O(v)$ with respect to the Chow grading, since there might be non-trivial dependence on the variables $b_1,\dots,b_m$ and we have the homogeneity condition. That is, if we denote by $O(v)_{\{d\}}$ the component of $O(v)$ in $R^d(\oM_{g(v),|H(v)|})$, then $(O(v))_{d} =  (O(v))_{\{d\}}$ for $v\not= v_r$, and for the root vertex 
% \begin{align}
% \Coeff{\prod_{i=1}^m b_i^{d_i}} (O(v))_{d} = \Coeff{\prod_{i=1}^m b_i^{d_i}} (O(v))_{\{d-\sum_{i=1}^m d_i\}}
% \end{align}
% or, in other words,
% \begin{align}
%     (O(v_r))_{d} & = 
%     \sum_{d_1,\dots,d_m=0}^\infty \prod_{i=1}^m b_i^{d_i} \Coeff{\prod_{i=1}^m b_i^{d_i}} O_{g(v_r),|H(v_r)|}\big(a(h_1),\dots,
%     \\ \notag & \qquad \qquad \qquad \qquad \qquad \qquad \qquad \qquad a(h_{|H_+(v_r)|}),b_{1},\dots,b_{m}\big)_{\{d-\sum_{i=1}^n d_i\}}.
% \end{align}

\begin{definition}
For each $(g,n,m)$ such that $2g-2+n+m>0$ define the class
\begin{equation}
{}^{O}\!B_{g,n}^{m}\in R^{*}\left(\overline{\mathcal{M}}_{g,n+m}\right)\otimes_{\mathbb{Q}}Q
\end{equation}
as
\begin{align}
{}^{O}\!B_{g,n}^{m}\coloneqq\sum_{\substack{T\in\SRT(g,n,m),d\in\mathcal{D}(T),\ell\in\mathcal{L}(t)\\
\forall i<\ell(T)\colon d^{\mathrm{lvl}}(i)\leq2g^{\mathrm{lvl}}(i)-2+m
}
}(-1)^{\ell(T)}\biggl(\prod_{e\in E(T)}a(e)\biggr)(\mathrm{gl}_{T})_{*}\bigotimes_{v\in V(T)}O(v)_{\deg_a=d(v)}.
\end{align}
Here 
\begin{align}
    (\mathrm{gl}_{T})_{*} \colon \bigotimes_{v\in V(T)}R^{*}(\overline{\mathcal{M}}_{g(v),|H(v)|})\otimes_{\mathbb{Q}}Q \to R^{*}(\overline{\mathcal{M}}_{g,n+m})\otimes_{\mathbb{Q}}Q
\end{align} is the boundary pushforward map. 
\end{definition}
 
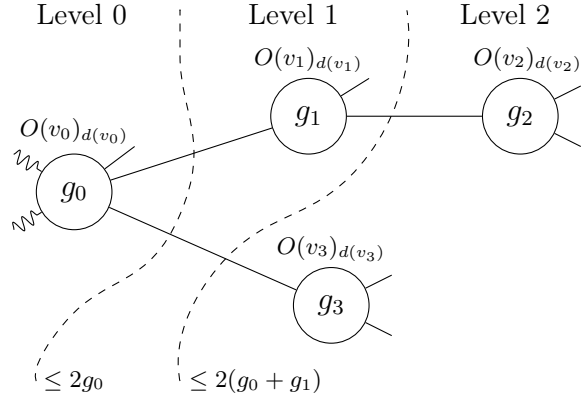
\begin{figure}[h!] \label{fig:SRT:levels}

\begin{center}
    \begin{tikzpicture}[
        level/.style={font=\small, inner sep=1pt, fill=white}
    ]
    
    % Nodes with custom styles
    \node[draw, circle, minimum size=1cm, inner sep=0pt] (g0) at (-0.4,0) {$g_0$};
    \node[draw, circle, minimum width=1cm, minimum height=1cm] (g1) at (2.7,1) {$g_1$};
    \node[draw, circle, minimum width=1cm, minimum height=1cm] (g2) at (5.5,1) {$g_2$};
    \node[draw, circle, minimum size=1cm, inner sep=0pt] (g3) at (3,-1.5) {$g_3$};
    
    % Edges
    \draw (g0) -- (g1);
    \draw (g1) -- (g2);
    \draw (g0) -- (g3);
    
    % Wavy arrows into q0
    \draw[decorate, decoration={snake, segment length=4pt}] (-1.2,0.5) -- (g0);
    \draw[decorate, decoration={snake, segment length=4pt}] (-1.2,-0.5) -- (g0);
    
    % Extra stubs
    \draw (g0) -- ++(0.8,0.6);
    \draw (g1) -- ++(0.8,0.5);
    \draw (g3) -- ++(0.8,0.4);
    \draw (g3) -- ++(0.8,-0.4);
    \draw (g2) -- ++(0.8,-0.4);
    \draw (g2) -- ++(0.8,0.4);
    
    % Dashed curves
    \draw[dashed] plot [smooth, tension=1] coordinates {(-0.9,-2.5) (-0.5,-1.7) (1,-0.3) (1,1.8) (1,2.4)};
    \draw[dashed] plot [smooth, tension=1] coordinates {(1,-2.5) (1.5,-1) (3.5,0.5) (4,2.4)};
    
    % Level labels
    \node[level] at (-0.3,2.35) {Level 0};
    \node[level] at (2.5,2.35) {Level 1};
    \node[level] at (5.3,2.35) {Level 2};
    
    % Math expression
    \node at (-0.4,-2.5) {\scriptsize $\leq 2g_0$};
    \node at (2,-2.5) {\scriptsize $\leq 2(g_0 + g_1)$};
    \node at (-0.4,0.8) {\scriptsize $O(v_0)_{d(v_0)}$};
    %\node at (0.8,0.3) {\scriptsize $O(v_0)_{d_6}$};
    %\node at (0.8,-0.3) {\scriptsize $O(v_0)_{d_7}$};
    
    \node at (2.7,1.75) {\scriptsize $O(v_1)_{d(v_1)}$};
    %\node at (3.8,1.2) {\scriptsize $O(v_1)_{d_8}$};
    
    \node at (3,-0.75) {\scriptsize $O(v_3)_{d(v_3)}$};
    %\node at (4.2,-1.7) {\scriptsize $O(v_3)_{d_5}$};
    
    \node at (5.6,1.75) {\scriptsize $O(v_2)_{d(v_2)}$};
    
    \end{tikzpicture}    
\end{center}
    \caption{A leveled stable rooted tree with a degree function and decorated with cohomological classes $O(v_i)$ involved in the definition of ${}^{O}\!B_{g,n=6}^{m=2}$, where $g=g_0+g_1+g_2+g_3$. The level function is represented with the dashed
lines: the root is assigned level 0, the vertex 1 lies at level 1, and the vertices 2 and 3 are
at level 2. The degrees $d(v_i)$ are degrees of the homogeneous polynomials in the $a_1,\dots,a_n$ extracted from each of the $O(v_i)$ for every vertex $v_i$; they must satisfy $d(v_0) \leq 2g_0$ and $d(v_0) +d(v_1) + 1 \leq 2(g_0+g_1)$.} 
\end{figure}

\begin{remark}
	\label{rem:homogeneity} The class ${}^{O}\!B_{g,n}^{m}$ satisfies the following homogeneity property 
\begin{align}
    \big({}^{O}\!B_{g,n}^{m}\big)_{\deg_{R^\bullet}=d} \in R^{d}\left(\overline{\mathcal{M}}_{g,n+m}\right)\otimes_\Q \Q[a_1,\dots,a_n,b_1,\dots,b_m]
\end{align}
is a homogeneous polynomial of  degree $d$ in the variables $a_1,\dots,a_n,b_1,\dots,b_m$. This follows from the homogeneity condition in Definition~\ref{def:int-obs}.
\end{remark}

\subsubsection{The \texorpdfstring{$\protect\bA$}{AA} and \texorpdfstring{$\protect\bA^1$}{AA1} classes}\label{Sec: def A and A1}
Let $a_1, \dots, a_n$ be positive integers. Recall that we use the notation $\suma \coloneqq \sum_{i=1}^n a_i$. For each $g\geq 0$, $n\geq 1$ consider the moduli space 
\begin{align}
	\overline{\mathcal{M}}_{g}^{\sim}(\mathbb{P}^1,a_1,\dots,a_n,-\suma), \qquad \suma \coloneqq \sum_{i=1}^n a_i
\end{align}
of rubber stable maps to $(\mathbb{P}^1,0,\infty)$, where the profile over $0$ is $(a_1, \dots, a_n)$, and the profile over $\infty$ is $(-\suma)$.
Its projection onto the source curve is denoted by 
\begin{align}
	s\colon \overline{\mathcal{M}}_{g}^{\sim}(\mathbb{P}^1,a_1,\dots,a_n,-\suma) \to \oM_{g,n+1}.
\end{align}
Consider also its projection on the target curve
\begin{equation}
t\colon \oM_{g}^{\sim}(\mathbb{P}^1,a_1,\dots,a_n,-\suma)\to LM_{2g-1+n},
\end{equation}
where $LM_{m}$ denotes the Losev-Manin space with $m$ marked points.
Let $t^*\psi_0$ be the pull-back by $t$ of the $\psi$-class at the point $0$ in the cohomology (or Chow ring) of the Losev-Manin space. 

Define 

\begin{equation} \label{eq: A-1}
\bA_{g,n}^{1}(a_{1},\dots,a_{n})\coloneqq\lambda_{g}s_{*}\left(\frac{1}{1-t^{*}\psi_{0}}\left[\overline{\mathcal{M}}_{g}^{\sim}\left(\mathbb{P}^{1},a_{1},\dots,a_{n},-\suma\right)\right]^{\mathrm{vir}}\right),
\end{equation}
where $\lambda_g$ is the top Chern class of the Hodge bundle on $\oM_{g,n+1}$. 

Define also 
\begin{equation}
\bA_{g,n}\coloneqq \frac{1}{\suma}\pi_{*}\bA_{g,n}^{1}, 
%\quad
%\bA^m_{g,n}\coloneqq 0 \; \text{ for } m \geq 2,
\end{equation}
where $\pi\colon \oM_{g,n+1} \to \oM_{g,n}$ is the morphism forgetting the last marked point. 

\begin{remark} The classes $\bA_{g,n} \in R^* (\oM_{g,n})\otimes_\Q \Q[a_1,\dots,a_n]$, $\bA^1_{g,n}\in R^* (\oM_{g,n+1})\otimes_\Q \Q[a_1,\dots,a_n]$. Moreover, both classes possess the homogeneity property as in Def.~\ref{def:int-obs}. These properties are proved in Sec.~\ref{subsec: polynomiality genus 0 A}.
\end{remark}

\subsubsection{The leveled rooted tree relations} 

\begin{definition}[LRT-$1$ and LRT-$2$ relations] \label{def: stable rooted tree relations} 
We define the following relations.
\begin{itemize}
    \item We say that the polynomial-valued classes $\{O_{g,n}(a_1,\dots,a_n)\}$ satisfy \textit{the leveled rooted tree relations of type $1$ or LRT-$1$}, if for any $g\geq 0$, $n\geq 1$, $2g-1+n>0$, 
\begin{align} \label{eq: type $1$ relation}
    \deg_a \left(\Coeff{b_1^0}{}^{O}\!B_{g,n}^{1}-\bA_{g,n}^{1}\right)\leq2g-1.
\end{align}
    \item They satisfy the \textit{leveled rooted tree relations of type $2$ or LRT-$2$}, if for any $g\geq 0$, $n\geq 0$, $2g+n>0$, 
\begin{align}
    \deg_a {}^{O}\!B_{g,n}^{2}\leq 2g.
\end{align}
\end{itemize}

\end{definition}
This is the fourth property that one has to require in Definition~\ref{def:int-obs}.

\begin{remark}
    More generally, Theorems~\ref{thm: integrable observables vector potential} and~\ref{thm: integrable observables tau function} remain valid if one replaces the leveled rooted tree relation of type $2$ by the following weaker condition: for any $g\geq 0$, $n\geq 0$, $2g+n>0$, 
\begin{align}
    \deg_a \Coeff{b_1^pb_2^0}{}^{O}\!B_{g,n}^{2}\leq 2g
\end{align}
for any $p\geq 0$.
\end{remark}

\begin{remark} \label{rem: full system of leveled rooted tree relations}
It is also geometrically very natural to consider a full version of the leveled rooted tree relations. For $g,n,m\geq0$, such that $2g-2+n+m>0$, the \textit{full system of leveled rooted tree relations} reads:

\begin{align}
    \text{LRT-m:}\qquad & \deg_a {}^{O}\!B_{g,n}^{m} & & \leq2g-2+m, \qquad  m\geq 2; \\
\text{LRT-1:}\qquad & \deg_a \left(\Coeff{b_1^0}{}^{O}\!B_{g,n}^{1}-\bA_{g,n}^{1}\right) & & \leq 2g-1; 
\\ 
\text{LRT-0:}\qquad &\deg_a \left({}^{O}\!B_{g,n}^{0} - \bA_{g,n} \right)& & \leq 2g-2.
\end{align} 
%It is natural to call the last relation the leveled rooted tree relations of type zero (LRT-0). 
However, among these relations only the type $1$ and type $2$ relations have immediate connection to integrability, which we discuss in detail in Sec.~\ref{sec:integrable}. We will show in Lemma~\ref{lemma: Type 0 relation} that LRT-$1$ implies LRT-0.
The leveled rooted tree relations for $m>2$, whose consequence are discussed Sec.~\ref{sec: application mLRT relations}, do not have direct connection with integrability.
\end{remark}

\begin{remark} There is also a different type of relations that we call \emph{master relations} and introduce below. Their special cases can serve as an alternative for the leveled rooted tree relations in the definition of integrable observables. We discuss this alternative in Sec.~\ref{sec:srt-master}.
\end{remark}

\subsection{Old and new 
examples of integrable observables}
In this section, we present three examples of integrable observables. 
The first one is the classical Witten's $\psi$-classes arranged into generating $n$-point functions. This example is so far the only one fully studied in the literature. The second example is the $\bA$-class that leads to the DR hierarchy but in an unconventional way, and thus needs a revision and proofs of its key properties. The third example is new and constitutes the main geometric result of this paper. 

The original Witten approach with $\psi$ classes is reproduced in the following way. On each moduli space $\oM_{g,n}$, $g\geq 0$, $n\geq 1$, $2g-2+n>0$, we consider the following classes arranged into generating polynomials: 
\begin{align}
	\bPsi_{g,n}(a_1,\dots,a_n) \coloneqq %\frac{1}
    {\prod_{i=1}^n \frac{1}{1-a_i\psi_i}}.
\end{align}

\begin{theorem} \label{thm: Psi} The classes $\{\bPsi_{g,n}\}$ are integrable observables of type $1$ and of type $2$.
\end{theorem}

The corresponding integrable systems associated to F-CohFTs and P-CohFTs by Theorems~\ref{thm: integrable observables vector potential} and~\ref{thm: integrable observables tau function}
are called the \textit{Dubrovin-Zhang hierarchies}.

\smallskip

The DR integrable systems, used as canonical reference points in Theorems~\ref{thm: integrable observables vector potential} and~\ref{thm: integrable observables tau function} can themselves be constructed with the classes $\{\bA_{g,n}\}$ used as integrable observables:

\begin{theorem} \label{thm: A} The classes $\{\bA_{g,n}\}$ are integrable observables of type $1$.
\end{theorem}

The DR hierarchies of conservations laws is defined for F-CohFTs in~\cite{BR-spin}, and in Hamiltonian form for P-CohFTs in~\cite{Bur,BR-spin} that can also be considered in the normal coordinates~\cite{BDGR1}. It follows from~\cite{BDGR1,BDGR20,BS22} in combination with the further observations and simplifications in~\cite{BLS-DRDZ,BLS-Quantum} that using $\{\bA_{g,n}\}$ as integrable observables  with an F-CohFT (resp., P-CohFT) yields the DR hierarchy in ``normal'' coordinates. We discuss this in details in Sec.~\ref{sec:integrable}. 

\smallskip

% \begin{remark} Results of~\cite{BDGR1,BDGR20,BS22} in combination with the further observations and simplifications in~\cite{BLS-DRDZ,BLS-Quantum} imply that using these integrable observables with an F-CohFT (resp., P-CohFT), one obtains that DR hierarchy in natural (resp., normal) coordinates. We discuss this further in Sec.~\ref{sec:integrable}.
% \end{remark}

Let us also present a genuinely new example of integrable observables. More precisely, the classes themselves are not new, but their interpretation and usage as integrable observables is a new result. 
To this end, we recall the definition of the so-called $\Omega$-class. Let $g,n$ be two nonnegative integers such that $2g-2+n>0$. Let $r$ and $s$ be integers such that $r$ is positive, and let $a_1, \ldots, a_n$ be integers satisfying the modular constraint
\begin{equation} \label{eq: modular constraint}
	\suma \equiv (2g-2+n)s \pmod{r}
\end{equation}
(recall that $\suma \coloneqq \sum_{i=1}^n a_i$).
We denote by  $\overline{\mathcal{M}}_{g}^{r,s}(a_1, \ldots, a_n)$ the moduli stack of $r$-spin structures parametrizing $r$-th roots $L$ of the line bundle
\begin{equation}
	L^{\otimes r} \cong \omega_{C, \log}^{\otimes s}\biggl(-\sum_{i=1}^n a_i p_i \biggr),
\end{equation}
where $\omega_{C, \log} = \omega_C(\sum_i p_i)$ is the log-canonical bundle of the stable curve $C$. The modular condition~\eqref{eq: modular constraint} is a necessary, but also sufficient, condition for the existence of a $r$th root. Let $\pi \colon \overline{\mathcal{C}}_{g}^{r,s}(a_1, \ldots, a_n) \to \overline{\mathcal{M}}_{g}^{r,s}(a_1, \ldots, a_n)$ be the universal curve, and $\mathcal{L} \to \overline{\mathcal C}_{g}^{r,s}(a_1, \ldots, a_n)$ the universal $r$-th root. 
There is moreover a natural forgetful morphism
\begin{equation}
    \label{epsilon map}
	\epsilon \colon
	\overline{\mathcal{M}}^{r,s}_{g}(a_1, \ldots, a_n)
	\longrightarrow
	\overline{\mathcal{M}}_{g,n}
\end{equation}
which forgets the line bundle. The $\Omega$-classes are defined as
\begin{equation}\label{eqn:Omega}
	\Omega_{g,n}^{[x]}(r,s;a_1,\dots,a_n)
	=
	\epsilon_{\ast}
	\left(\sum_{i\geq 0} x^i c_i\left( -R^*\pi_* \mathcal L \right) 
	\right)
	\in
	R^{*}(\overline{\mathcal{M}}_{g,n}).
\end{equation}
The $\bOm$-classes are defined as a very specific parametrisation of the $\Omega$-classes, precisely as
\begin{align}
	\bOm_{g,n}(a_1,\dots,a_n) \coloneqq 
	%\lambda_g \hej_{g,n}(a_1,\dots,a_n) = 
	\suma^{1-g} \lambda_g \Omega^{[\suma]}_{g,n}(\suma,0;-a_1,\dots,-a_n).
\end{align}

\begin{remark}
Due to its importance, in particular for the computations with the quantized integrable systems, the class $\suma^{1-g} \Omega^{[\suma]}_{g,n}(\suma,0;-a_1,\dots,-a_n)$ got a special notation $\hej_{g,n}(a_1,\dots,a_n)$ in~\cite{BLS-Omega,BLS-Quantum}. With respect to this notation $\bOm_{g,n} = \lambda_g\hej_{g,n}$, cf.~\cite[Def.~3.1]{BLS-Omega}.
%, where $\hej$ is the class mentioned in Section~\ref{sub:sec:hej}.
\end{remark}

\begin{remark} This way the classes $\bOm_{g,n}(a_1,\dots,a_n)$ are only defined for $a_1,\dots,a_n\in \Z$. One can show, however, that  $\bOm_{g,n}(a_1,\dots,a_n)$ for integer $a_1,\dots,a_n$ are the values of a polynomial in $n$ variables, and abusing notation we denote it by $\bOm_{g,n}(a_1,\dots,a_n)$ as well. Thus, $\bOm_{g,n}(a_1,\dots,a_n) \in R^* (\oM_{g,n})\otimes_\Q \Q[a_1,\dots,a_n]$. Note also that
these classes possess the same homogeneity property as in Definition~\ref{def:int-obs}, see Sec.~\ref{subsec: Poly Homo g0 Hejgn}.
%.  and Remark~\ref{rem:homogeneity}.
\end{remark}

\begin{remark} By properties of the $\Omega$-class (see \cite{BLS-Omega, GLN23}) we have:
\begin{equation}
\label{eq:Omega-negative-ai-psi}
\Omega^{[\suma]}_{g,n}(\suma,0;-a_1,\dots,-a_n) = \frac{\Omega^{[\suma]}_{g,n}(\suma,0;\suma-a_1,\dots,\suma-a_n)}{\prod_{i=1}^n (1-a_i\psi_i)}
\end{equation}
so that 
\begin{equation}\label{eq:hey:omega}
\bOm_{g,n}(a_1,\dots,a_n) = \suma^{1-g} \lambda_g \Omega^{[\suma]}_{g,n}(\suma, 0; \suma - a_1,\dots,\suma - a_n) \bPsi_{g,n}(a_1,\dots,a_n).
\end{equation}    
\end{remark}

\begin{theorem} \label{thm: Hej} The classes $\{\bOm_{g,n}\}$ are integrable observables of type $1$ and of type $2$.
\end{theorem}

\begin{definition}[$\bOm$-hierarchy] The system of conservation laws associated to an F-CohFT (resp. the Hamiltonian tau-symmetric integrable system associated to a P-CohFT) associated to the integrable observables $\{\bOm_{g,n}\}$ is called the \emph{$\bOm$-hierarchy}.
\end{definition}

Theorems~\ref{thm: integrable observables vector potential} and~\ref{thm: integrable observables tau function} imply that these three hierarchies are equivalent via Miura transformations. Further relations between these three hierarchies are discussed in Sec.~\ref{sec: interrelations}.

We summarize the integrable observable properties of the three instances in the table below.

{\tiny
\begin{table}[H]
\centering
\begin{tabular}{|p{3.1cm}|p{3.1cm}|p{3.1cm}|p{3.1cm}|p{3.1cm}|}
\hline
Integrable    Observables & $\bPsi$ & $\bOm$ & $\bA$ \\ \hline\hline

String equation  
& %\begin{minipage}[t]{2.8cm}\centering
\cite[Eq.~(2.36)]{Witten-KdV}
%\end{minipage}
& %\begin{minipage}[c]{3.0cm}\centering
\cite[Thm 4.1(v)]{GLN23} or Sec.~\ref{subsubsec: push-forward HHej}
%\end{minipage}
& %\begin{minipage}[c]{3cm}\centering
\cite[Prop.~4.2]{BGR19} w/
\cite[Lemma 2.6]{BLS-DRDZ}, new proof in Sec.~\ref{subsubsec: push-forward A}
%\end{minipage}
\\ \hline

Dilaton equation
& %\begin{minipage}[t]{2.8cm}\centering
\cite[Eq.~(2.44)]{Witten-KdV}
%\end{minipage}
& %\begin{minipage}[c]{3.0cm}\centering
Sec.~\ref{subsubsec: push-forward HHej}
%\end{minipage}
& %begin{minipage}[c]{3cm}\centering
\cite[Prop.~4.5]{BGR19} w/ \cite[Lemma 2.6]{BLS-DRDZ}, new proof in Sec.~\ref{subsubsec: push-forward A}
%\end{minipage}
\\ \hline

Genus zero 
& %\begin{minipage}[c]{3.1cm}\centering
by definition
%\end{minipage}
& %\begin{minipage}[c]{3.1cm}\centering
\cite[Lemma 3.8]{BLRS} w/ Eq.~(\ref{eq:Omega-negative-ai-psi})
%\end{minipage}
& %\begin{minipage}[c]{3.1cm}\centering
%\cite[Thm.~2.3]{BS22},\\
%\cite[Thm 2.8]{BLS-DRDZ},\\
%\cite[Sec.~2.2]{BSS25}
\cite[Prop.~4.4]{BGR19} w/
\cite[Lemma 2.6]{BLS-DRDZ}
%\end{minipage}
\\ \hline

%\begin{minipage}[c]{2.8cm}\centering
Homogeneity
%\end{minipage} 
& %\begin{minipage}[c]{2.8cm}\centering
by definition
%\end{minipage}
& %\begin{minipage}[c]{3cm}\centering
\cite[Prop.~3.4]{BLRS}
%\end{minipage}
& \cite[Prop.~2.3]{BLS-Quantum} 
\\ \hline

%\begin{minipage}[c]{3cm}\centering
LRT-$1$
%\end{minipage}
& %\begin{minipage}[c]{2.8cm}\centering
\cite[Thm~2.1]{BSS25} w/ \cite[Thm.~3]{BLS-DRDZ}
%\end{minipage}
& %\begin{minipage}[c]{3.0cm}\centering
\cite[Thm.~3.4]{BLS-Omega}
%\end{minipage}
& Sec.~\ref{subsubsec:master-A} 
\\ \hline

%\begin{minipage}[c]{3.1cm}\centering
LRT-$2$
%\end{minipage}
& %\begin{minipage}[c]{3.1cm}\centering
\cite[Thm~2.1]{BSS25} w/ \cite[Thm.~3]{BLS-DRDZ}
%\end{minipage}
& 
%\begin{minipage}[c]{3.0cm}\centering
Prop.~\ref{prop: geometric master} $\&$ Sec.~\ref{sec:localization}
%\end{minipage}
& $\qquad\quad\quad$ \textemdash 
\\ \hline
\end{tabular}
\end{table}
}

Note that in many cases there are different possible proofs available in the literature, and we discuss below the possible alternative routes and in some cases give new proofs for the required statements. 

\subsection{Organization of the paper}  In Sec.~\ref{sec:srt-master} we discuss the system of leveled rooted tree relations and their relatives called master relations, which gives concrete tools to prove them. In Sec.~\ref{sec:integrable} we survey how integrable observables are used for constructions of integrable systems, formulate precise versions of Theorems~\ref{thm: integrable observables vector potential} and~\ref{thm: integrable observables tau function}, and discuss relations between these systems. In Sec.~\ref{sec:proofs} we give proofs (or, in some cases, guides to the literature once the proofs are already available) of all integrable observables properties for the three examples that we have, outsourcing the localization argument that we need for the leveled tree relations for $\{\bOm_{g,n}\}$ to Sec.~\ref{sec:localization}. Finally, in Sec.~\ref{sec:examples} we present some computations given as examples. 

\subsection{Acknowledgments} X.~B.~and S.~S.~are supported by the Dutch Research Council, grant no.~OCENW.M.21.233. D.~L.~is supported by the University of Trieste and by the INdAM group
GNSAGA.

%%%%%%%%%%%%%%%%%%%%%%%%%%%%%%%%%%%%%%%%%%%%%%%%%%%%%%%%%%%%%%%%%%%
\vspace{1cm}
\section{Leveled rooted tree relations and master relations}
\label{sec:srt-master}
%%%%%%%%%%%%%%%%%%%%%%%%%%%%%%%%%%%%%%%%%%%%%%%%%%%%%%%%%%%%%%%%%%%

\subsection{Full system of leveled rooted tree relations}

As we mentioned in Remark~\ref{rem: full system of leveled rooted tree relations}, it is natural to consider the  leveled rooted tree relations of type~$m$ (LRT-$m$), which we recall here for convenience:
\begin{align}
    \text{LRT-m:}\qquad & \deg_a {}^{O}\!B_{g,n}^{m} & & \leq2g-2+m, \qquad  m\geq 2; \\
\text{LRT-1:}\qquad & \deg_a \left(\Coeff{b_1^0}{}^{O}\!B_{g,n}^{1}-\bA_{g,n}^{1}\right) & & \leq 2g-1; 
\\ 
\text{LRT-0:}\qquad &\deg_a \left({}^{O}\!B_{g,n}^{0} - \bA_{g,n} \right)& & \leq 2g-2.
\end{align} 

The type $1$ relation has the following direct corollary:
\begin{lemma}[LRT-$1$ implies LRT-0] \label{lemma: Type 0 relation} If $\{O_{g,n}\}$ satisfy the leveled rooted tree relations of type $1$, then for any $g\geq 0$, $n\geq 1$, $2g+n>0$, we have
\begin{align} \label{eq: LRT relation 0}
    \deg_a \left({}^{O}\!B_{g,n}^{0}-\bA_{g,n}\right)\leq
    2g-2.
\end{align}
\end{lemma}
\begin{proof}
    Apply the push-forward that forgets the last marked point in~\eqref{eq: type $1$ relation} and divide by $\suma$ (the latter factor emerges in the push-forward of $\Coeff{b_1^0} {}^{O}\!B_{g,n}^{1}$ by the string equation for $\{O_{g,n}\}$, and this matches the definition of $\bA_{g,n}$).
\end{proof}

% As we mentioned in Remark~\ref{rem: full system of leveled rooted tree relations}, for $m\geq 2$ it is natural to consider the $m$-th leveled rooted tree relations (SRT-$m$) defined as 
% \begin{align}
%     \deg_a {}^{O}\!B_{g,n}^{m} & \leq2g-2+m.
% \end{align}
% So, in particular, for $m=2$ the $2$-nd leveled rooted tree relations are more general than the type $2$ relations, since they hold for any $\Coeff{b_1^p b_2^q}$, while for type $2$ we restrict ourselves for $q=0$. Finally, by $1$-st leveled rooted tree relations we call exactly the type $1$ relations~\eqref{eq: type $1$ relation}.

% Denote the full set of $m$-th leveled tree relations by $m\text{-L}$, $m\geq 0$. Let $m\text{-L}_0\coloneqq \Coeff{\prod_{i=1}^m b_i^0} m\text{-L}$, $m\geq 2$.

Denote by $\text{LRT}_0\text{-}m\coloneqq \Coeff{\prod_{i=1}^m b_i^0} \text{LRT-}m$, $m\geq 2$.
Then Lemma~\ref{lemma: Type 0 relation} has the following extension:

\begin{lemma} \label{lem:push-forward} $\pi_* (\text{LRT}_0\text{-}m) = \text{LRT}_0\text{-}(m-1)$ for $m\geq 3$ and $\pi_* (\text{LRT}_0\text{-}2) = 0$. 
\end{lemma}

This lemma is an easy observation proved in~\cite{BLRS} for classes $\{\bOm_{g,n}\}$. Since the proof only uses the combinatorics of leveled graphs and the string equation, it can be universally stated for any $\{O_{g,n}\}$ satisfying the string equation.

\begin{remark} \label{rem:LiuWang} Interestingly enough, we see that the statement of Lemma~\ref{lem:push-forward} allows to derive $\text{LRT}_0\text{-}m$ from $\text{LRT}_0\text{-}m'$ for $m'>m$. However, the way these relations are applied to integrable systems (see Sec.~\ref{sec:integrable}) rather suggests that there should be a way to derive $\text{LRT}\text{-}m$ from $\text{LRT}\text{-}2$ for any $m\geq 2$. This is totally not obvious from the structure of these relations, however, it was proved in the context of classes $\{\bPsi_{g,n}\}$ in~\cite{LiuWang}.
\end{remark}

\subsection{The master relations}
For any $g\geq 0, n\geq 1$ such that $2g-2+n \geq 0$ and for any $a_1,\dots, a_{n}$, let 
\begin{align}
	\bD_{g,n+1}(a_1,\dots,a_{n})& \coloneqq -\frac{\lambda_{g}\DR_{g}\big(a_1,\dots,a_{n},-\suma\big)}{(1+\suma\psi_{n+1})} 
\end{align} 
where $\suma \coloneqq\sum_{i=1}^{n} a_i$.
If $T\in \SRT_{g,n,m}$ and $v\in V_{nr}(T)$, we introduce
\begin{align}
	\bD(v)& \coloneqq \bD_{g(v),|H(v)|}(a(h_1),\dots,a(h_{|H_+(v)|}))
	\in R^*(\oM_{g(v),|H(v)|})\otimes_{\Q}Q. 
\end{align}
For $m\geq1$, consider the following class in $R^*(\oM_{g,n+m})\otimes_{\Q}Q$:
\begin{align} \label{eq:masterrel}
& {}^O\Xi_{g,n}^{m}(a_{1},\dots,a_{n},b_{1},\dots,b_{m}) \coloneqq 
\\ \notag & \qquad 
\delta_{m,1}\mathbb{D}_{g,n+1}\left(a_{1},\dots,a_{n}\right)+O_{g,n+m}\left(a_{1},\dots,a_{n},b_{1},\dots,b_{m}\right) 
\\ \notag & \qquad 
  +\sum_{\substack{T\in\SRT(g,n,m)\\ 
\ell\in\mathcal{L}(t),\,\ell(T)=1
}
}\biggl(\prod_{e\in E(T)}a(e)\biggr)(b_{T})_{*}\bigg(O(v_{r})\otimes\bigotimes_{v\in V_{nr}(T)}\bD(v)\bigg),
\end{align}
recalling that $O(v_r)$ is defined in Eq.~\eqref{eq:def-O(v)}.
Note that the Chow degree $d$ of ${}^O\Xi_{g,n}^{m}$ is a homogeneous polynomial of degree $d$ in the variables $a_{1},\dots,a_{n},b_{1},\dots,b_{m}$.

\begin{definition}[Master relations] \label{def: master relation} For $m\geq 2$, we say that the classes $\{O_{g,n}\}$ satisfy the \textit{$m$-th master relation}, or M-$m$, if
\begin{align}
\deg_{a} {}^O\Xi_{g,n}^{m} \leq 2g-2+m. 
\end{align}
For $m=1$, we say that the classes $\{O_{g,n}\}$ satisfy the \textit{$1$-st master relation}, or M-1, if   
\begin{align} \label{eq: m=1 master}
\deg_a \Big( \Coeff{b_1^0} {}^O\Xi_{g,n}^{1}(a_1,\dots,a_n,b_1) 
%- \bA^1_{g,n}(a_1,\dots,a_n) 
\Big) \leq 2g-1.
\end{align}
Note that in this case $\deg_{R^\bullet} = \deg_a$ by homogeneity.
\end{definition}

\subsection{Connections between relations} 
% Denote the full set of $m$-th master relations by $\text{M-}m$, $m\geq 1$. Denote the full set of $m$-th leveled tree relations by $\text{LRT-}m$, $m\geq 0$. 
We have the following system of logical implications between the sets of relations:

\begin{lemma} \label{lem: LRS - master equivalence} $\text{M-}1 \Leftrightarrow \text{LRT-}1$. Moreover, $\text{M-}1\ \&\ \text{M-}m\Rightarrow \text{LRT-}m$. 
\end{lemma}

The statements of this Lemma are proved in~\cite{BLS-DRDZ,BLS-Omega} in application to $\{\bPsi_{g,n}\}$ and $\{\bOm_{g,n}\}$ classes, but the proofs use only combinatorics of graphs, properties of double ramification cycles / rubber relative stable maps to the projective line (to treat the $\mathbb{D}_{g,n+1}\left(a_{1},\dots,a_{n}\right)$ class in the master relations), and combinatorics of graphs. 

This suggests some alternatives for the fourth condition in the definition of integrable observables (Def.~\ref{def:int-obs}). Namely, for type $1$ integrability to can alternatively demand $\text{M-}1$ instead of $\text{LRT-}1$. For type $2$ integrability we can demand $\text{M-}1\ \&\ \text{M-}2$. The master relations are easier to prove from the scratch, and indeed, for $\{\bA_{g,n}\}$ we prove below $\text{M-}1$ rather than $\text{LRT-}1$ thus establishing the type $1$ integrability.

The natural statement that can be uniformly proved for $\{\bOm_{g,n}\}$ and $\{\bPsi_{g,n}\}$ is yet a bit different, and we formulate it in the next subsection.

\subsection{The geometric master relation} For $m\geq1$, consider the following class in $R^*(\oM_{g,n+m})\otimes_{\Q}Q$:
\begin{align}
& {}^O\Upsilon_{g,n}^{m}(a_{1},\dots,a_{n},b_{1},\dots,b_{m})  \coloneqq \\
&\delta_{m,1}\mathbb{D}_{g,n+1}\left(a_{1},\dots,a_{n}\right)+O_{g,n+m}\left(a_{1},\dots,a_{n},b_{1},\dots,b_{m}\right)\prod_{i=1}^{m}\left(1-b_{i}\psi_{i+n}\right)\nonumber \\ \notag 
 & +\sum_{\substack{T\in\SRT(g,n,m)\\
\ell\in\mathcal{L}(t),\,\ell(T)=1
}
}\biggl(\prod_{e\in E(T)}a(e)\biggr)(b_{T})_{*}\bigg(O(v_{r})\otimes\bigotimes_{v\in V_{nr}(T)}\bD(v)\bigg)\prod_{i=1}^{m}\left(1-b_{i}\psi_{i+n}\right).
\end{align}
That is, compared to ${}^O\Xi_{g,n}^{m}$ we just multiplied the two non-exceptional summands in~\eqref{eq:masterrel} by $\prod_{i=1}^{m}(1-b_{i}\psi_{i+n})$. In particular, in the case of $\{O_{g,n}=\bPsi_{g,n}\}$, we have
\begin{align}
    {}^\bPsi\Upsilon_{g,n}^{m}(a_{1},\dots,a_{n},b_{1},\dots,b_{m}) 
\end{align}
does not depend on $b_1,\dots,b_m$. 

\begin{definition} Let $m\geq 1$. For a collection of polynomial valued classes $\{O_{g,n}(a_1,\dots,a_n)\}$ satisfying the homogeneity assumption, we say that is satisfies the $m$-th \emph{geometric master relations}, or GM-$m$, if 
\begin{align}
    \deg_{R^\bullet} {}^O\Upsilon_{g,n}^{m}(a_{1},\dots,a_{n},b_{1},\dots,b_{m}) \leq 2g-2+m.
\end{align} 
It satisfies geometric master relations if GM-$m$ holds for any $m \geq 1$.
\end{definition}

This new definition is justified by the following statements and remarks. First, the clear advantage of geometric master relations is that we use the geometrically defined degree $\deg_{R^\bullet}$ in their formulation, and this does help to prove them in examples.

\begin{lemma} \label{lem: geometric master} For a collection of polyomial valued classes $\{O_{g,n}(a_1, \dots, a_n)\}$ we have that GM-$m$ implies M-$m$, for any $m\geq 1$.
\end{lemma}

% \begin{proof} For $m=1$ we just take the constant term in $b_1$ in ${}^O\Upsilon_{g,n}^{1}(a_{1},\dots,a_{n},b_{1})$ in order to obtain $1\text{-M}$. For $m\geq 2$ note that
% \begin{align}
%     &
%     \Coeff{\prod_{i=1}^m b_i^{k_i}}{}^O\Xi_{g,n}^{m}(a_{1},\dots,a_{n},b_{1},\dots,b_{m}) = 
%     \\ \notag & \sum_{\substack{0 \leq l_i\leq k_i, \\ i=1,\dots,m}} \prod_{i=1}^m \psi_{n+i}^{l_i} \Coeff{\prod_{i=1}^m b_i^{k_i-l_i}} {}^O\Upsilon_{g,n}^{m}(a_{1},\dots,a_{n},b_{1},\dots,b_{m}).
% \end{align}
% Thus 
% \begin{align}
%     &
%     \deg_a \Coeff{\prod_{i=1}^m b_i^{k_i}}{}^O\Xi_{g,n}^{m}(a_{1},\dots,a_{n},b_{1},\dots,b_{m}) 
%     \\ \notag &
%     \leq \max_{\substack{0 \leq l_i\leq k_i, \\ i=1,\dots,m}}
%      \deg_a \Coeff{\prod_{i=1}^m b_i^{k_i-l_i}} {}^O\Upsilon_{g,n}^{m}(a_{1},\dots,a_{n},b_{1},\dots,b_{m}).
% \end{align}
% Since $\deg_a {}^O\Upsilon_{g,n}^{m} = \deg_{R^\bullet} {}^O\Upsilon_{g,n}^{m} - \deg_{b} {}^O\Upsilon_{g,n}^{m}$ \xtodo{This is wrong!}\todo{But the next formulas is still OK, right?}, the latter expression is equal to 
% \begin{align}
%     \max_{\substack{0 \leq l_i\leq k_i, \\ i=1,\dots,m}}
%      \deg_{R^\bullet} \Coeff{\prod_{i=1}^m b_i^{k_i-l_i}} {}^O\Upsilon_{g,n}^{m}(a_{1},\dots,a_{n},b_{1},\dots,b_{m}) - \sum_{i=1}^m (k_i-l_i),
% \end{align}
% which is $\leq 2g-2+m$ by assumption. 
% \end{proof}

\begin{proof}
% [Alternative proof] \xtodo{please check}\todo{I think, it is correct. Should we replace the previous proof with this one?} 
For $m=1$, we just take the constant term in $b_1$ in ${}^O\Upsilon_{g,n}^{1}(a_{1},\dots,a_{n},b_{1})$ in order to obtain $\text{M-}1$. For $m\geq 2$, observe that
\begin{align}
    {}^O\Upsilon_{g,n}^{m}(a_{1},\dots,a_{n},b_{1},\dots,b_{m}) = {}^O\Xi_{g,n}^{m}(a_{1},\dots,a_{n},b_{1},\dots,b_{m}) \prod_{i=1}^m(1-b_i\psi_{i+n}). \label{eq: master -- geometric master}
\end{align}
Moreover, $\deg_{R^\bullet}{}^O\Upsilon_{g,n}\geq\deg_{a}{}^O\Upsilon_{g,n}$. Thus, the geometric master relation implies 
\begin{align}
    \deg_{a}{}^O\Upsilon_{g,n} \leq 2g-2+m,\quad m\geq 2. 
\end{align}
The $m$th master relation for $m\geq 2$ then follows from  Eq.~(\ref{eq: master -- geometric master}) and the fact that 
\begin{align}
    \deg_{a}\left({}^O\Xi_{g,n} \prod_{i=1}^m(1-b_i\psi_{i+n})\right) =\deg_{a}{}^O\Xi_{g,n}.
\end{align}
\end{proof}

\begin{remark} Note that for $\{\bPsi_{g,n}\}$ the geometric master relations trivially coincide with the master relations. 
\end{remark}

\begin{proposition} \label{prop: geometric master}
    The collections of classes  $\{\bPsi_{g,n}\}$ and $\{\bOm_{g,n}\}$ satisfy the geometric master relations. 
\end{proposition} 

As an immediate corollary $\{\bPsi_{g,n}\}$ and $\{\bOm_{g,n}\}$ satisfy the master relations $\text{M-}m$, which, by the discussion above, implies the integrability of both type $1$ and type $2$, once the other necessary properties are established. We prove this proposition for $\{\bOm_{g,n}\}$ in Sec.~\ref{sec:localization} as a part of the proof of Thm.~\ref{thm: Hej}, and for $\{\bPsi_{g,n}\}$ it is proved in~\cite{BSS25}.

\begin{remark}
	Specializing to $b_{1}=\cdots=b_{m}=0$, the geometric and ordinary master relations are equivalent. In this specialization the statement of Prop.~\ref{prop: geometric master} for $\{\bOm_{g,n}\}$ was conjectured in~\cite{BLRS} and proved in \cite[Theorem 3.4]{BLS-Omega}.
\end{remark}

%%%%%%%%%%%%%%%%%%%%%%%%%%%%%%%%%%%%%%%%%%%%%%%%%%%%%%%%%%%%%%%%%%%
\vspace{1cm}
\section{From integrable observables to integrable systems}
\label{sec:integrable}
%%%%%%%%%%%%%%%%%%%%%%%%%%%%%%%%%%%%%%%%%%%%%%%%%%%%%%%%%%%%%%%%%%%

In this section we survey the main statements spread over the literature~\cite{BDGR1,BDGR20,BGR19,BS22} that allow to use the properties of integrable observables to construct integrable systems, either Hamiltonian, as in the case of P-CohFTs, or systems of conservation laws, as in the case of F-CohFTs. The constructions in these works are applied solely to Witten's integrable observables $\{\bPsi_{g,n}\}$, but as we shall see, all arguments are universal in the framework of Def.~\ref{def:int-obs}.  

In general, we follow closely the exposition of~\cite[Sec.~4]{BS22}; all results in this Section are a straightforward reformulation and generalization of the results in \emph{op.~cit.}

The section is organized as follows. We start in Sec.~\ref{sec: basic setup} with an introduction to the basic setup for evolutionary integrable systems (following~\cite{dubrovin2001normalformshierarchiesintegrable,BDGR1}). Sec.~\ref{sec: DR F-CohFT} and Sec.~\ref{sec: DR for P-CohFT}  introduce the canonical ``reference'' integrable systems: the DR system of conservation laws associated to an F-CohFT and the Hamiltonian DR hierarchy associated to a P-CohFT, respectively.
%~\cite{Bur,BR-spin}
%~\cite{BR-spin,ABLR}.
 
We then explain in Sec.~\ref{sec: F-CohFT to conservation laws} (resp. Sec.~\ref{sec: P-CohFT to hamilto}) how to pass from an F-CohFT (resp., P-CohFT) to a system of conservation laws (resp., Hamiltonian system) for an arbitrary collection of integrable observables, giving precise version of Thm.~\ref{thm: integrable observables vector potential} (resp., Thm.~\ref{thm: integrable observables tau function}), together with their proofs in Sec.~\ref{sec: proofs and further details}.

Next, we identify in Sec.~\ref{sec: identification A and DR} the $\bA$-hierarchies and the DR hierarchies, and derive explicit coordinate change formulas. 
%which yields an explicit change of coordinates between two natural presentation of the DR hierarchies. 
Finally we list in Sec.~\ref{sec: interrelations} properties of the integrable systems  associated to the integrable observables of Thm.~\ref{thm: Psi}, \ref{thm: A}, and \ref{thm: Hej}.

Note that in this section, we take Thm.~\ref{thm: Psi}, \ref{thm: A}, and \ref{thm: Hej} for granted; their proofs are deferred to Sec.~\ref{sec:proofs} and~\ref{sec:localization}.

\subsection{Basic setup for evolutionary equations}
\label{sec: basic setup}
The phase space is the formal loop space that can be realized as the spectrum of 
\begin{align}\label{eq:jet:coords}
\mathcal{A}_w[[\epsilon]]\coloneqq \C[[\epsilon,\{w^\alpha\}_{\alpha=1,\dots,N}]][\{w^{\alpha,d}\}_{{\alpha=1,\dots,N;\; d\geq 1}}].	
\end{align}
 Here $w^{\alpha} = w^{\alpha,0}$ are the formal dependent variables, and  $w^{\alpha,d}$ play the role of the jet coordinates for the dependence on the spatial variable $x$. The formal variable $\epsilon$ is the dispersion parameter. In particular $\partial_x \colon \mathcal{A}_w[[\epsilon]]\to \mathcal{A}_w[[\epsilon]]$ is given by
 \begin{align}
 	\d_x \coloneqq \sum_{d=0}^\infty w^{\alpha,d+1} {\partial}_{w^{\alpha,d}}.
 \end{align}
 Define the grading $\deg_{\d_x}$ on $\mathcal{A}_w[[\epsilon]]$ by $\deg_{\d_x} (w^{\alpha,d})\coloneqq d$, $\deg_{\d_x} (\epsilon)\coloneqq -1$. Let $(\mathcal{A}_w[[\epsilon]])_{\deg_{\d_x}=k}$ denote the component of $\mathcal{A}_w[[\epsilon]]$ of homogeneous $\deg_{\d_x}$ degree $k$. 
 
An \textit{evolutionary equation} is an equation of the form
\begin{align} \label{eq: basic evolutionaty equation}
	{\partial_t} w^\alpha = P^\alpha,
\end{align}
where $P^\alpha\in\mathcal{A}_w[[\epsilon]]_{\deg_{\d_x} = 1}$.  

We say that  $P^\alpha$ is a \textit{conservation law} if 
\begin{equation}
    P^\alpha = \partial_x Q^\alpha
\end{equation} 
for some $Q^\alpha\in\mathcal{A}_w[[\epsilon]]_{\deg_{\d_x} = 0}$ called \textit{flux}. Two evolutionary equations $\d_{t^i} w^\alpha = P^\alpha_i$, $i=1,2$, are \textit{compatible} (or, equivalently, the flows $\d_{t^i}$ commute), if 
\begin{equation}
    \d_{t^1}(P_2^\alpha) = \d_{t^2}(P_1^\alpha), \qquad \alpha=1,\dots,N.
\end{equation} 
An \textit{integrable system} is a system of mutually compatible evolutionary equations.  

It is useful to consider the action action of the group of Miura transformations, which is the semidirect product of two subgroups, the so-called Miura transformations of the first and the second kind. \textit{Miura transformations of the first kind} are the changes of the dependent variables  $\tilde w^{\alpha} = W^\alpha(w^1,\dots,w^N)$, where $W^\alpha$, $\alpha=1,\dots,N,$ are formal power series in $w^1,\dots,w^N$ with vanishing constant terms and $\det(\d W^\alpha /\d w^\beta)|_{w^1=\cdots=w^N=0}\not=0$. We will not need them in this paper, due to normalization condition in genus $0$ (i.e. at $\epsilon = 0$). 
\textit{Miura transformations of the second kind} (also called Miura transformations close to the identity) are the changes of the dependent variables of the type $\tilde w^{\alpha} = w^\alpha + \epsilon U^\alpha$, where $U^\alpha_k\in \mathcal{A}_w[[\epsilon]]_{\deg_{\d_x}=1}$. 

Let 
\begin{equation}
    \mathcal{F}_w[[\epsilon]]\coloneqq \mathcal{A}_w[[\epsilon]] / \d_x \mathcal{A}_w[[\epsilon]]
\end{equation} be the associated space of \textit{local functionals} and let 
\begin{equation}
    \int\!dx \colon \mathcal{A}_w[[\epsilon]] \to \mathcal{F}_w[[\epsilon]], 
\qquad f\mapsto \bar f \coloneqq \int f dx
\end{equation} 
be the projection that sends densities to their functionals (the degree $\deg_{\d_x}$ descends to $\mathcal{F}_w[[\epsilon]]$. Let $\delta_{w^\alpha} \colon \mathcal{F}_w[[\epsilon]]\to \mathcal{A}_w[[\epsilon]]$ be the variational derivatives, defined by
\begin{equation}
    \delta_{w^\alpha}\coloneqq \sum_{d\geq 0} (-\d_x)^d \partial_{w^{\alpha,d}}.
\end{equation} 
A bilinear map $\mathcal{F}_w[[\epsilon]]\times \mathcal{F}_w[[\epsilon]] \to \mathcal{F}_w[[\epsilon]]$, 
\begin{equation}
    (\bar f_1,\bar f_2) \mapsto [\bar f_1,\bar f_2] \coloneqq \iint (\delta_{w^\alpha} \bar f_1) \{w^\alpha(x),w^\beta(y)\} (\delta_{w^\beta} \bar f_2) dxdy,
\end{equation}
where $\{w^\alpha(x),w^\beta(y)\} = \sum_{d\geq 0} A^{\alpha\beta}_d\delta^{(d)}(x-y)$, $A^{\alpha\beta}_d \in \mathcal{A}_w[[\epsilon]]_{\deg_{\d_x}=1-d}$, gives a Poisson bracket if it satisfies the Jacobi identity. 

We say that Eq.~\eqref{eq: basic evolutionaty equation} possess a Hamiltonian structure if it can be written as 
\begin{align}
\d_t w^\alpha = \sum_{d\geq 0} A^{\alpha\beta}_d \d_x^d \delta_{w^\beta} \bar h
\end{align}
for a Poisson bracket $\sum_{d\geq 0} A^{\alpha\beta}_d\delta^{(d)}(x-y)$ and a Hamiltonian $\bar h \in \mathcal{F}_w[[\epsilon]]_{\deg_{\d_x}=0}$. If the Hamiltonians $\bar h_i$, $i=1,2$, commute with respect to the Poisson bracket, then their flows commute as well, so one can define an integrable system by a Poisson bracket and a system of pairwise commuting Hamiltonians.

We say that a bracket $\{w^\alpha(x),w^\beta(y)\} = \sum_{d\geq 0} A^{\alpha\beta}_d\delta^{(d)}(x-y)$ has \textit{hydrodynamical limit}, if at $\epsilon=0$ it reduces to \begin{equation}
    g^{\alpha\beta}(\{w^{*,0}\}) \delta'(x-y) + \Gamma^{\alpha\beta}_\gamma (\{w^{*,0}\}) w^{\gamma,1} \delta(x-y)
\end{equation} 
for a non-degenerate metric $g_{\alpha,\beta}$. With our normalization condition in genus $0$, we shall always have $g^{\alpha\beta} = \eta^{\alpha\beta}$, where $\eta$ is the constant metric defined by the bilinear symmetric map $\eta\colon V^{\otimes 2}\to \C$ on the underlying vector space $V$ and $\Gamma^{\alpha\beta}_\gamma=0$. 

All integrable systems that we study below have infinite set of commuting flows indexed by $(\alpha,p)$, $\alpha =1,\dots,N$, $p=0,\dots,\infty$, that is, they are of the form 

\begin{align}
    \label{eq: typical conservation laws}
    \d_{t^{\beta,p}} w^\alpha = \d_x Q^\alpha_{\beta,p} 
\end{align}
in the case of the conservation laws, or, in the case of Hamiltonian systems, we have
\begin{align}
    \d_{t^{\beta,p}} w^\alpha = \sum_{d\geq 0} A^{\alpha\beta}_d \d_x^d \delta_{w^\beta} \bar h_{\beta,p}.
\end{align}
Here $\{\bar h_{\beta,p}\}$ are the pairwise commuting Hamiltonians. 

In the former case, pick a solution $w^{\alpha}_{{\rm str}}$ of Eq.~(\ref{eq: typical conservation laws}) such that $Q^{\alpha}_{\beta,p}|_{w^{\gamma,q}=\partial^{q}_{t^{1,0}}w^{\gamma}_{\mathrm{str}}}dt^{\beta,p}$ is closed. Since the closeness implies exactness, there exists a unique vector (up to constant term)
\begin{equation}
X^{\alpha}(\{t^{\gamma,q}\})\partial_{t^{\alpha,0}},
\end{equation}
called \textit{vector potential}, such that
\begin{equation}
\label{eq: def vector potential}
Q^{\alpha}_{\beta,p}|_{w^{\gamma,q}=\partial^{q}_{t^{1,0}}w^{\gamma}_{\mathrm{str}}}=\partial_{t^{\beta,p}}X^{\alpha}.
\end{equation}
In the coordinates such that $w^{\alpha}=Q^{\alpha}_{1,0}$, we therefore have $w^{\alpha}_{\mathrm{str}}=\partial_{t^{1,0}}X^{\alpha}$. We repeatedly use the following: if $w^{\alpha}_{\mathrm{str}}=\partial_{t^{1,0}}X^{\alpha}$ satisfies the string equation (to be explicitly defined below)%\stodo{I added ``(to be explicitly defined below)''. Otherwise it all good, just one suggestion in the e-mail.}
, then the fluxes $Q^{\alpha}_{\beta,p}$ are reconstructed as the unique elements of $\mathcal{A}_{w}[[\epsilon]]_{\deg_{\partial_{x}}=0}$ satisfying Eq.~(\ref{eq: def vector potential}) (see Remark~\ref{rem:RemarkIdentification} below). In that sense we say that the system is controlled by the vector potential $X^{\alpha}(\{t^{\gamma,q}\})\partial_{t^{\alpha,0}}$.

% In the former case, we say that the system is controlled by the  so-called \textit{vector potential}
% \begin{align}
% F^\alpha(\{t^{\gamma,q}\})\d_{t^{\alpha,0}}
% \end{align}
% (note that we assign to this term in a bit different meaning than in~\cite{ABLR}) and a special solution 
% \begin{align}
% w^{\alpha}_{\mathrm{str}} = \d_{t^{1,0}}F^\alpha
% \end{align}
% (typically distinguished below by the string equation) if all fluxes $Q^{\alpha}_{\beta,p}$ are reconstructed as the unique elements of $\mathcal{A}_w[[\epsilon]]_{\deg_{\d_x} = 0}$ satisfying
% \begin{align}
%     Q^{\alpha}_{\beta,p}|_{w^{\gamma,q} = \d_{t^{1,0}}^q w^{\gamma}_{\mathrm{str}}} = \d_{t^{\beta,p}} F^\alpha
% \end{align}
% as formal power series in the variables $\{t^{\zeta,k}\}$. Notice that a vector potential exists for any integrable system of this type, since the closeness of $Q^{\alpha}_{\beta,p}|_{w^{\gamma,q} = \d_{t^{1,0}}^q w^{\gamma}_{\mathrm{str}}} dt^{\beta,p}$ implies its exactness.

In the Hamiltonian case, there is a notion of tau symmetric integrable system (see~\cite{BDGR1} for a comprehensive introduction),  in which one can define \textit{tau functions}. For our purposes, what matters is that a choice of tau function $Z$ (or rather its logarithm $F(\{t^{\gamma,q}\})=\epsilon^2\log Z$) together with a special solution $w^{\alpha}_{\mathrm{str}} = \eta^{\alpha\mu} \d_{t^{1,0}}\d_{t^{\mu,0}} F$ controls particular choices of densities and fluxes. More precisely, all densities $h_{\beta,p}$ and fluxes $Q^{\alpha}_{\beta,p}$ are reconstructed as the unique elements of $\mathcal{A}_w[[\epsilon]]_{\deg_{\d_x} = 0}$ satisfying
\begin{align}
    h_{\beta,p}|_{w^{\gamma,q}  = \d_{t^{1,0}}^q w^{\gamma}_{\mathrm{str}}} &= \d_{t^{\beta,p+1}}\d_{t^{1,0}} F
    \\ \notag 
    Q^{\alpha}_{\beta,p}|_{w^{\gamma,q}  = \d_{t^{1,0}}^q w^{\gamma}_{\mathrm{str}}} & = \eta^{\alpha\mu}\d_{t^{\beta,p}}\d_{t^{\mu,0}} F
\end{align}
as formal power series in the variables $\{t^{\zeta,k}\}$.
% In the Hamiltonian case, there is a concept of tau symmetry (see~\cite{BDGR1} for a comprehensive introduction), which for our purposes means that we can control particular choices of densities of the Hamiltonians and equations through a choice of tau function $Z$ (more precisely, we always use its logarithm $F(\{t^{\gamma,q}\})=\epsilon^2\log Z$) and a special solution $w^{\alpha}_{\mathrm{str}} = \eta^{\alpha\mu} \d_{t^{1,0}}\d_{t^{\mu,0}} F$ such that all densities $h_{\beta,p}$ and fluxes $Q^{\alpha}_{\beta,p}$ are reconstructed as the unique elements of $\mathcal{A}_w[[\epsilon]]_{\deg_{\d_x} = 0}$ satisfying
% \begin{align}
%     h_{\beta,p}|_{w^{\gamma,q}  = \d_{t^{1,0}}^q w^{\gamma}_{\mathrm{str}}} &= \d_{t^{\beta,p+1}}\d_{t^{1,0}} F
%     \\ \notag 
%     Q^{\alpha}_{\beta,p}|_{w^{\gamma,q}  = \d_{t^{1,0}}^q w^{\gamma}_{\mathrm{str}}} & = \eta^{\alpha\mu}\d_{t^{\beta,p}}\d_{t^{\mu,0}} F
% \end{align}
% as formal power series in $\{t^{\zeta,k}\}$. 
There is a subgroup of the group of Miura transformations of second kind that preserves the existence of the tau structure and the dispersionless limit and whose action is extended in a coherent way on all ingredients of the integrable system, the so-called \textit{normal Miura transformations}. They are given by $\tilde w^\alpha = w^\alpha + \eta^{\alpha\mu}\d_x \d_{t^{\mu,0}} G$ for any $G\in \mathcal{A}_w[[\epsilon]]_{\deg_{\d_x} = -2}$, and we evaluate $\d_{t^{\mu,0}} G$ as an element of $\mathcal{A}_w[[\epsilon]]_{\deg_{\d_x} = -1}$ using the integrable system itself, that is, $\d_{t^{\mu,0}} G = \sum_{q=0}^\infty \d_{w^{\gamma,q}}\d_x^{q+1}Q^{\gamma}_{\mu,0}$.

\subsection{DR hierarchy of conservation laws}
\label{sec: DR F-CohFT}
%The $\bOm$-hierarchy and $\bPsi$-Dubrovin-Zhang hierarchy cannot be compared directly. The tool for their comparison is the so-called Buryak 
The DR hierarchy, originally inspired by the construction of the quantum hierarchy in symplectic field theory~\cite{Eliashberg-SFT}, was introduced in~\cite{Bur} for CohFTs. 
In ~\cite{BR-spin} the construction was further extended to the cases of P-CohFTs and F-CohFTs. We still work with the target space of finite dimension $N$, though it is possible to extend the construction to the infinite-dimensional setup, cf.~Remark~\ref{rem:infinite-dim}.

Let $\{\mathfrak{fc}_{g,n}\}$ be a $V$-valued F-CohFT. Following~\cite{BR-spin}, for any $\alpha,\beta=1,\dots,N$ and $p\ge 0$, define the flux $Q^\alpha_{\beta,p}\in \mathcal{A}_u[[\epsilon]]_{\deg_{\d_x} = 0}$ by
\begin{align}\label{eq:definition of flux DR}
Q^\alpha_{\beta,p} & \coloneqq \sum_{\substack{g,n\ge 0\\ 2g+n>0}}\frac{\epsilon^{2g}}{n!} \sum_{\substack{d_1,\dots,d_n\geq 0 \\ d_1+\cdots+d_n=2g}} \prod_{i=1}^n u^{\gamma_i,d_i}  \times
\\ \notag & \qquad \int_{\oM_{g,n+2}}\psi_2^p \Coeff{\prod_{i=1}^n a_i^{d_i}}\lambda_g \DR_g (-\suma, 0, a_1,\dots,a_n) 
\times 
\\ \notag & \qquad \qquad
e^{\alpha} \bigg(\mathfrak{fc}_{g,n+1}\Big(e_\beta\otimes \Vec{\bigotimes}_{i=1}^n e_{\gamma_i}\Big)\bigg).
\end{align}
Here $u^{\gamma_i,d_i}$ are jet coordinates defined in~\eqref{eq:jet:coords}. These fluxes define the commuting flows by~\cite[Thm.~5.1]{BR-spin}, that is, the system of conservation laws
\begin{align} \label{eq: DR conservation laws}
    \d_{t^{\beta,p}} u^\alpha = \d_x Q^\alpha_{\beta,p}
\end{align}
is integrable, and it is called the DR hierarchy associated to an F-CohFT.  Note that in particular the flow $\d_{t^{1,0}}$ is identified with $\d_x$, and we adopt this identification in what follows. The coordinates $u^\alpha$ used in the definition are called the \emph{natural coordinates}.

This hierarchy can also be described in terms of its vector potential and the F-string solution, which we now introduce. The vector potential is $X^{\DR,\alpha}\partial_{t^{\alpha, 0}}$, defined by 
\begin{align}
    X^{\DR,\alpha} & \coloneqq \sum_{\substack{g\ge 0,n\ge 1\\ 2g-1+n>0}} \frac{\epsilon^{2g}}{n!} \sum_{{d_1,\dots,d_n\geq 0}} \prod_{i=1}^n t^{\gamma_i,d_i} \times 
    \\ \notag & \qquad 
    \int_{\oM_{g,n+1}} \bigg[\Coeff{\prod_{i=1}^n a_i^{d_i}} \bA^1_{g,n}(a_1,\dots,a_n)\bigg]^\circlearrowright \times
    \\ \notag & \qquad \qquad e^{\alpha} \bigg(\mathfrak{fc}_{g,n}\Big(\Vec{\bigotimes}_{i=1}^n e_{\gamma_i}\Big)\bigg),
    & \alpha=1,\dots,N,
\end{align}

 where we recall that $[\cdots]^\circlearrowright$ is the operation of relabeling marked points 
%of ``$\mod (n+1)$ relabeling'' of the marked points 
that makes the 
%$(n+1)$-st 
last marked point the zeroth one, which is needed to match the convention we used in the definition of F-CohFT. The \textit{F-string solution} is defined by
\begin{align}
    u^{\alpha}_{\text{F-str}} \coloneqq \d_{t^{1,0}}X^{\DR,\alpha}, \qquad \alpha=1,\dots,N.
\end{align}
%be a special solution of~\eqref{eq: DR conservation laws} specified by $u^{\alpha}_{\text{F-str}}|_{t^{*,\geq 1} = 0} = t^{\alpha,0}$, $\alpha=1,\dots,N$.
This is in fact a special solution of~\eqref{eq: DR conservation laws} specified by $u^{\alpha}_{\text{F-str}}|_{t^{*,\geq 1} = 0} = t^{\alpha,0}$, $\alpha=1,\dots,N$, see~\cite[Thm.~4.9]{BS22} in combination with~\cite[Lemma~2.6]{BLS-DRDZ}. It can be expressed in terms of the class $\bA_{g,n+1}^{1}(a_{1},\dots,a_{n},0)$ as
\begin{align} \label{eq: def: F-str u}
      u^{\alpha}_{\text{F-str}} & \coloneqq \sum_{{g\ge 0,n\ge 1}} \frac{\epsilon^{2g}}{n!} \sum_{{d_1,\dots,d_n\geq 0}} \prod_{i=1}^n t^{\gamma_i,d_i} \times 
    \\ \notag & \qquad 
    \int_{\oM_{g,n+2}} \bigg[\Coeff{\prod_{i=1}^n a_i^{d_i}} \bA^1_{g,n+1}(a_1,\dots,a_n,0)\bigg]^\circlearrowright \times
    \\ \notag & \qquad \qquad e^{\alpha} \bigg(\mathfrak{fc}_{g,n+1}\Big(\Vec{\bigotimes}_{i=1}^n e_{\gamma_i} \otimes e_1\Big)\bigg),
        & \alpha=1,\dots,N.
\end{align}

Then, identifying $u^{\gamma,d}_{\text{F-str}}$ with $\d_{t^{1,0}}^d u^{\gamma}_{\text{F-str}}$, $\gamma =1,\dots, N$, $d\geq 0$, we can reconstruct $Q^\alpha_{\beta,p}$ as the unique elements of $\mathcal{A}_u[[\epsilon]]_{\deg_{\d_x} = 0}$ satisfying
\begin{align}
    \partial_{t^{\beta,p}}X^{\DR,\alpha} = Q^\alpha_{\beta,p}|_{u^{\gamma,q}= u^{\gamma,q}_{\text{F-str}},\,\gamma=1,\dots,N,\, q\geq 0}.
\end{align}
for all $\alpha,\beta=1,\dots,N$, $p\geq 0$. To this end, one needs to use the string equation and~\cite[Lemma 4.3]{BS22} for identification of fluxes, see also Remark~\ref{rem:RemarkIdentification} below.
We refer to~\cite[Sec.~4.4.1, 4.4.2]{BS22} for a full exposition of this approach. 

\subsection{Hamiltonian DR hierarchy}
\label{sec: DR for P-CohFT}
Let $\{\mathfrak{pc}_{g,n}\}$ be a $(V,\eta)$-valued P-CohFT.
%, where we recall that $V=\langle e_1,\dots,e_N\rangle$ with the inner product $\eta$. 
We can repeat the construction of the previous section using only part of the structure of P-CohFT, that is, the canonically associated F-CohFT. Then the formulas for the fluxes ~\cite{BR-spin} can be reproduced as follows. For any $\alpha,\beta=1,\dots,N$ and $p\ge 0$, define the flux $Q^\alpha_{\beta,p}\in \mathcal{A}_u[[\epsilon]]_{\deg_{\d_x} = 0}$ by
\begin{align}\label{eq:definition of flux DR PCohFT} 
Q^\alpha_{\beta,p} & \coloneqq \sum_{\substack{g,n\ge 0\\ 2g+n>0}}\frac{\epsilon^{2g}}{n!} \sum_{\substack{d_1,\dots,d_n\geq 0 \\ d_1+\cdots+d_n=2g}} \prod_{i=1}^n u^{\gamma_i,d_i}  \times
\\ \notag & \qquad \int_{\oM_{g,n+2}}\psi_2^p \Coeff{\prod_{i=1}^n a_i^{d_i}}\lambda_g \DR_g (-\suma, 0, a_1,\dots,a_n) 
\times 
\\ \notag & \qquad \qquad
\eta^{\alpha\mu} \mathfrak{pc}_{g,n+2}\Big(e_{\mu}\otimes e_\beta\otimes \Vec{\bigotimes}_{i=1}^n e_{\gamma_i}\Big).
\end{align}
However, in this case the integrable hierarchy~\eqref{eq: DR conservation laws} appears to be Hamiltonian. For any $1\le\alpha\le N$ and $p\ge -1$, define $h_{\alpha,p}\in \mathcal{A}_u[[\epsilon]]_{\deg_{\d_x} = 0}$ by
\begin{align}\label{eq:definition of DR hamiltonians}
h_{\alpha,p} & \coloneqq \eta_{1\mu} Q^{\mu}_{\alpha,p+1}. 
\end{align}
Consider the Poisson bracket $\{u^\alpha,u^\beta\}_{\eta^{-1}\d_x}\coloneqq \eta^{\alpha\beta} \delta(x-y)$.
The main result of \cite{Bur} states that 
\begin{equation}
    \{\bar h_{\alpha_1,d_1},\bar h_{\alpha_2,d_2}\}_{\eta^{-1}\d_x}=0
\end{equation}
for any $1\le\alpha_1,\alpha_2\le N$ and $d_1,d_2\ge 0$. This allows to rewrite the system of conservation laws in natural coordinates~\eqref{eq: DR conservation laws} as a Hamiltonian system:
%
%\begin{definition}The Hamiltonian hierarchy of PDEs corresponding to the Poisson bracket given by the operator $\eta^{\alpha\beta}\d_x$ and the Hamiltonian given by local functionals $\overline{g}_{\alpha,d}$ is called the Hamiltonian DR hierarchy. As a system of evolutionary equations it reads:
\begin{align}\label{eq: DR hierarchy Hamilonian form}
\d_{t^{\beta,p}} u^\alpha = \d_x Q^{\alpha}_{\beta,p} = \eta^{\alpha\mu}\d_x \delta_{u^\mu} \bar h_{\beta,p}.
\end{align}

In this setting of P-CohFT, the system is proved to be tau-symmetric in~\cite{BDGR1} and it can be described in the \textit{normal coordinates}
\begin{align}
    u^{\alpha}_{\mathrm{norm}} \coloneqq \eta^{\alpha\mu} {h}_{\mu,-1}
\end{align}
in terms of the tau function. 
\begin{remark}
    \label{rem: normal for F-CohFT}
    Note that in the context of F-CohFT, the same change of variable holds in the form 
    \begin{align}
        \label{eq: normal coordinates for F CohFT}
        u^{\alpha}_{\mathrm{norm}} & = \sum_{\substack{g,n\ge 0\\ 2g+n>0}}\frac{\epsilon^{2g}}{n!} \sum_{\substack{d_1,\dots,d_n\geq 0 \\ d_1+\cdots+d_n=2g}} \prod_{i=1}^n u^{\gamma_i,d_i}  \times
\\ \notag & \qquad \int_{\oM_{g,n+2}} \Coeff{\prod_{i=1}^n a_i^{d_i}}\lambda_g \DR_g (0,-\suma, a_1,\dots,a_n) 
\times 
\\ \notag & \qquad \qquad e^{\alpha} \bigg(\mathfrak{fc}_{g,n+1}\Big(e_1\otimes \Vec{\bigotimes}_{i=1}^n e_{\gamma_i}\Big)\bigg).
    \end{align}
    However, this does not provide normal coordinates in the usual sense, as the system is not tau-symmetric.
\end{remark}
The logarithm of the tau function is given by 
\begin{align}
    F^{\DR} & \coloneqq \sum_{\substack{g\ge 0,n\ge 1\\ 2g-2+n>0}} \frac{\epsilon^{2g}}{n!} \sum_{{d_1,\dots,d_n\geq 0}} \prod_{i=1}^n t^{\gamma_i,d_i} \times 
    \\ \notag & \qquad 
    \int_{\oM_{g,n}} \Coeff{\prod_{i=1}^n a_i^{d_i}} \bA_{g,n}(a_1,\dots,a_n) %\times
    %\\ \notag & \qquad \qquad  
    \mathfrak{pc}_{g,n}\Big(\Vec{\bigotimes}_{i=1}^n e_{\gamma_i}\Big).
\end{align}
Note that, by Lemma~\ref{lem:push-forward-A1}, we have $\d_{t^{1,0}}F^{\DR} = \eta_{1\mu}X^{\DR,\mu}$, where the latter $X^{\DR,\mu}$ are the ones defined for the canonically associated F-CohFT.
%
%The tau symmetry was established in \cite{BDGR20}: for $d_{1},d_{2}\geq0$ and $1\leq\alpha_{1},\alpha_{2}\leq N$ 
%\begin{align}
%\left\{g_{d_{1}-1,\alpha_{1}},\overline{g}_{d_{2},\alpha_2}\right\}_{\eta\d_x} = \left\{g_{d_{2}-1,\alpha_2},\overline{g}_{d_{1},\alpha_1}\right\}_{\eta\d_x}.
%\end{align}
%
The \textit{P-string solution} is defined by
\begin{align}
    u^{\alpha}_{\text{P-str}} \coloneqq \eta^{\alpha\mu}\d_{t^{\mu,0}}\d_{t^{1,0}}F^{\DR}, \qquad \alpha=1,\dots,N.
\end{align}
This is in fact a solution of~\eqref{eq: DR hierarchy Hamilonian form} considered in normal coordinates, see~\cite[Sec. 4.4.4]{BS22} in combination with~\cite[Lemma~2.6]{BLS-DRDZ}. It can be expressed in terms of the class $\bA_{g,n+1}^{1}(a_{1},\dots,a_{n},0)$ as
\begin{align}
      u^{\alpha}_{\text{P-str}} & \coloneqq \sum_{{g\ge 0,n\ge 1}} \frac{\epsilon^{2g}}{n!} \sum_{{d_1,\dots,d_n\geq 0}} \prod_{i=1}^n t^{\gamma_i,d_i} \times 
    \\ \notag & \qquad 
    \int_{\oM_{g,n+2}} \Coeff{\prod_{i=1}^n a_i^{d_i}} \bA^1_{g,n+1}(a_1,\dots,a_n,0) \times
    \\ \notag & \qquad \qquad \eta^{\alpha\mu}\mathfrak{pc}_{g,n+2}\Big(\Vec{\bigotimes}_{i=1}^n e_{\gamma_i} \otimes e_\mu \otimes e_1\Big),
        & \alpha=1,\dots,N.
\end{align}
(note the difference of this formula and Eq.~\eqref{eq: def: F-str u} for $u^{\alpha}_{\text{F-str}}$ for the canonically associated F-CohFT).
Then, identifying $u^{\gamma,d}_{\text{P-str}}$ with $\d_{t^{1,0}}^d u^{\gamma}_{\text{P-str}}$, $\gamma =1,\dots, N$, $d\geq 0$, we can reconstruct $Q^\alpha_{\beta,p}$ and $h_{\beta,p}$ as the unique elements of $\mathcal{A}_u[[\epsilon]]_{\deg_{\d_x} = 0}$ satisfying
\begin{align}
    \eta^{\alpha\mu}\partial_{t^{\mu,0}}\partial_{t^{\beta,p}}F^{\DR} = Q^\alpha_{\beta,p}|_{u^{\gamma,q}= u^{\gamma,q}_{\text{P-str}},\,\gamma=1,\dots,N,\, q\geq 0}; 
    \\ \notag 
    \partial_{t^{1,0}}\partial_{t^{\beta,p+1}}F^{\DR} = h_{\beta,p}|_{u^{\gamma,q}= u^{\gamma,q}_{\text{P-str}},\,\gamma=1,\dots,N,\, q\geq 0};     
\end{align}
for all $\alpha,\beta=1,\dots,N$, $p\geq 0$. 
We refer to~\cite[Sec.~4.4.4]{BS22} for a full exposition of this approach. 

\subsection{From F-CohFT to a system of conservation laws} 
\label{sec: F-CohFT to conservation laws}
Consider the vector field ${}^O\!X = {{}^O\!X}^{\alpha} \partial_{t^{\alpha,0}}$, defined in Eq.~(\ref{vector potential integrable observable}), associated to an F-CohFT $\{\mathfrak{fc}_{g,n}\}$ for an arbitrary system of integrable observables $\{O_{g,n}(a_1,\dots,a_n)\}$.

The goal is to construct an integrable system of conservation laws with a distinguished solution (the so-called \emph{topological solution}) given by $w_{\mathrm{top}}^{\alpha} \coloneqq \d_{t^{1,0}} {{}^O\!X}^{\alpha}$ and $w_{\mathrm{top}}^{\alpha,d} \coloneqq \d^{d+1}_{t^{1,0}} {{}^O\!X}^{\alpha}$ for $d\geq 0$. To this end, we have the following 

\begin{proposition} \label{prop: integrability of integrable observables for F-CohFT} (1) Assume $\{O_{g,n}(a_1,\dots,a_n)\}$ are integrable observables. Then there exists a unique system of fluxes $R^{\alpha}_{\beta,p}\in \mathcal{A}_w[[\epsilon]]_{\deg_{\d_x}=0}$ such that 
\begin{align} \label{eq: O-obs fluxed F-CohFT}
    \d_{t^{\beta,p}} {{}^O\!X}^{\alpha} = R^{\alpha}_{\beta,p}\big|_{w^{\gamma,q}=w^{\gamma,q}_{\mathrm{top}},\,\gamma=1,\dots,N,\, q\geq 0}.
\end{align}
The system of conservations laws 
\begin{align} \label{eq: O-obs conservation laws F-CohFT}
    \d_{t^{\beta,p}} w^\alpha = \d_x R^{\alpha}_{\beta,p}, \qquad \beta=1,\dots,N,\, p\geq 0,
\end{align}
is integrable and $w_{\mathrm{top}}^{\alpha}$ is a solution.

Moreover, in the dispersionless limit $\epsilon\to0$, the fluxes $R^{\alpha}_{\beta,p}\vert_{\epsilon=0}$ are independent of the choice of the integrable observable.

(2) If $\{O_{g,n}(a_1,\dots,a_n)\}$ are integrable observables of type $1$, then this system of conservation laws is Miura equivalent to the DR hierarchy of conservation laws associated to $\{\mathfrak{fc}_{g,n}\}$. The Miura transformation is given by
\begin{align} \label{eq:Miura-O-F-CohFT}
    u^{\alpha} & = w^{\alpha} - \d_x R^\alpha,
    \\ \notag
    R^\alpha & = \sum_{g=0}^\infty\sum_{n=1}^\infty \frac{\epsilon^{2g}}{n!} \sum_{\substack{q_1,\dots,q_n\geq 0 \\ q_1+\cdots+q_n=2g-1}} \prod_{i=1}^n w^{\gamma_i,q_i}
\times \\ \notag & \quad \qquad  
\int_{\oM_{g,n+1}} \bigg[\Coeff{b_1^0\prod_{i=1}^n a_i^{q_i}} {}^{O}\!B^{1}_{g,n} \bigg]^\circlearrowright
%\times \\ \notag & \qquad \qquad 
e^{\alpha}\bigg(\mathfrak{fc}_{g,n}\Big( \Vec{\bigotimes}_{i=1}^n e_{\gamma_i} \Big) \bigg).
\end{align}

(3) If $\{O_{g,n}(a_1,\dots,a_n)\}$ are integrable observables of type $2$, then we have the following formula: 
\begin{align}
R^{\alpha}_{\beta,p} & = \sum_{g=0}^\infty\sum_{n=1}^\infty \frac{\epsilon^{2g}}{n!} \sum_{\substack{q_1,\dots,q_n\geq 0 \\ q_1+\cdots+q_n=2g}} \prod_{i=1}^n w^{\gamma_i,q_i}  
\times \\ \notag & \qquad  
\int_{\oM_{g,n+2}} \bigg[\Coeff{b_1^p b_2^0\prod_{i=1}^n a_i^{q_i}} {}^{O}\!B^{2}_{g,n} \bigg]^\circlearrowright
%\times \\ \notag & \qquad \qquad 
e^{\alpha}\bigg(\mathfrak{fc}_{g,n+1}\Big( \Vec{\bigotimes}_{i=1}^n e_{\gamma_i} \otimes e_\beta\Big) \bigg).
\end{align}
%where $[\cdots]^\circlearrowright$ is the operation of ``$\mod (n+2)$ relabeling'' of the marked points that makes the $(n+2)$-nd marked point the zeroth one, which is needed to match the convention we used in the definition of F-CohFT. 
\end{proposition}

%Using the fluxes $R^{\alpha,0}_{\beta,p}$, we construct the integrable system:
%\begin{proposition}
%The system of conservations laws 
%\begin{align}
%    \d_{t^{\beta,p}} w^\alpha = \d_x R^{\alpha,0}_{\beta,p}, \qquad \beta=1,\dots,N,\, p\geq 0,
%\end{align}
%is integrable and is Miura equivalent to the DR hierarchy of conservation laws. 
%\end{proposition}

The proof follows \emph{mutatis mutandis} the proof of~\cite[Thm.~4.6 and Thm.~4.10]{BS22} presented there for $\{O_{g,n} = \bPsi_{g,n}\}$, with the further improvements and observations given in~\cite[Thm.~2.15]{BLRS} and~\cite[Lem.~2.6]{BLS-DRDZ}. Below we guide the reader through the main steps of the proof, with precise references, and also give an explicit formula for the Miura transformation in the case of type $1$ integrable observables.

\subsection{From P-CohFT to a Hamiltonian system}
\label{sec: P-CohFT to hamilto}
Consider the logarithm of the partition function ${}^O\!F$, defined in Eq.~(\ref{tau funtion integrable observable}), associated to a P-CohFT $\{\mathfrak{pc}_{g,n}\}$ for an arbitrary system of integrable observables $\{O_{g,n}(a_1,\dots,a_n)\}$. The goal is to construct an integrable system of conservation laws with a distinguished solution (the so-called \emph{topological solution}) given by $w_{\mathrm{top}}^{\alpha} \coloneqq \eta^{\alpha\mu}\d_{t^{1,0}} \d_{t^{\mu,0}}  {}^O\!F$ and $w_{\mathrm{top}}^{\alpha,d} \coloneqq \d^{d}_{t^{1,0}} w_{\mathrm{top}}^{\alpha}$ for $d\geq 0$. To this end, we have the following. 

\begin{proposition} \label{prop: integrability of integrable observables for P-CohFT} (1) Assume $\{O_{g,n}(a_1,\dots,a_n)\}$ are integrable observables. Then there exists a unique system of fluxes $R^{\alpha}_{\beta,p}\in \mathcal{A}_w[[\epsilon]]_{\deg_{\d_x}=0}$ such that 
\begin{align}
    \eta^{\alpha\mu} \d_{t^{\beta,p}}\d_{t^{\mu,0}} {}^O\!F = R^{\alpha}_{\beta,p}\big|_{w^{\gamma,q}=w^{\gamma,q}_{\mathrm{top}},\,\gamma=1,\dots,N,\, q\geq 0}.
\end{align}
The system of conservations laws 
\begin{align}
    \d_{t^{\beta,p}} w^\alpha = \d_x R^{\alpha}_{\beta,p}, \qquad \beta=1,\dots,N,\, p\geq 0,
\end{align}
is integrable and $w_{\mathrm{top}}^{\alpha}$ is a solution. Moreover, in the dispersionless limit $\epsilon\to0$, the fluxes $R^{\alpha}_{\beta,p}\vert_{\epsilon=0}$ are independent of the choice of the integrable observable.

(2) If $\{O_{g,n}(a_1,\dots,a_n)\}$ are integrable observables of type $1$, then this system of conservation laws is normal Miura equivalent to the DR hierarchy of conservation laws associated to a P-CohFT $\{\mathfrak{pc}_{g,n}\}$, and hence possesses a Hamiltonian structure and it is tau symmetric. 

%\xtodo{give explicit Miura} \todo{Normal Miura becomes explicit only if we also assume type $2$ integrability, see the proof. We can add point (4) with the assumption of both type $1$ and type $2$.}

(3) If $\{O_{g,n}(a_1,\dots,a_n)\}$ are integrable observables of type $2$, then we have the following formulas for the fluxes and the Hamiltonian densities: 
\begin{align} \label{eq:R-alpha-beta-p}
R^{\alpha}_{\beta,p} & = \sum_{g=0}^\infty\sum_{n=1}^\infty \frac{\epsilon^{2g}}{n!} \sum_{\substack{q_1,\dots,q_n\geq 0 \\ q_1+\cdots+q_n=2g}} \prod_{i=1}^n w^{\gamma_i,q_i}  
\times \\ \notag & \qquad  
\int_{\oM_{g,n+2}} \Coeff{b_1^p b_2^0\prod_{i=1}^n a_i^{q_i}} {}^{O}\!B^{2}_{g,n} 
%\times \\ \notag & \qquad \qquad 
\eta^{\alpha\mu}\mathfrak{pc}_{g,n+1}\Big( \Vec{\bigotimes}_{i=1}^n e_{\gamma_i} \otimes e_\beta\otimes e_\mu\Big);
\\ 
h_{\beta,p} & = \sum_{g=0}^\infty\sum_{n=1}^\infty \frac{\epsilon^{2g}}{n!} \sum_{\substack{q_1,\dots,q_n\geq 0 \\ q_1+\cdots+q_n=2g}} \prod_{i=1}^n w^{\gamma_i,q_i}  
\times \\ \notag & \qquad  
\int_{\oM_{g,n+2}} \Coeff{b_1^{p+1} b_2^0\prod_{i=1}^n a_i^{q_i}} {}^{O}\!B^{2}_{g,n} 
\mathfrak{pc}_{g,n+1}\Big( \Vec{\bigotimes}_{i=1}^n e_{\gamma_i} \otimes e_\beta\otimes e_1\Big);
\end{align}

(4) If $\{O_{g,n}(a_1,\dots,a_n)\}$ are integrable observables of type $1$ and type $2$ simultaneously, then the normal Miura transformation mentioned above can be given explicitly as
\begin{align} \label{eq: Normal Miura explicit}
    u^{\alpha}_{\mathrm{norm}} & = w^{\alpha}  - \eta^{\alpha\mu}\d_{x}\Big(\d_{w^{\zeta,k}} R \cdot \d_x^{d+1} R^\zeta_{\mu,0}\Big), 
    \\ \notag 
    R & = \sum_{g=0}^\infty\sum_{n=1}^\infty \frac{\epsilon^{2g}}{n!} \sum_{\substack{q_1,\dots,q_n\geq 0 \\ q_1+\cdots+q_n=2g-2}} \prod_{i=1}^n w^{\gamma_i,q_i}
%\times \\ \notag & \quad \qquad  
\int_{\oM_{g,n}} \Coeff{\prod_{i=1}^n a_i^{q_i}} {}^{O}\!B^{0}_{g,n} 
%\times \\ \notag & \qquad \qquad 
\mathfrak{pc}_{g,n}\Big( \Vec{\bigotimes}_{i=1}^n e_{\gamma_i} \Big),
\end{align}
and $R^\zeta_{\mu,0}$ are given above in Eq.~\eqref{eq:R-alpha-beta-p}.
\end{proposition}

%Using the fluxes $R^{\alpha,0}_{\beta,p}$, we construct the integrable system:
%\begin{proposition}
%The system of conservations laws 
%\begin{align}
%    \d_{t^{\beta,p}} w^\alpha = \d_x R^{\alpha,0}_{\beta,p}, \qquad \beta=1,\dots,N,\, p\geq 0,
%\end{align}
%is integrable and is Miura equivalent to the DR hierarchy of conservation laws. 
%\end{proposition}

The proof follows \emph{mutatis mutandis} the argument of~\cite[Sec.~4.4.4]{BS22} presented there for $\{O_{g,n} = \bPsi_{g,n}\}$, which itself summarizes the original arguments of~\cite{BDGR1,BDGR20,BGR19}, with the further improvements and observations given in~\cite[Thm.~2.15]{BLRS} and~\cite[Lem.~2.6]{BLS-DRDZ}. Below we guide the reader through the main steps of the proof, with precise references, and also give an explicit formula for the normal Miura transformation in the case of integrable observables that are simultaneously of type $1$ and type $2$. 

\subsection{Proofs and further details} 
\label{sec: proofs and further details}
The main steps of the proofs of Prop.~\ref{prop: integrability of integrable observables for F-CohFT} and~\ref{prop: integrability of integrable observables for P-CohFT} are the following. 

\begin{lemma} (1) Let ${{}^O\!X}$ be the vector potential of an F-CohFT $\{ \mathfrak{fc}_{g,n}\}$. For any $\alpha=1,\dots,N$ we have 
\begin{align}
\d_{t^{1,0}} {{}^O\!X}^{\alpha} & = t^{\alpha,0} + \sum_{k=0}^\infty t^{\beta,k+1} \d_{t^{\beta,k}} {{}^O\!X}^{\alpha};
\\ 
\d_{t^{1,1}}  {{}^O\!X}^{\alpha} & = \sum_{k=0}^\infty t^{\beta,k} \d_{t^{\beta,k}} {{}^O\!X}^{\alpha} + \epsilon \partial_\epsilon {{}^O\!X}^{\alpha} - {{}^O\!X}^{\alpha}. 
\end{align}

(2) Let ${{}^O\!F}$ be the logarithm of the partition function of a P-CohFT $\{ \mathfrak{pc}_{g,n}\}$. We have 
\begin{align}
\d_{t^{1,0}} {{}^O\!F} & = \frac 12 \eta_{\alpha\beta}t^{\alpha,0}t^{\beta,0} + \sum_{k=0}^\infty t^{\alpha,k+1} \d_{t^{\alpha,k}} {{}^O\!F};
\\ 
\d_{t^{1,1}}  {{}^O\!F} & = \sum_{k=0}^\infty t^{\alpha,k} \d_{t^{\alpha,k}} {{}^O\!F} + \epsilon \partial_\epsilon {{}^O\!F} - 2\,{{}^O\!F}. 
\end{align}
\end{lemma}

These are the string and the dilaton equation in the context of F-CohFTs and P-CohFTs, respectively. They follow from the push-forward property of $\{O_{g,n}(a_1,\dots,a_n)\}$ given in~\eqref{eq:push-forward-O} combined with the flat unit property of F-CohFT. 

\begin{remark} \label{rem:RemarkIdentification}
    One of the main consequences of the string equation is the following property of the so-called topological dependent coordinates, which are defined as $w_{\mathrm{top}}^{\alpha,d} \coloneqq \d^{d+1}_{t^{1,0}} {{}^O\!X}^{\alpha}$ in the context of F-CohFTs and $w_{\mathrm{top}}^{\alpha,d} \coloneqq \eta^{\alpha\mu}\d^{d+1}_{t^{1,0}} \d_{t^{\mu,0}} {{}^O\!F}$ in the context of P-CohFTs:
    \begin{align}
        w_{\mathrm{top}}^{\alpha,d} = t^{\alpha,d} + \delta^{\alpha}_{1} \delta^d_1 + r^{\alpha,d} + O(\epsilon^2),
    \end{align}
    where $r^{\alpha,d}$ is a series in the variables $\{t^{\beta,k}\}$ such that $\Coeff{\prod_{i=1}^n t^{\beta_i,k_i}} r^{\alpha,s} = 0 $ if $n=1$ or $\sum_{i=1}^n k_i \leq d$. This gives us a change of variables that maps an arbitrary formal power series in $t^{\alpha,d}$ to a formal power series in $w_{\mathrm{top}}^{\alpha,d}-\delta^{\alpha}_1\delta^{d}_1$, and we can use it to check various identities between elements of $\mathcal{A}_w[[\epsilon]]$ by substitution $w^{\alpha,d} = w_{\mathrm{top}}^{\alpha,d}$, cf.~\cite[Lemma 4.3]{BS22}
\end{remark}

Notice that these are the only equations for $\bPsi$ used in~\cite[Sec.~7.3]{BDGR1},~\cite[Sec.~3.4]{BGR19}, and~\cite[Sec. 4]{BS22} (along with the homogeneity, which is hidden in the notation that identifies $\prod_{i=1}^n\psi_i^{d_i}$ with $\Coeff{\prod_{i=1}^n\psi_i^{d_i}}\bPsi(a_1,\dots,a_n)$ as in~\cite[Sec. 2.5]{BLRS}) to prove the following statement. This is the reason why one can conclude that the same statement holds for any $\{O_{g,n}\}$ satisfying the string, the dilaton, and the homogenenity property:

\begin{lemma} \label{lemma: vector potential expansion} 
Let ${{}^O\!X}$ be the vector potential of an F-CohFT $\{ \mathfrak{fc}_{g,n}\}$. Let 
$$
w_{\mathrm{top}}^{\alpha} \coloneqq \d_{t^{1,0}} {{}^O\!X}^{\alpha}
$$ 
and $w_{\mathrm{top}}^{\alpha,d} \coloneqq \d^{d+1}_{t^{1,0}} {{}^O\!X}^{\alpha}$ for $d\geq 0$. Then there exist a unique $R^\alpha \in \mathcal{A}_w[[\epsilon]]_{\deg_{\d_x}=-1}$ and a unique $R^{\alpha}_{\beta,p}\in \mathcal{A}_w[[\epsilon]]_{\deg_{\d_x}=0}$ such that 
\begin{align}
    \Coeff{\epsilon^{2g}\prod_{i=1}^n t^{\zeta_iq_i}} \bigg( {{}^O\!X}^\alpha -R^\alpha|_{w^{\gamma,d}=w_{\mathrm{top}}^{\gamma,d}} \bigg) & = 0 & \text{for}\ &  
    \sum\nolimits_{i=1}^n q_i \leq 2g-1,
    \\
    \Coeff{\epsilon^{2g}\prod_{i=1}^n t^{\zeta_iq_i}} \bigg( \d_{t^{\beta,p}}{{}^O\!X}^\alpha -R^\alpha_{\beta,p}|_{w^{\gamma,d}=w_{\mathrm{top}}^{\gamma,d}} \bigg) & = 0 & \text{for}\  &
    \sum\nolimits_{i=1}^n q_i \leq 2g,
\end{align}
and any $1 \leq \zeta_i \leq N$.
More explicitly, we have
\begin{align} \label{eq:X-alpha-decomposition}
    {{}^O\!X}^{\alpha} &=  \sum_{g=0}^\infty\sum_{n=1}^\infty \frac{\epsilon^{2g}}{n!} \sum_{\substack{q_1,\dots,q_n\geq 0 \\ q_1+\cdots+q_n\geq 2g}} \prod_{i=1}^n t^{\gamma_i,q_i}  
\times \\ \notag & \quad \qquad  
\int_{\oM_{g,n+1}} \bigg[\Coeff{b_1^0\prod_{i=1}^n a_i^{q_i}} {}^{O}\!B^{1}_{g,n} \bigg]^\circlearrowright
%\times \\ \notag & \qquad \qquad 
e^{\alpha}\bigg(\mathfrak{fc}_{g,n}\Big( \Vec{\bigotimes}_{i=1}^n e_{\gamma_i} \Big) \bigg)
\\ \notag & \quad + \sum_{g=0}^\infty\sum_{n=1}^\infty \frac{\epsilon^{2g}}{n!} \sum_{\substack{q_1,\dots,q_n\geq 0 \\ q_1+\cdots+q_n=2g-1}} \prod_{i=1}^n w^{\gamma_i,q_i}_{\mathrm{top}}
\times \\ \notag & \quad \qquad  
\int_{\oM_{g,n+1}} \bigg[\Coeff{b_1^0\prod_{i=1}^n a_i^{q_i}} {}^{O}\!B^{1}_{g,n} \bigg]^\circlearrowright
%\times \\ \notag & \qquad \qquad 
e^{\alpha}\bigg(\mathfrak{fc}_{g,n}\Big( \Vec{\bigotimes}_{i=1}^n e_{\gamma_i} \Big) \bigg),
\end{align}
where the second summand is equal to $R^\alpha|_{w^{\gamma,d}=w_{\mathrm{top}}^{\gamma,d}}$, and
\begin{align} \label{eq:X-alpha-beta-p-decomposition} 
\d_{t^{\beta,p}}{{}^O\!X}^\alpha 
& = \sum_{g=0}^\infty\sum_{n=1}^\infty \frac{\epsilon^{2g}}{n!} \sum_{\substack{q_1,\dots,q_n\geq 0 \\ q_1+\cdots+q_n\geq 2g+1}} \prod_{i=1}^n t^{\gamma_i,q_i}  
\times \\ \notag & \qquad  
\int_{\oM_{g,n+2}} \bigg[\Coeff{b_1^p b_2^0\prod_{i=1}^n a_i^{q_i}} {}^{O}\!B^{2}_{g,n} \bigg]^\circlearrowright
%\times \\ \notag & \qquad \qquad 
e^{\alpha}\bigg(\mathfrak{fc}_{g,n+1}\Big( \Vec{\bigotimes}_{i=1}^n e_{\gamma_i} \otimes e_\beta\Big) \bigg)
\\ \notag 
& \quad + \sum_{g=0}^\infty\sum_{n=1}^\infty \frac{\epsilon^{2g}}{n!} \sum_{\substack{q_1,\dots,q_n\geq 0 \\ q_1+\cdots+q_n= 2g}} \prod_{i=1}^n w_{\mathrm{top}}^{\gamma_i,q_i}  
\times \\ \notag & \quad \qquad  
\int_{\oM_{g,n+2}} \bigg[\Coeff{b_1^p b_2^0\prod_{i=1}^n a_i^{q_i}} {}^{O}\!B^{2}_{g,n} \bigg]^\circlearrowright
%\times \\ \notag & \qquad \qquad 
e^{\alpha}\bigg(\mathfrak{fc}_{g,n+1}\Big( \Vec{\bigotimes}_{i=1}^n e_{\gamma_i} \otimes e_\beta\Big) \bigg),
\end{align}
where the second summand is equal to $R^\alpha_{\beta,p}|_{w^{\gamma,d}=w_{\mathrm{top}}^{\gamma,d}}$.
\end{lemma}

The integrability statement for F-CohFTs is a direct corollary.

\begin{proof}[Proof of Prop.~\ref{prop: integrability of integrable observables for F-CohFT}] Assume that $\{O_{g,n}\}$ are type $1$ integrable observables. Then the first summand in~\eqref{eq:X-alpha-decomposition} is equal to $X^{\DR,\alpha}$. More precisely, we have
\begin{align}
    X^{\DR,\alpha} = {{}^O\!X}^{\alpha} - R^\alpha|_{w^{\gamma,d}=w_{\mathrm{top}}^{\gamma,d}},
\end{align}
which implies 
\begin{align}
\d_{t^{1,0}}X^{\DR,\alpha} = \d_{t^{1,0}} {{}^O\!X}^{\alpha} - \d_x R^\alpha|_{w^{\gamma,d}=w_{\mathrm{top}}^{\gamma,d}},    
\end{align}
or equivalently
\begin{align}
    u^{\alpha}_{\text{F-str}} = w_{\mathrm{top}}^{\alpha} - \d_x R^\alpha|_{w^{\gamma,d}=w_{\mathrm{top}}^{\gamma,d}},
\end{align}
and thus gives the Miura transformation, 
$u^{\alpha} = w^{\alpha} - \d_x R^\alpha$.
This Miura transformation can be reverted, namely, there exist a unique $T^\alpha \in \mathcal{A}_u[[\epsilon]]_{\deg_{\d_x}=-1}$ such that 
\begin{align}
    w_{\mathrm{top}}^{\alpha} = u^{\alpha}_{\text{F-str}} + \d_x T^\alpha|_{u^{\gamma,d}=u^{\alpha,d}_{\text{F-str}}}. 
\end{align}
Then we can use the DR hierarchy of conservation laws to obtain an integrable system of conservation laws 
\begin{align} \label{eq: conservation laws type $1$}
    \partial_{t^{\beta,p}}w^{\alpha} = \d_x \Big( Q^{\alpha}_{\beta,p} + \sum\nolimits_{d\geq 0} \partial_{u^{\gamma,d}} T^\alpha \cdot \d_x^{d+1} Q^{\alpha}_{\beta,p} \Big)\Big|_{u^{\alpha,d} = w^{\alpha,d} - \d^{d+1}_x R^\alpha}.
\end{align}
It is integrable since by construction it is Miura equivalent to the DR hierarchy, which is integrable. Moreover, $w_{\mathrm{top}}^{\alpha}$ gives a solution since $u^{\alpha}_{\text{F-str}}$ gives a solution of the DR hierarchy. This proves the proposition for the integrable observables of type $1$. 

Now assume that $\{O_{g,n}\}$ are type $2$ integrable observables. Then the first summand in~\eqref{eq:X-alpha-beta-p-decomposition} is equal to $0$. This implies that
\begin{align}
    \d_{t^{\beta,p}} w^\alpha_{\mathrm{top}} = \d_x R^{\alpha}_{\beta,p}|_{w^{\gamma,d}=w_{\mathrm{top}}^{\gamma,d}},
\end{align}
hence the system of conservation laws 
\begin{align} \label{eq: conservation laws type $2$}
    \d_{t^{\beta,p}} w^\alpha = \d_x R^{\alpha}_{\beta,p}
\end{align}
is integrable, and $w^\alpha_{\mathrm{top}}$ gives a solution.

Moreover, in genus $0$, all integrable observable coincide, therefore they all provide the same integrable system in the limit $\epsilon=0$.
\end{proof}

\begin{remark} \label{rem: type $1$ II for FCohFT clarification}
Both type $1$ and type $2$ integrable observables provide an integrable hierarchy of conservation laws. The ones for type $1$~\eqref{eq: conservation laws type $1$} do not necessarily come with explicit formulae. However, they are related via the Miura transformation $u^{\alpha} = w^{\alpha} - \d_x R^\alpha$ to the DR hierarchy of conservation laws, with this Miura transformation being fully explicit (in one of the directions, in the other one has to formally invert it). On the other hand, the conservations laws for type $2$~\eqref{eq: conservation laws type $2$} are explicit, although their connection with the DR hierarchy is in principle lost. 

For the integrable observables $\{\bPsi_{g,n}\}$ and $\{\bOm_{g,n}\}$ we have both types of integrability simultaneously, since they are integrable observables of type 1 and type 2. For the integrable observables $\{\bA_{g,n}\}$ we have only type 1 instead, we employ the observation above to establish the link with DR in usual coordinates. It turns out that the $\bA$-hierarchy coincides with the DR hierarchy in normal coordinates, see Sec.~\ref{sec: identification A and DR}.
\end{remark}

Once we start with a P-CohFT, we can still consider the canonically associated F-CohFT and use Lemma~\ref{lemma: vector potential expansion} to establish integrability. However, in this case we have a more refined statement that gives a normal Miura equivalence with the DR hierarchy in normal coordinates and it allows to endow the integrable system with a Hamiltonian structure. 

\begin{lemma} \label{lemma: log partition expansion} Let ${{}^O\!F}$ be the logarithm of the partition function of a P-CohFT $\{ \mathfrak{pc}_{g,n}\}$. Let $w_{\mathrm{top}}^{\alpha} \coloneqq \eta^{\alpha\mu }\d_{t^{1,0}} \d_{t^{\mu,0}} {{}^O\!F}$ and $w_{\mathrm{top}}^{\alpha,d} \coloneqq \eta^{\alpha\mu } \d^{d+1}_{t^{1,0}} \d_{t^{\mu,0}} {{}^O\!F}$ for $d\geq 0$. Then there exist unique 
\begin{equation}
    R \in \mathcal{A}_w[[\epsilon]]_{\deg_{\d_x}=-2},\quad 
    R^\alpha \in \mathcal{A}_w[[\epsilon]]_{\deg_{\d_x}=-1},\quad
    R^{\alpha}_{\beta,p}\in \mathcal{A}_w[[\epsilon]]_{\deg_{\d_x}=0}
\end{equation} 
such that
\begin{align}
    \Coeff{\epsilon^{2g}\prod_{i=1}^n t^{\zeta_iq_i}} \bigg( {{}^O\!F} -R|_{w^{\gamma,d}=w_{\mathrm{top}}^{\gamma,d}} \bigg) & = 0 & \text{for}\ &  
    \sum\nolimits_{i=1}^n q_i \leq 2g-2,
    \\ \label{eq F-R alpha}
    \Coeff{\epsilon^{2g}\prod_{i=1}^n t^{\zeta_iq_i}} \bigg( \eta^{\alpha\mu } \d_{t^{\mu,0}}{{}^O\!F} -R^\alpha|_{w^{\gamma,d}=w_{\mathrm{top}}^{\gamma,d}} \bigg) & = 0 & \text{for}\ &  
    \sum\nolimits_{i=1}^n q_i \leq 2g-1,
    \\ \label{eq F-R alpha beta p}
    \Coeff{\epsilon^{2g}\prod_{i=1}^n t^{\zeta_iq_i}} \bigg( \eta^{\alpha\mu } \d_{t^{\beta,p}}\d_{t^{\mu,0}}{{}^O\!F} -R^\alpha_{\beta,p}|_{w^{\gamma,d}=w_{\mathrm{top}}^{\gamma,d}} \bigg) & = 0 & \text{for}\  &
    \sum\nolimits_{i=1}^n q_i \leq 2g,
\end{align}
and any $1 \leq \zeta_i \leq N$. More explicitly, we have the following:
\begin{align} \label{eq:F-decomposition}
    {{}^O\!F} &=  \sum_{g=0}^\infty\sum_{n=1}^\infty \frac{\epsilon^{2g}}{n!} \sum_{\substack{q_1,\dots,q_n\geq 0 \\ q_1+\cdots+q_n\geq 2g-1}} \prod_{i=1}^n t^{\gamma_i,q_i}  
%\times \\ \notag & \quad \qquad  
\int_{\oM_{g,n}} \Coeff{\prod_{i=1}^n a_i^{q_i}} {}^{O}\!B^{0}_{g,n} 
%\times \\ \notag & \qquad \qquad 
\mathfrak{pc}_{g,n}\Big( \Vec{\bigotimes}_{i=1}^n e_{\gamma_i} \Big) 
\\ \notag & \quad + \sum_{g=0}^\infty\sum_{n=1}^\infty \frac{\epsilon^{2g}}{n!} \sum_{\substack{q_1,\dots,q_n\geq 0 \\ q_1+\cdots+q_n=2g-2}} \prod_{i=1}^n w^{\gamma_i,q_i}_{\mathrm{top}}
%\times \\ \notag & \quad \qquad  
\int_{\oM_{g,n}} \Coeff{\prod_{i=1}^n a_i^{q_i}} {}^{O}\!B^{0}_{g,n} 
%\times \\ \notag & \qquad \qquad 
\mathfrak{pc}_{g,n}\Big( \Vec{\bigotimes}_{i=1}^n e_{\gamma_i} \Big),
\end{align}
where the second summand is equal to $R|_{w^{\gamma,d}=w_{\mathrm{top}}^{\gamma,d}}$,
\begin{align} \label{eq:F-alpha-decomposition}
    \eta^{\alpha\mu } \d_{t^{\mu,0}}{{}^O\!F} &=  \sum_{g=0}^\infty\sum_{n=1}^\infty \frac{\epsilon^{2g}}{n!} \sum_{\substack{q_1,\dots,q_n\geq 0 \\ q_1+\cdots+q_n\geq 2g}} \prod_{i=1}^n t^{\gamma_i,q_i}  
\times \\ \notag & \quad \qquad  
\int_{\oM_{g,n+1}} \Coeff{b_1^0\prod_{i=1}^n a_i^{q_i}} {}^{O}\!B^{1}_{g,n} 
%\times \\ \notag & \qquad \qquad 
\eta^{\alpha\mu}\mathfrak{pc}_{g,n+1}\Big( \Vec{\bigotimes}_{i=1}^n e_{\gamma_i} \otimes e_\mu \Big) 
\\ \notag & \quad + \sum_{g=0}^\infty\sum_{n=1}^\infty \frac{\epsilon^{2g}}{n!} \sum_{\substack{q_1,\dots,q_n\geq 0 \\ q_1+\cdots+q_n=2g-1}} \prod_{i=1}^n w^{\gamma_i,q_i}_{\mathrm{top}}
\times \\ \notag & \quad \qquad  
\int_{\oM_{g,n+1}} \Coeff{b_1^0\prod_{i=1}^n a_i^{q_i}} {}^{O}\!B^{1}_{g,n} 
%\times \\ \notag & \qquad \qquad 
\eta^{\alpha\mu}\mathfrak{pc}_{g,n+1}\Big( \Vec{\bigotimes}_{i=1}^n e_{\gamma_i} \otimes e_\mu \Big),
\end{align}
where the second summand is equal to $R^\alpha|_{w^{\gamma,d}=w_{\mathrm{top}}^{\gamma,d}}$, and
\begin{align} \label{eq:F-alpha-beta-p-decomposition} 
\eta^{\alpha\mu } \d_{t^{\beta,p}}\d_{t^{\mu,0}}{{}^O\!F}
& = \sum_{g=0}^\infty\sum_{n=1}^\infty \frac{\epsilon^{2g}}{n!} \sum_{\substack{q_1,\dots,q_n\geq 0 \\ q_1+\cdots+q_n\geq 2g+1}} \prod_{i=1}^n t^{\gamma_i,q_i}  
\times \\ \notag & \qquad  
\int_{\oM_{g,n+2}} \Coeff{b_1^p b_2^0\prod_{i=1}^n a_i^{q_i}} {}^{O}\!B^{2}_{g,n} 
%\times \\ \notag & \qquad \qquad 
\eta^{\alpha\mu}\mathfrak{pc}_{g,n+2}\Big( \Vec{\bigotimes}_{i=1}^n e_{\gamma_i} \otimes e_\beta\otimes e_\mu\Big) \bigg)
\\ \notag 
& \quad + \sum_{g=0}^\infty\sum_{n=1}^\infty \frac{\epsilon^{2g}}{n!} \sum_{\substack{q_1,\dots,q_n\geq 0 \\ q_1+\cdots+q_n= 2g}} \prod_{i=1}^n w_{\mathrm{top}}^{\gamma_i,q_i}  
\times \\ \notag & \quad \qquad  
\int_{\oM_{g,n+2}} \Coeff{b_1^p b_2^0\prod_{i=1}^n a_i^{q_i}} {}^{O}\!B^{2}_{g,n} 
%\times \\ \notag & \qquad \qquad 
\eta^{\alpha\mu}\mathfrak{pc}_{g,n+2}\Big( \Vec{\bigotimes}_{i=1}^n e_{\gamma_i} \otimes e_\beta\otimes e_\mu\Big) \bigg),
\end{align}
where the second summand is equal to $R^\alpha_{\beta,p}|_{w^{\gamma,d}=w_{\mathrm{top}}^{\gamma,d}}$.
\end{lemma}

\begin{proof}[Proof of Prop.~\ref{prop: integrability of integrable observables for P-CohFT}]
Note that Eqs.~\eqref{eq F-R alpha}, \eqref{eq F-R alpha beta p}, \eqref{eq:F-alpha-decomposition}, and~\eqref{eq:F-alpha-beta-p-decomposition} just repeat the corresponding equations in Lemma~\ref{lemma: vector potential expansion} applied to the F-CohFT canonically associated to $\{ \mathfrak{pc}_{g,n}\}$ by the correspondence given in~Eq.~\eqref{eq:From-PCohFT-to-FCohFT}. Thus, we have all statements immediately except for the normal Miura equivalence. 

To this end we employ Lemma~\ref{lemma: Type 0 relation}, applicable for all type $1$ integrable observables. It implies that the first summand in~\eqref{eq:F-decomposition} is equal to $F^{\DR}$. More precisely, we have
\begin{align}
    F^{\DR} = {{}^O\!F} - R|_{w^{\gamma,d}=w_{\mathrm{top}}^{\gamma,d}},
\end{align}
which implies 
\begin{align}
    \eta^{\alpha\mu}\d_{t^{1,0}}\d_{t^{\mu,0}} F^{\DR} = \eta^{\alpha\mu}\d_{t^{1,0}}\d_{t^{\mu,0}} {{}^O\!F} - \eta^{\alpha\mu}\d_{t^{1,0}}\d_{t^{\mu,0}} R|_{w^{\gamma,d}=w_{\mathrm{top}}^{\gamma,d}}.
\end{align}
This is already a normal Miura transformation, but its description can only be made explicit if there is an explicit formula for $\d_{t^{\mu,0}} w^{\zeta,0}$. As we have seen in the proof of Prop.~\ref{prop: integrability of integrable observables for F-CohFT}, cf. also Remark~\ref{rem: type $1$ II for FCohFT clarification}, the fluxes for this system are really explicit only if we also assume type $2$ integrability. Then we have:
\begin{align}
\eta^{\alpha\mu}\d_{t^{1,0}}\d_{t^{\mu,0}}F^{\DR} = \eta^{\alpha\mu}\d_{t^{1,0}}\d_{t^{\mu,0}} {{}^O\!F} - \eta^{\alpha\mu}\d_{x}\Big(\d_{w^{\zeta,k}} R \cdot \d_x^{d+1} R^\zeta_{\mu,0}\Big) |_{w^{\gamma,d}=w_{\mathrm{top}}^{\gamma,d}},    
\end{align}
or equivalently
\begin{align} \label{eq: Normal Miura}
    u^{\alpha}_{\text{P-str}} = w_{\mathrm{top}}^{\alpha}  - \eta^{\alpha\mu}\d_{x}\Big(\d_{w^{\zeta,k}} R \cdot \d_x^{d+1} R^\zeta_{\mu,0}\Big) |_{w^{\gamma,d}=w_{\mathrm{top}}^{\gamma,d}},    
\end{align}
and thus gives the desired normal Miura transformation explicitly.
\end{proof}

\begin{remark} \label{rem:TwoMiura} Let us stress that for a P-CohFT we derived two different Miura transformations that identify the corresponding integrable system with two different presentations of the DR hierarchy: in the natural and in the normal coordinates. The first Miura transformation can be written explicitly under the assumption of type $1$ integrability and is given by 
\begin{align}
\label{eq: change to usual DR coordinates}
u^\alpha = w^\alpha - \d_x R^\alpha,    
\end{align}
where $R^\alpha$ is determined by the second summand in \eqref{eq:F-alpha-decomposition}. The second Miura transformation is normal, and it can be written explicitly only under assumption that both type $1$ and type $2$ integrability hold. Then it is given by
\begin{align}
\label{eq: change to normal coordinates}
u^\alpha_{\mathrm{norm}} = w^\alpha - \eta^{\alpha\mu}\d_{x}\Big(\d_{w^{\zeta,k}} R \cdot \d_x^{d+1} R^\zeta_{\mu,0}\Big),    
\end{align}
where $R$ is determined by the second summand in \eqref{eq:F-decomposition} and $R^\zeta_{\mu,0}$ is determined by the second summand in \eqref{eq:F-alpha-beta-p-decomposition}.
\end{remark}

\subsection{Identification of the \texorpdfstring{$\bA$}{A}-hierarchies and DR hierarchies}
\label{sec: identification A and DR}
For P-CohFT and F-CohFT, the DR hierarchies  can be written in the natural coordinates (\ref{eq: DR conservation laws}) and normal coordinates \eqref{eq: normal coordinates for F CohFT}. We remarked, however, that for F-CohFT the normal coordinates are not normal \textit{per s\'e}, see Remark~\ref{rem: normal for F-CohFT}.

For P-CohFT, the integrable system associated to the integrable observable $\{\bA_{g,n} \}$ coincides with the DR hierarchy in normal coordinates, since ${{}^{\bA}\!F}=F^{\DR}$. For F-CohFT, we have the following observation:
\begin{proposition} \label{prop: A-DR normal F-CohFT} Let $\{\mathfrak{fc}_{g,n}\}$ be an F-CohFT. For the system of integrable observables $\{\bA_{g,n}\}$, the integrable system associated to $\{\mathfrak{fc}_{g,n}\}$ by Eq.~\eqref{eq: O-obs conservation laws F-CohFT} (with the fluxes given by Eq.~\eqref{eq: O-obs fluxed F-CohFT}) is obtained from the DR hierarchy given by Eq.~\eqref{eq: DR conservation laws} via the Miura transformation given by Eq.~\eqref{eq: normal coordinates for F CohFT}.
\end{proposition}

The proof of this proposition is given in Sec.~\ref{sub:sec:prop A-DR proof}. Thus, by slight abuse of terminology under the caveat above, one can say that the $\bA$-hierarchy for F-CohFTs (in normal coordinates) coincides with the DR hierarchy in normal coordinates. 

Now, the fact that the $\{\bA_{g,n} \}$-integrable system coincides with the DR hierarchy in normal coordinates, both in the case of F-CohFT and P-CohFT, has a nice consequence. Namely, since $\{\bA_{g,n} \}$ are integrable observables of type $1$, it allows to obtain a formula to invert the Miura transformation from the natural to normal coordinates. 

In the case of an F-CohFT, the expression of the normal coordinates in terms of the natural coordinates is given by Eq.~\eqref{eq: normal coordinates for F CohFT}, and specializing Eq.~\eqref{eq:Miura-O-F-CohFT} to the case $\{O_{g,n}=\bA_{g,n} \}$ and using Prop.~\ref{prop: A-DR normal F-CohFT}, we have:
\begin{align} \label{eq:Miura-A-F-CohFT}
    u^{\alpha} & = u^{\alpha}_{\mathrm{norm}} - \sum_{g=0}^\infty\sum_{n=1}^\infty \frac{\epsilon^{2g}}{n!} \sum_{\substack{q_1,\dots,q_n\geq 0 \\ q_1+\cdots+q_n=2g-1}} \prod_{i=1}^n u^{\gamma_i,q_i}_{\mathrm{norm}}
\times \\ \notag & \quad \qquad  
\int_{\oM_{g,n+1}} \bigg[\Coeff{b_1^0\prod_{i=1}^n a_i^{q_i}} {}^{\bA}\!B^{1}_{g,n} \bigg]^\circlearrowright
%\times \\ \notag & \qquad \qquad 
e^{\alpha}\bigg(\mathfrak{fc}_{g,n}\Big( \Vec{\bigotimes}_{i=1}^n e_{\gamma_i} \Big) \bigg).
\end{align}

\begin{corollary} For any F-CohFT $\{\mathfrak{fc}_{g,n}\}$ the Miura transformations~\eqref{eq: normal coordinates for F CohFT} and~\eqref{eq:Miura-A-F-CohFT} are inverse to each other. 
\end{corollary}

In the case of a P-CohFT, the expression of the normal coordinates in terms of the natural coordinates is given by a specialization of Eq.~\eqref{eq:definition of flux DR PCohFT}:  %

\begin{align} \label{eq:u-norm-u}
u^{\alpha}_{\mathrm{norm}} & = \sum_{\substack{g,n\ge 0\\ 2g+n>0}}\frac{\epsilon^{2g}}{n!} \sum_{\substack{d_1,\dots,d_n\geq 0 \\ d_1+\cdots+d_n=2g}} \prod_{i=1}^n u^{\gamma_i,d_i}  \times
\\ \notag & \qquad \int_{\oM_{g,n+2}} \Coeff{\prod_{i=1}^n a_i^{d_i}}\lambda_g \DR_g (-\suma, 0, a_1,\dots,a_n) 
\times 
\\ \notag & \qquad \qquad \eta^{\alpha\mu}
 \mathfrak{pc}_{g,n+2}\Big(e_1 \otimes e_\mu \otimes \Vec{\bigotimes}_{i=1}^n e_{\gamma_i}\Big).
\end{align}
%Its inverse exists, and 
Using Prop.~\ref{prop: integrability of integrable observables for F-CohFT} and~\ref{prop: integrability of integrable observables for P-CohFT}, more precisely, Remark~\ref{rem:TwoMiura}, we can explicitly write its inverse as
\begin{align} \label{eq:u-u-norm}
    u^{\alpha} & = u^{\alpha}_{\mathrm{norm}} - \d_x \sum_{g=0}^\infty\sum_{n=1}^\infty \frac{\epsilon^{2g}}{n!} \sum_{\substack{q_1,\dots,q_n\geq 0 \\ q_1+\cdots+q_n=2g-1}} \prod_{i=1}^n u^{\gamma_i,q_i}_{\mathrm{norm}}
\times \\ \notag & \quad \qquad  
\int_{\oM_{g,n+1}} \Coeff{b_1^0\prod_{i=1}^n a_i^{q_i}} {}^{\bA}\!B^{1}_{g,n} 
%\times \\ \notag & \qquad \qquad 
\eta^{\alpha\mu}\mathfrak{pc}_{g,n+1}\Big( \Vec{\bigotimes}_{i=1}^n e_{\gamma_i} \otimes e_\mu \Big).
\end{align}
%Indeed, since ${{}^{\bA}\!F}=F^{\DR}$, the integrable system associated to the integrable observable $\{\bA_{g,n} \}$ coincides with the DR hierarchy in normal coordinates. Therefore, the change of variable (\ref{eq: change to usual DR coordinates}) 
This equation is a specialization of (\ref{eq: change to usual DR coordinates}) to the case of the integrable observable $\{\bA_{g,n} \}$, thus it expresses the natural coordinates of DR in terms of the normal coordinates:

\begin{corollary} For any P-CohFT $\{\mathfrak{pc}_{g,n}\}$ the Miura transformations~\eqref{eq:u-norm-u} and~\eqref{eq:u-u-norm} are inverse to each other. 
\end{corollary}

Note also that the specializations of~\eqref{eq: normal coordinates for F CohFT} and~\eqref{eq:Miura-A-F-CohFT} for the F-CohFT canonically associated to a P-CohFT recover~\eqref{eq:u-norm-u} and~\eqref{eq:u-u-norm}.

\subsection{Diagrams of Miura transformations} 
Let us summarize the Miura transformations established above (those of Remark~\ref{rem:TwoMiura} and the transformations between the DR hierarchies in natural and normal coordinates) via the following diagrams.

For any P-CohFT and any integrable observable $\{O_{g,n}\}$, we have the following change of coordinates between the coordinates $w^{\alpha}$ of the hierarchy associated to $\{O_{g,n}\}$, the natural coordinates $u^{\alpha}$ of the DR hierarchy, and the normal coordinates $u^{\alpha}_{\mathrm{norm}}$ of the DR hierarchy:   

\begin{figure}[H]
\begin{center}
\begin{tikzpicture}[
    >={Stealth[length=2.5mm]},
    node distance=2.5cm,
    every node/.style={font=\normalsize}
]

% Titles
\node (titleDR) at (0, 2.4) {DR hierarchy};
\node (titleO)  at (7, 2.4) {$O$-hierarchy};

% Left side: DR-hierarchies box
\node (Uusual) at (0, 1.5) {$u^{\alpha}$};
\node (Unorma) at (0, -1.5) {$u^{\alpha}_{\mathrm{norm}}$};

% Dotted box around DR-hierarchies
\draw[dotted, thick] (-2, -2.1) rectangle (2, 2.1);

% Dotted box around O-hierarchies
\draw[dotted, thick] (5, -2.1) rectangle (9, 2.1);

% Curved arrows between Uusual and Unorma
\draw[->, thick] (Uusual) to[bend right=40] node[left,  midway] {(\ref{eq:u-norm-u})} (Unorma);
\draw[->, thick] (Unorma) to[bend right=40] node[right, midway] {(\ref{eq:u-u-norm})} (Uusual);

% Right side: O^a
 \node (Oa) at (7, 0) {$w^{\alpha}$};

% Arrows from O^a to the two U's
\draw[->, thick] (Oa) -- node[above, midway, sloped] {(\ref{eq: change to usual DR coordinates})} node[below, midway, sloped] {type 1} (Uusual);
\draw[->, thick] (Oa) -- 
node[above, midway, sloped]{(\ref{eq: change to normal coordinates}), normal}
node[below, midway, sloped] {type 1 and 2} (Unorma);

\end{tikzpicture}
\end{center}
\caption{Miura transformations for P-CohFTs}
\label{fig: Miura P-CohFT}
\end{figure}
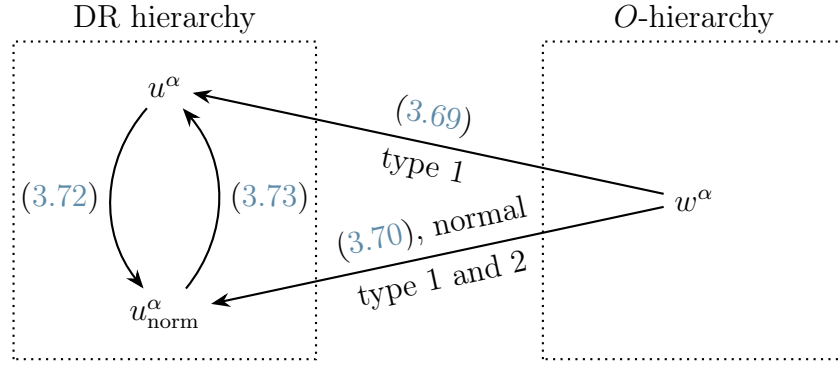

Each arrow represents a Miura transformation, and the arrow (\ref{eq: change to normal coordinates}) is furthermore a normal Miura transformation. Note that the explicit formula~\eqref{eq: change to normal coordinates} requires $\{O_{g,n}\}$ to be integrable observables of type 1 and type 2 simultaneously, while for~\eqref{eq: change to usual DR coordinates} type 1 is sufficient. 

In the F-CohFT case, the transformation indicated by the lower arrow does not hold any longer. The upper arrow is replaced by \eqref{eq:Miura-O-F-CohFT} and it requires the type 1 integrability. In the meanwhile, \eqref{eq:u-norm-u} and~\eqref{eq:u-u-norm} are replaced by~\eqref{eq: normal coordinates for F CohFT} and~\eqref{eq:Miura-A-F-CohFT}, respectively. We get:
% only the top arrow (with label (\ref{eq: change to usual DR coordinates})) survives --- it is replaced by \eqref{eq:Miura-O-F-CohFT} and it requires the type 1 integrability. In the meanwhile, \eqref{eq:u-norm-u} and~\eqref{eq:u-u-norm} are replaced by~\eqref{eq: normal coordinates for F CohFT} and~\eqref{eq:Miura-A-F-CohFT}, respectively.
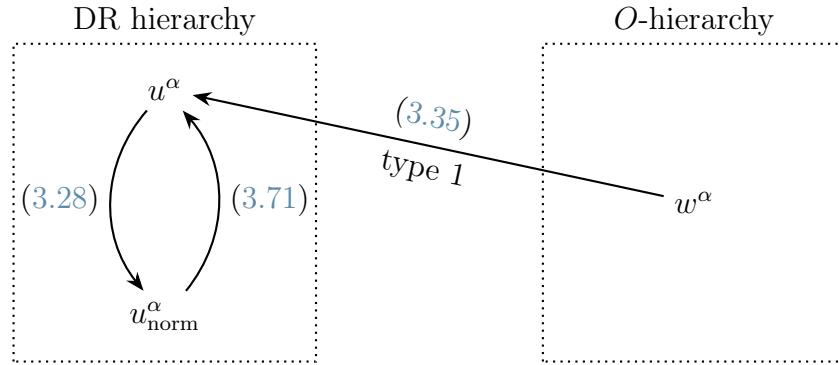
\begin{figure}[H]
\begin{center}
\begin{tikzpicture}[
    >={Stealth[length=2.5mm]},
    node distance=2.5cm,
    every node/.style={font=\normalsize}
]

% Titles
\node (titleDR) at (0, 2.4) {DR hierarchy};
\node (titleO)  at (7, 2.4) {$O$-hierarchy};

% Left side: DR-hierarchies box
\node (Uusual) at (0, 1.5) {$u^{\alpha}$};
\node (Unorma) at (0, -1.5) {$u^{\alpha}_{\mathrm{norm}}$};

% Dotted box around DR-hierarchies
\draw[dotted, thick] (-2, -2.1) rectangle (2, 2.1);

% Dotted box around O-hierarchies
\draw[dotted, thick] (5, -2.1) rectangle (9, 2.1);

% Curved arrows between Uusual and Unorma
\draw[->, thick] (Uusual) to[bend right=40] node[left,  midway] {(\ref{eq: normal coordinates for F CohFT})} (Unorma);
\draw[->, thick] (Unorma) to[bend right=40] node[right, midway] {(\ref{eq:Miura-A-F-CohFT})} (Uusual);

% Right side: O^a
 \node (Oa) at (7, 0) {$w^{\alpha}$};

% Arrows from O^a to the two U's
\draw[->, thick] (Oa) -- node[above, midway, sloped] {(\ref{eq:Miura-O-F-CohFT})}
node[below, midway, sloped] {type 1}(Uusual);
%\draw[->, thick] (Oa) -- node[below, midway, sloped] {(\ref{eq: change to normal coordinates})} (Unorma);

\end{tikzpicture}
\end{center}
\caption{Miura transformations for F-CohFTs}
\label{fig: F-CohFT}
\end{figure}

Finally, in the exceptional case of $\{O_{g,n} = \bA_{g,n}\}$, the integrable observables are only of type $1$, but we have identified the $\bA$-hierarchy with the DR hierarchy in the normal coordinates, so, in the P-CohFT case (Figure~\ref{fig: Miura P-CohFT}), \eqref{eq: change to normal coordinates} is trivial and \eqref{eq: change to usual DR coordinates} coincides with \eqref{eq:u-u-norm}. In the F-CohFT case (Figure~\ref{fig: F-CohFT}), (\ref{eq:Miura-O-F-CohFT}) coincides with (\ref{eq:Miura-A-F-CohFT}).

To complete the picture, we collect in the following table the logarithm of tau functions and components of vector potentials controlling the integrable system in each coordinate: 
% \begin{center}
% \begin{tabular}{|c|c|c|}
% \hline 
% Coordinate & Tau function & Vector potential\tabularnewline
% \hline 
% \hline 
% $u^{\alpha}$ & -- & $X^{{\rm DR},\alpha}$\tabularnewline
% \hline 
% $u^{\alpha}_{{\rm norm}}$ & $F^{{\rm DR}}=\,^{\bA}\!F$ & $\,^{\bA}\!X^{\alpha}$\tabularnewline
% \hline 
% $w^{\alpha}$ & $\,^{O}\!F$ & $\,^{O}\!X$\tabularnewline
% \hline 
% \end{tabular}
% \par\end{center}
\begin{center}
\begin{tabular}{lcc}
\toprule
Coordinate & Logarithm of the tau function & Components of the vector potential \\
\midrule
$u^{\alpha}$                 & ---                              & $X^{{\rm DR},\alpha}$ \\
$u^{\alpha}_{{\rm norm}}$    & $F^{{\rm DR}} = \,^{\bA}\!F$    & $\,^{\bA}\!X^{\alpha}$ \\
$w^{\alpha}$                 & $\,^{O}\!F$                     & $\,^{O}\!X^{\alpha}$ \\
\bottomrule
\end{tabular}
\end{center}
Recall that tau functions are defined for P-CohFTs only, whereas vector potentials are defined for P-CohFTs and F-CohFTs. 

\subsection{Interrelations between the obtained hierarchies}
\label{sec: interrelations}
As we discussed above, for the integrable observables $\{\bPsi_{g,n}\}$, the integrable systems associated to F-CohFTs and P-CohFTs by Prop.~\ref{prop: integrability of integrable observables for F-CohFT} and~\ref{prop: integrability of integrable observables for P-CohFT} are called the Dubrovin-Zhang hierarchies. 
For the integrable observables $\{\bA_{g,n}\}$, we get the DR hierarchies in normal coordinates. 
% of conservations laws both for F-CohFTs~\cite{BR-spin} and the Hamiltonian DR hierarchies for P-CohFTs~\cite{Bur,BR-spin} that can also be considered in the normal coordinates~\cite{BDGR1}.

An immediate corollary of  Prop.~\ref{prop: integrability of integrable observables for F-CohFT} and~\ref{prop: integrability of integrable observables for P-CohFT} is the following:

\begin{corollary} Consider the Dubrovin-Zhang, the Double Ramification, and $\bOm$- hierarchies.
\begin{enumerate}
    \item For any P-CohFT or F-CohFT, they have the same dispersionless limit.
    \item For any F-CohFT, the three systems of conservation laws are connected by Miura transformations of second kind.
    \item For any P-CohFT, the three integrable hierarchies are tau symmetric, Hamiltonian, and all normal Miura equivalent to each other.
\end{enumerate}
\end{corollary}

\begin{remark} It is important to stress that we don't have a construction of a direct (normal) Miura transformation between the Dubrovin-Zhang and the $\bOm$-hierarchies. Their (normal) Miura equivalence is a consequence of their relations with the DR hierarchy.
\end{remark}

Moreover, let us state the following direct corollary of~\cite[Prop. 5.3]{BLS-Quantum}:

\begin{proposition} \label{prop:coincide}
(1) If $\deg_{R^\bullet} \mathfrak{fc}_{g,n} \leq g-2+n$ for all $g$ and $n$, then the $\bOm$-hierarchy coincides with the DR hierarchy of conservation laws. 

(2) If $\deg_{R^\bullet} \mathfrak{pc}_{g,n} \leq g-2+n$ for all $g$ and $n$, then the $\bOm$-hierarchy coincides with the Hamiltonian DR hierarchy in normal coordinates. 
\end{proposition}

\begin{proof} The second statement here is implied by \cite[Thm.~2]{BLS-Quantum} that can be interpreted in our terms as the statement that ${}^\bA\!F = {}^\bOm\!F$ in the case $\deg_{R^\bullet} \mathfrak{pc}_{g,n} \leq g-2+n$.

The first statement is also available in \cite{BLS-Quantum} as a part of the proof of \cite[Thm. 2]{BLS-Quantum}, namely, it is the way \cite[Prop. 5.3]{BLS-Quantum} is used in the proof of Cor.~5.5 in \emph{op. cit.} In our terms here it can be interpreted as the statement that ${}^\bA\!X = {}^\bOm\!X$ in the case $\deg_{R^\bullet} \mathfrak{fc}_{g,n} \leq g-2+n$ for all $g$ and $n$. 
\end{proof}

\subsection{Further remarks} \label{sec: application mLRT relations} In this section we briefly discuss, following~\cite[Sec.~4]{BS22}, implications of LRT-$m$ relations for $m\geq 2$ that we did not include in the definition of integrable observables. 
Let us do it only for F-CohFTs, as in \emph{op.~cit.}, since for P-CohFTs the corresponding statement is straightforward via Eq.~\eqref{eq:From-PCohFT-to-FCohFT} 

\begin{lemma} \label{lemma: full system of relations for FCohFT} (1) Let ${{}^O\!X}$ be the vector potential of an F-CohFT $\{ \mathfrak{fc}_{g,n}\}$. Let $w_{\mathrm{top}}^{\alpha} \coloneqq \d_{t^{1,0}} {{}^O\!X}^{\alpha}$ and $w_{\mathrm{top}}^{\alpha,d} \coloneqq \d^{d+1}_{t^{1,0}} {{}^O\!X}^{\alpha}$ for $d\geq 0$. For $k\geq 1$ define
\begin{align}
{{}^O\!X}^{\alpha,p}_{(\beta_1,q_1),\dots,(\beta_k,q_k)} & \coloneqq 
\prod_{i=1}^k \d_{t^{\beta_i,q_i}} \sum_{g,n} \frac{\epsilon^{2g}}{n!} \sum_{d_1,\dots,d_n=0}^\infty \prod_{i=1}^n t^{\alpha_i,d_i} \times \\ \notag & \int_{\oM_{g,1+n}} e^{\alpha} \bigg(\mathfrak{fc}\Big( \Vec{\bigotimes}_{i=1}^n e_{\alpha_i}\Big) \bigg) 
\Coeff{a_0^p\prod_{i=1}^n a_i^{d_i}} O_{g,n+1}(a_0,a_1,\dots,a_n) 
\end{align}
Then there exist and unique $R^{(\alpha,p)}_{(\beta_1,q_1),\dots,(\beta_k,q_k)} \in \mathcal{A}_w[[\epsilon]]_{\deg_{\d_x}=k-1}$ such that 
\begin{align}
    \Coeff{\epsilon^{2g}\prod_{i=1}^n t^{\zeta_iq_i}} \bigg( {{}^O\!X}^{\alpha,p}_{(\beta_1,q_1),\dots,(\beta_k,q_k)} -R^{(\alpha,p)}_{(\beta_1,q_1),\dots,(\beta_k,q_k)}|_{w^{\gamma,d}=w_{\mathrm{top}}^{\gamma,d}} \bigg) & = 0
\end{align}
for $\sum\nolimits_{i=1}^n q_i \leq 2g-1+k$. 
More explicitly, we have the following:
\begin{align} \label{eq:X-alpha-beta-p-ktimes-decomposition} 
& {{}^O\!X}^{\alpha,p}_{(\beta_1,q_1),\dots,(\beta_k,q_k)} 
 = 
\\ \notag & \quad \sum_{g=0}^\infty\sum_{n=1}^\infty \frac{\epsilon^{2g}}{n!} \sum_{\substack{q_1,\dots,q_n\geq 0 \\ q_1+\cdots+q_n\geq 2g+k}} \prod_{i=1}^n t^{\gamma_i,q_i}  
\times \\ \notag & \qquad  
\int_{\oM_{g,n+k+1}} \bigg[\Coeff{\prod_{i=1}^k b_i^{q_i} b_{k+1}^p\prod_{i=1}^n a_i^{q_i}} {}^{O}\!B^{k+1}_{g,n} \bigg]^\circlearrowright
%\times \\ \notag & \qquad \qquad 
e^{\alpha}\bigg(\mathfrak{fc}_{g,n+1}\Big( \Vec{\bigotimes}_{i=1}^n e_{\gamma_i} \otimes \Vec{\bigotimes}_{i=1}^k e_{\beta_i}\Big) \bigg)
\\ \notag 
& \quad + \sum_{g=0}^\infty\sum_{n=1}^\infty \frac{\epsilon^{2g}}{n!} \sum_{\substack{q_1,\dots,q_n\geq 0 \\ q_1+\cdots+q_n= 2g-1+k}} \prod_{i=1}^n w_{\mathrm{top}}^{\gamma_i,q_i}  
\times \\ \notag & \quad \qquad  
\int_{\oM_{g,n+k+1}} \bigg[\Coeff{\prod_{i=1}^k b_i^{q_i} b_{k+1}^p\prod_{i=1}^n a_i^{q_i}} {}^{O}\!B^{k+1}_{g,n} \bigg]^\circlearrowright
%\times \\ \notag & \qquad \qquad 
e^{\alpha}\bigg(\mathfrak{fc}_{g,n+1}\Big( \Vec{\bigotimes}_{i=1}^n e_{\gamma_i} \otimes \Vec{\bigotimes}_{i=1}^k e_{\beta_i}\Big) \bigg),
\end{align}
where the second summand is equal to $R^{(\alpha,p)}_{(\beta_1,q_1),\dots,(\beta_k,q_k)}|_{w^{\gamma,d}=w_{\mathrm{top}}^{\gamma,d}}$. 
\end{lemma}

If $\{O_{g,n}\}$ satisfies LRT-$(k+1)$, the first summand in \eqref{eq:X-alpha-beta-p-ktimes-decomposition} vanishes and we obtain the following formula.
\begin{corollary} For any $k\geq 1$, if $\{O_{g,n}\}$ satisfies LRT-$(k+1)$ then 
\begin{align} 
     {{}^O\!X}^{\alpha,p}_{(\beta_1,q_1),\dots,(\beta_k,q_k)} = R^{(\alpha,p)}_{(\beta_1,q_1),\dots,(\beta_k,q_k)}|_{w^{\gamma,d}=w_{\mathrm{top}}^{\gamma,d}}
\end{align}
in the notation of Lemma~\ref{lemma: full system of relations for FCohFT}. 
\end{corollary}

Note that once it is established that ${{}^O\!X}^{\alpha,p}_{(\beta_1,q_1)}$ is equal to a differential polynomial in $\mathcal{A}_w[[\epsilon]]_{\deg_{\d_x}=k-1}$ upon substitution $w^{\gamma,d}=w_{\mathrm{top}}^{\gamma,d}$, one can use the integrable system to prove analogous statement for any $k\geq 1$. This indicates that there should be a way to derive LRT-$m$ for all $m\geq 2$ just from LRT-$2$. As we mentioned in Remark~\ref{rem:LiuWang}, it is indeed the core of the argument proposed in~\cite{LiuWang}.

%%%%%%%%%%%%%%%%%%%%%%%%%%%%%%%%%%%%%%%%%%%%%%%%%%%%%%%%%%%%%%%%%%%%%%%%%
\vspace{1cm}
\section{Proofs of three examples of integrable observables}
\label{sec:proofs}
%%%%%%%%%%%%%%%%%%%%%%%%%%%%%%%%%%%%%%%%%%%%%%%%%%%%%%%%%%%%%%%%%%%%%%%%%

\subsection{Properties of \texorpdfstring{$\bPsi_{g,n}$}{Psign}} In this section we prove Theorem~\ref{thm: Psi}. This theorem is not new, all required properties are either obvious or proved elsewhere in the literature, so we just give a guide to the references.  

The homogeneity property of $\bPsi_{g,n}$ is immediate, and its the push-forward property follow from standard properties of the $\psi$-classes, see~\cite[Eqs.~(2.36) and (2.44)]{Witten-KdV}. The genus $0$ normalization holds by definition. 

The leveled rooted tree relations for $\bPsi$
%, known as the DR--DZ relations, 
were originally conjectured in a bit different form in~\cite{BDGR1,BDGR20} for $m=0$, then revisited and improved in \cite{BGR19} (still for $m=0$, though $m=1$ is somehow present there as well), further generalized as conjectures for any $m$ in \cite{BS22}, improved in~\cite{BLRS}, further improved and reduced to the master relations in~\cite{BLS-DRDZ}, and finally proved in \cite{BSS25}.

\subsection{Properties of \texorpdfstring{$\bA_{g,n}$}{Agn}} The goal of this section is to present a proof of Theorem~\ref{thm: A}.

\subsubsection{Polynomiality, homogeneity, and genus \texorpdfstring{$0$}{0} properties of \texorpdfstring{$\bA_{g,n}$}{Agn}}

\label{subsec: polynomiality genus 0 A}

The class 
\begin{align}
s_{*}\left(t^{*}\left(\psi_{0}^{d}\right)\left[\overline{\mathcal{M}}{}_{g}^{\sim}\left(\mathbb{P}^{1},a_{1},\dots,a_{n},-\suma\right)\right]^{\mathrm{vir}}\right)
\end{align} 
is a polynomial in the variables $a_{1},\dots,a_{n}$ of degree $2g+d$ as shown in \cite[Proposition 2.3]{BLS-Quantum}. It is also apparent from the proof that when restricted to the locus of compact type curves, this becomes a homogeneous polynomial of degree $2g+d$ (since the DR-cycle $s_{*}(\left[\overline{\mathcal{M}}{}_{g}^{\sim}\left(\mathbb{P}^{1},a_{1},\dots,a_{n},-\suma\right)\right]^{\mathrm{vir}})$ is homogenous polynomial of degree $2g$ on the compact type locus \cite{Hain}). Since the virtual dimension of the moduli space of rubber maps is $2g-2+n$ and $\lambda_{g}$ vanishes on the complement of the compact type locus, the class 
\begin{align}
	\lambda_{g}s_{*}\left(t^{*}\left(\psi_{0}^{d}\right)\left[\overline{\mathcal{M}}{}_{g}^{\sim}\left(\mathbb{P}^{1},a_{1},\dots,a_{n},-\suma\right)\right]^{\mathrm{vir}}\right)
\end{align}
that we denoted above by $\bA^1_{g,n}(a_1,\dots,a_n)$ (see Eq.~\eqref{eq: A-1})
is of Chow degree $2g+d$ and a homogeneous polynomial of degree $2g+d$. Finally, its push-forward by $\pi$ is divisible by $\suma$ as shown in \cite[Proposition 2.3]{BLS-Quantum}. This proves the polynomiality and homogeneity properties of $\bA_{g,n}$.

The genus $0$ property $\bA_{0,n}(a_1,\dots,a_n) = \bPsi_{0,n}(a_1,\dots,a_n)$ is an instance of the so-called $A=B$ relation originally conjectured in~\cite[Conj.~2.5]{BGR19} for any genus and proved in genus $0$ in Prop.~4.4 in \emph{op. cit}. To recognize this, one has to use the interpretation of the $A$-side of this conjecture given in~\cite[Lem.~2.7]{BLS-DRDZ}, the interpretation of the $B$-side of the conjecture given in~\cite[Thm.~2.15]{BLRS}, while for the genus $0$ case there are different proofs available in~\cite[Prop.~4.4]{BGR19} and ~\cite[Thm.~2.3]{BS22}, or one can refer to the genus $0$ specialization of~\cite[Thm.~2]{BLS-DRDZ} or~\cite[Sec.~2.2]{BSS25}.

\subsubsection{Push-forward property of \texorpdfstring{$\bA_{g,n}$}{Agn}}
\label{subsubsec: push-forward A}

It is proved in~\cite[Prop.~4.2 and~4.5]{BGR19}, once these statements are combined with the conceptual explanation of the nature of $\bA_{g,n}$ class give in~\cite[Lemma 2.6]{BLS-DRDZ}. However, the revisited definition of the $\bA_{g,n}$ obtained in~\cite{BLS-DRDZ} that we use here prompts a new proof, which we present below. 

First, consider the class $\bA^1_{g,n}(a_1,\dots,a_n,b)$.
For this class we prove the following lemma:
\begin{lemma} \label{lem:push-forward-A1} Let $\pi'\colon \oM_{g,n+2}\to \oM_{g,n+1}$ be the map that forgets the $(n+1)$-st marked point. In $R^*(\oM_{g,n+1})[a_1,\dots,a_n,b]$ we have:
\begin{align} \label{eq:statment-lemma-push-A}
	\pi'_* \bA^1_{g,n+1}(a_1,\dots,a_n,b) = (\suma + (2g-1+n) b) \bA^1_{g,n}(a_1,\dots,a_n) + O(b^2).
\end{align}
\end{lemma}

\begin{remark}\label{rem:push-forward-A-no-lambda} The proof that we give below is an improved paraphrase of the arguments given in~\cite[Sec.~6]{BLS-Quantum}. Note that the argument and the statement work also for the class 
	\begin{align}
		s_* \left(\frac{1}{1-t^*\psi_0} \left[\oM_{g}^{\sim}\left(\mathbb{P}^1,a_1,\dots,a_n,b,-\suma-b \right)\right]^{\mathrm{vir}}\right),
	\end{align}
but we don't need this generality here.
\end{remark}

An immediate corollary of this lemma is the following property of $\bA_{g,n}$ that is a part of the statement of Theorem~\ref{thm: A}:
\begin{corollary}\label{cor:push-forward-A} Let $\pi\colon \oM_{g,n+1}\to \oM_{g,n}$ be the map that forgets the last marked point. In $R^*(\oM_{g,n})[a_1,\dots,a_n,b]$ we have:
	\begin{align}
		\pi_* \bA_{g,n+1}(a_1,\dots,a_n,b) = (\suma + (2g-2+n) b) \bA_{g,n}(a_1,\dots,a_n) + O(b^2).
	\end{align}
\end{corollary}

Let us first derive Corollary~\ref{cor:push-forward-A} from Lemma~\ref{lem:push-forward-A1}, and then prove the latter lemma. 

\begin{proof}[Proof of Corollary~\ref{cor:push-forward-A}] Slightly abusing notation, let $\pi,\pi'\colon \oM_{g,n+2}\to \oM_{g,n+1}$ be the maps that forget the last and the next to the last marked points, respectively, and $\pi\colon \oM_{g,n+1}\to \oM_{g,n}$ be the map that forgets the last marked point. 

Recall that $(\suma+b)\bA_{g,n+1}(a_1,\dots,a_n,b) = \pi_* \bA^1_{g,n+1}(a_1,\dots,a_n,b)$. Moreover,
\begin{align}
	& (\suma+b)\pi_*\bA_{g,n+1}(a_1,\dots,a_n,b) = \pi_*\pi_*\bA^1_{g,n+1}(a_1,\dots,a_n,b) 
	\\ \notag 
	& \qquad \qquad= \pi_*\pi'_*\bA^1_{g,n+1}(a_1,\dots,a_n,b) 
	\\ \notag 
	& \qquad \qquad= (\suma + (2g-1+n) b) \pi_* \bA^1_{g,n}(a_1,\dots,a_n) + O(b^2)
	\\ \notag 
	& \qquad \qquad = \suma (\suma + (2g-1+n) b) \bA_{g,n}(a_1,\dots,a_n) + O(b^2).
\end{align}
This is an equality of two polynomials in $a_1,\dots,a_n,b$, and the right hand side is divisible by $(\suma+b)$. We can instead consider it as an equality of two formal power series in $b$ with $a_1,\dots,a_n$ being fixed positive integers.  Then we can use that 
\begin{align}
	\frac{\suma (\suma + (2g-1+n) b)}{\suma+b} & = (\suma + (2g-1+n) b)\Big(1-\frac b\suma + O(b^2)\Big) \\ \notag 
	& = \suma + (2g-2+n) b + O(b^2),
\end{align}
which implies that 
\begin{align}
 \pi_*\bA_{g,n+1}(a_1,\dots,a_n,b) =  (\suma + (2g-2+n) b) \bA_{g,n}(a_1,\dots,a_n) + O(b^2)	
\end{align}
in this realm (as formal power series in $b$ with $a_1,\dots,a_n$ being fixed positive integers), and using the polynomiality property on both sides of this equation establishes the equality of the respective polynomials. 
\end{proof}

\begin{proof}[Proof of Lemma~\ref{lem:push-forward-A1}] As we mentioned in Remark~\ref{rem:push-forward-A-no-lambda}, we proceed with the argument that does not use the $\lambda_g$ factor, that is, we in fact prove that 
\begin{align}
	& \pi'_* s_* \left(\frac{1}{1-t^*\psi_0} \left[\oM_{g}^{\sim}\left(\mathbb{P}^1,a_1,\dots,a_n,b,-\suma-b \right)\right]^{\mathrm{vir}}\right)
	\\ \notag & 
	= (\suma + (2g-1+n)b) \, s_* \left(\frac{1}{1-t^*\psi_0} \left[\oM_{g}^{\sim}\left(\mathbb{P}^1,a_1,\dots,a_n,-\suma \right)\right]^{\mathrm{vir}}\right) + O(b^2).
\end{align}

We proceed as in~\cite[Sec.~6.2]{BLS-Quantum}. We use that
\begin{align} \label{eq: split 1 / 1-psi}
	\frac{1}{1-t^*\psi_0} = 1+ \frac{b\cdot s^* \psi_{n+1}}{1-t^*\psi_0}  + \frac{D_{n+1}}{1-t^*\psi_0},
\end{align}
where $D_{n+1}$ is divisor of maps to the degenerate target, where the $(n+1)$-st point lies on an unstable component of the source curve that is contracted by the map $s$, cf.~\cite{BSSZ}. Then we have:
\begin{align} \label{eq:lemma-push-A-1summand}
	& \pi'_* s_* \left( \left[\oM_{g}^{\sim}\left(\mathbb{P}^1,a_1,\dots,a_n,b,-\suma-b \right)\right]^{\mathrm{vir}}\right)
	= O(b^2)
\end{align}
by~\cite[Prop.~2.12]{BLS-Quantum}. 

Moreover, 
\begin{align} \label{eq:lemma-push-A-2summand}
	& \pi'_* s_* \left(\frac{b\cdot s^* \psi_{n+1}}{1-t^*\psi_0} \left[\oM_{g}^{\sim}\left(\mathbb{P}^1,a_1,\dots,a_n,b,-\suma-b \right)\right]^{\mathrm{vir}}\right)
	\\ \notag & 
	\overset{(1)}{=} b\cdot \pi'_* s_* \left(\frac{s^* \psi_{n+1}}{1-t^*\psi_0} \left[\oM_{g}^{\sim}\left(\mathbb{P}^1,a_1,\dots,a_n,b,-\suma-b \right)\right]^{\mathrm{vir}}\right)\bigg|_{b=0}  + O(b^2)
	\\ \notag & 
	\overset{(2)}{=} b\cdot \pi'_* \left(\psi_{n+1} s_* \left(\frac{1}{1-t^*\psi_0} \left[\oM_{g}^{\sim}\left(\mathbb{P}^1,a_1,\dots,a_n,0,-\suma \right)\right]^{\mathrm{vir}}\right)\right)  + O(b^2)
	\\ \notag & 
	\overset{(3)}{=} 
	%b\cdot \pi'_* \left(\psi_{n+1} (\pi')^* s_* \left(\frac{1}{1-t^*\psi_0} \left[\oM_{g}^{\sim}\left(\mathbb{P}^1,a_1,\dots,a_n,-\suma \right)\right]^{\mathrm{vir}}\right)\right)  + O(b^2)
	%	\\ \notag & 
	%\overset{(4)}{=} 
	(2g-1+n)b\cdot s_* \left(\frac{1}{1-t^*\psi_0} \left[\oM_{g}^{\sim}\left(\mathbb{P}^1,a_1,\dots,a_n,-\suma \right)\right]^{\mathrm{vir}}\right)  + O(b^2)\,.
\end{align}
Here equality (1) uses the polynomiality of the class~\cite[Prop.~2.3]{BLS-Quantum}, equality (2) combines polynomiality of the class and the projection formula, and equality (3) is proved in the following way. We notice that 
\begin{align} \label{eq:push-A-replacement-pullback}
& \pi'_* \left(\psi_{n+1} s_* \left(\frac{1}{1-t^*\psi_0} \left[\oM_{g}^{\sim}\left(\mathbb{P}^1,a_1,\dots,a_n,0,-\suma \right)\right]^{\mathrm{vir}}\right)\right)
\\ \notag &
 =  \pi'_* \left(\psi_{n+1} s_* \left( \left[\oM_{g}^{\sim}\left(\mathbb{P}^1,a_1,\dots,a_n,0,-\suma \right)\right]^{\mathrm{vir}}\right) \right)
\\ \notag & \qquad  
+  \pi'_* \left(\psi_{n+1}  s_* \left(\frac{t^*\psi_0}{1-t^*\psi_0} \left[\oM_{g}^{\sim}\left(\mathbb{P}^1,a_1,\dots,a_n,0,-\suma \right)\right]^{\mathrm{vir}}\right) \right)
\\ \notag &
=  \pi'_* \left(\psi_{n+1} s_* \left( \left[\oM_{g}^{\sim}\left(\mathbb{P}^1,a_1,\dots,a_n,0,-\suma \right)\right]^{\mathrm{vir}}\right) \right)
\\ \notag 
& \qquad +
\sum_{k=1}^\infty \frac 1{k!} \sum_{\substack{g_1+g_2+k-1 = g\\ 
		I_1\sqcup I_2 = \{1,\dots,n\} \\ (k,g_2,I_2)\not=(1,0,\emptyset)}}
	\sum_{\substack{a'_1,\dots a'_k \in \Z_{\geq 1} \\ \sum_{i=1}^k a'_i = \sum\nolimits_{j\in I_1}a_j}} \prod_{i=1}^k a'_i \cdot 
\\ \notag &\qquad \qquad \pi'_* \rho_* \bigg( s_* \left(\frac{{1}}{1-t^*\psi_0} \left[\oM_{g_1}^{\sim}\left(\mathbb{P}^1,\{a_i\}_{i\in I_1}, \{-a'_j\}_{j=1}^k \right)\right]^{\mathrm{vir}}\right) \otimes 
\\ \notag & \qquad \qquad \qquad \qquad 
\psi_{|I_2|+2} s_*\left( \left[\oM_{g_2}^{\sim}\left(\mathbb{P}^1,\{a'_j\}_{j=1}^k,\{a_i\}_{i\in I_2},0,-\suma \right)\right]^{\mathrm{vir}}\right) \bigg),
\end{align}
Here in the second sum we use evaluation of $t^*\psi_0$ in the numerator according to~\cite[Eq.~(1)]{BSSZ}. The sum runs over all irreducible strata of two-component curves connected by $k\geq 1$ nodes, and we assume with the points labelled by $I_1$ on the component of genus $g_1$ and the points labelled by $I_2\cup \{n+1,n+2\}$ are on the component of genus $g_2$, and $\rho$ denotes the corresponding boundary morphism. The case $(g_2,I_2)=(0,\emptyset)$ is excluded since it vanishes after multiplication by the $\psi_{n+1}$. Note that 
\begin{align}
s_* \left( \left[\oM_{g}^{\sim}\left(\mathbb{P}^1,a_1,\dots,a_n,0,-\suma \right)\right]^{\mathrm{vir}}\right) = (\pi')^*s_* \left( \left[\oM_{g}^{\sim}\left(\mathbb{P}^1,a_1,\dots,a_n,-\suma \right)\right]^{\mathrm{vir}}\right)
\end{align}
Applying this equality and subsequently the projection formula, we see that the right hand side of~\eqref{eq:push-A-replacement-pullback} can be rewritten as
\begin{align} \label{eq:push-A-replacement-pullback-2}
\eqref{eq:push-A-replacement-pullback} & =  \pi'_* \left(\psi_{n+1} (\pi')^* s_* \left( \left[\oM_{g}^{\sim}\left(\mathbb{P}^1,a_1,\dots,a_n,-\suma \right)\right]^{\mathrm{vir}}\right) \right)
\\ \notag 
&\quad  +
\sum_{k=1}^\infty \frac 1{k!} \sum_{\substack{g_1+g_2+k-1 = g\\ 
		I_1\sqcup I_2 = \{1,\dots,n\} \\ (k,g_2,I_2)\not=(1,0,\emptyset)}}
\sum_{\substack{a'_1,\dots a'_k \in \Z_{\geq 1} \\ \sum_{i=1}^k a'_i = \sum\nolimits_{j\in I_1}a_j}} \prod_{i=1}^k a'_i \cdot 
\\ \notag & \qquad \rho_* \bigg( s_* \left(\frac{{1}}{1-t^*\psi_0} \left[\oM_{g_1}^{\sim}\left(\mathbb{P}^1,\{a_i\}_{i\in I_1},\{-a'_j\}_{j=1}^k \right)\right]^{\mathrm{vir}}\right) \otimes 
\\ \notag & \qquad \qquad  
\pi'_* \left(\psi_{|I_2|+2} (\pi')^* s_*\left( \left[\oM_{g_2}^{\sim}\left(\mathbb{P}^1,\{a'_j\}_{j=1}^k,\{a_i\}_{i\in I_2},-\suma \right)\right]^{\mathrm{vir}}\right) \right) \bigg)
%\\ \notag & 
\end{align}
\begin{align}
& =
(2g-1+n) \cdot s_* \left( \left[\oM_{g}^{\sim}\left(\mathbb{P}^1,a_1,\dots,a_n,-\suma \right)\right]^{\mathrm{vir}}\right)
\\ \notag 
&\quad  +
\sum_{k=1}^\infty \frac 1{k!} \sum_{\substack{g_1+g_2+k-1 = g\\ 
		I_1\sqcup I_2 = \{1,\dots,n\} \\ (k,g_2,I_2)\not=(1,0,\emptyset)}}
\sum_{\substack{a'_1,\dots a'_k \in \Z_{\geq 1} \\ \sum_{i=1}^k a'_i = \sum\nolimits_{j\in I_1}a_j}} \prod_{i=1}^k a'_i \cdot  (2g_2+|I_2|+k-1)\cdot 
\\ \notag & \qquad \qquad \rho_* \bigg( s_* \left(\frac{{1}}{1-t^*\psi_0} \left[\oM_{g_1}^{\sim}\left(\mathbb{P}^1,\{a_i\}_{i\in I_1},\{-a'_j\}_{j=1}^k \right)\right]^{\mathrm{vir}}\right) \otimes 
\\ \notag & \qquad \qquad  \qquad \quad
s_*\left( \left[\oM_{g_2}^{\sim}\left(\mathbb{P}^1,\{a'_j\}_{j=1}^k,\{a_i\}_{i\in I_2},-\suma \right)\right]^{\mathrm{vir}}\right)  \bigg)
\end{align}
Using~\cite[Eq.~(2)]{BSSZ}, the latter expression assembles into
\begin{align}
	\eqref{eq:push-A-replacement-pullback-2} & = (2g-1+n) \cdot s_* \left( \left[\oM_{g}^{\sim}\left(\mathbb{P}^1,a_1,\dots,a_n,-\suma \right)\right]^{\mathrm{vir}}\right)
	\\ \notag & \quad 
	+ (2g-1+n) \cdot  s_* \left(\frac{t^*\psi_0}{1-t^*\psi_0} \left[\oM_{g}^{\sim}\left(\mathbb{P}^1,a_1,\dots,a_n,-\suma \right)\right]^{\mathrm{vir}}\right)
	\\ \notag & = (2g-1+n) \cdot  s_* \left(\frac{1}{1-t^*\psi_0} \left[\oM_{g}^{\sim}\left(\mathbb{P}^1,a_1,\dots,a_n,-\suma \right)\right]^{\mathrm{vir}}\right),
\end{align}
which exactly the expression that we use in equality (3) of Eq.~\eqref{eq:lemma-push-A-1summand}.

Finally, 
\begin{align} \label{eq:push-forward-A-case-Div}
	& \pi'_* s_* \left(\frac{D_{n+1}}{1-t^*\psi_0} \left[\oM_{g}^{\sim}\left(\mathbb{P}^1,a_1,\dots,a_n,b,-\suma-b \right)\right]^{\mathrm{vir}}\right)
	\\ \notag & 
	{=}\sum_{k=1}^\infty \frac 1{k!} \sum_{\substack{g_1+g_2+k-1 = g\\ 
			I_1\sqcup I_2 = \{1,\dots,n\} }}
	\sum_{\substack{a'_1,\dots a'_k \in \Z_{\geq 1} \\ \sum_{i=1}^k a'_i = \sum\nolimits_{j\in I_1}a_j}} \prod_{i=1}^k a'_i \cdot  
	\\ \notag &\qquad \qquad \pi'_* \rho_* \bigg( s_* \left(\frac{{1}}{1-t^*\psi_0} \left[\oM_{g_1}^{\sim}\left(\mathbb{P}^1,\{a_i\}_{i\in I_1},\{-a'_j\}_{j=1}^k \right)\right]^{\mathrm{vir}}\right) \otimes 
	\\ \notag & \qquad \qquad \qquad \qquad 
	s_*\left( \left[\oM_{g_2}^{\sim}\left(\mathbb{P}^1,\{a'_j\}_{j=1}^k,\{a_i\}_{i\in I_2},b,-\suma-b \right)\right]^{\mathrm{vir}}\right) \bigg),
\end{align}
where the sum is over all irreducible strata of two-component curves connected by $k\geq 1$ nodes, and we assume with the points labelled by $I_1$ on the component of genus $g_1$ and the points labelled by $I_2\cup \{n+1,n+2\}$ are on the component of genus $g_2$ (cf.~\cite[Sec.~2]{BSSZ}). Here $\rho$ denotes the corresponding boundary morphism. The latter expression can be further rewritten as 
\begin{align} \label{eq:lemma-push-A-3summand}
	\eqref{eq:push-forward-A-case-Div} & = \suma \cdot 
	s_* \left(\frac{{1}}{1-t^*\psi_0} \left[\oM_{g}^{\sim}\left(\mathbb{P}^1,a_1,\dots,a_n,-\suma \right)\right]^{\mathrm{vir}}\right)
	\\ \notag & \qquad
	+ \sum_{k=1}^\infty \frac 1{k!} \sum_{\substack{g_1+g_2+k-1 = g\\ 
			I_1\sqcup I_2 = \{1,\dots,n\} \\ (k,g_2,I_2)\not=(1,0,\emptyset)}}
	\sum_{\substack{a'_1,\dots a'_k \in \Z_{\geq 1} \\ \sum_{i=1}^k a'_i = \sum\nolimits_{j\in I_1}a_j}} \prod_{i=1}^k a'_i \cdot 
	\\ \notag &\quad \qquad \qquad  \rho_* \bigg( s_* \left(\frac{{1}}{1-t^*\psi_0} \left[\oM_{g_1}^{\sim}\left(\mathbb{P}^1,\{a_i\}_{i\in I_1},\{-a'_j\}_{j=1}^k \right)\right]^{\mathrm{vir}}\right) \otimes 
	\\ \notag & \quad \qquad \qquad \qquad \quad 
	\pi'_* s_*\left( \left[\oM_{g_2}^{\sim}\left(\mathbb{P}^1,\{a'_j\}_{j=1}^k,\{a_i\}_{i\in I_2},b,-\suma-b \right)\right]^{\mathrm{vir}}\right) \bigg)
	\\ \notag & = \suma \cdot 
	s_* \left(\frac{{1}}{1-t^*\psi_0} \left[\oM_{g}^{\sim}\left(\mathbb{P}^1,a_1,\dots,a_n,-\suma \right)\right]^{\mathrm{vir}}\right) + O(b^2),
\end{align}
where for the second equality we use~\cite[Prop.~2.12]{BLS-Quantum}. Now the sum of~\eqref{eq:lemma-push-A-1summand} and the final expressions in~\eqref{eq:lemma-push-A-2summand} and~\eqref{eq:lemma-push-A-3summand} gives the right hand side of~\eqref{eq:statment-lemma-push-A}, which proves the statement of the lemma. 
\end{proof}

\subsubsection{The master relation for \texorpdfstring{$\bA_{g,n}$}{Agn}} 
\label{subsubsec:master-A}
In this section we prove the $1$-st master relation, or M-1, for $\{\bA_{g,n}\}$ in the sense of Def.~\ref{def: master relation}. To this end, we use the techniques from the previous subsection. In particular, it is more convenient to work with the following classes:

Let $a_1,\dots,a_n$ be positive integers. Define the following class in $R^\bullet(\oM_{g,n+2})$:
\begin{equation} \label{eq: A-1-star}
\bA_{g,n}^{1,\star}(a_{1},\dots,a_{n})\coloneqq\lambda_{g}s_{*}\left(\frac{1}{1-t^{*}\psi_{0}}\left[\overline{\mathcal{M}}_{g,1}^{\sim}\left(\mathbb{P}^{1},a_{1},\dots,a_{n},-\suma\right)\right]^{\mathrm{vir}}\right),
\end{equation}
where $\overline{\mathcal{M}}_{g,1}^{\sim}\left(\mathbb{P}^{1},a_{1},\dots,a_{n},-\suma\right)$ denotes the moduli space of rubber maps of Sec.~\ref{Sec: def A and A1}, with one additional free marked point over $\mathbb{P}^{1}\backslash\left\{ 0,\infty\right\} $. The first $n$ marked points correspond to the multiplicities $a_1,\dots,a_n$, the $(n+1)$-st marked point is the free one, and the $(n+2)$-and marker point corresponds to the multiplicity $-\suma$. 
Define also 
\begin{equation}
\bA^\star_{g,n}\coloneqq \frac{1}{\suma}\pi_{*}\bA_{g,n+1}^{1,\star} \in R^\bullet(\oM_{g,n+1}), 
\end{equation}
where $\pi\colon \oM_{g,n+2} \to \oM_{g,n+1}$ is the morphism forgetting the last marked point. 

\begin{lemma} \label{lem:A1star-A} We have
\begin{align} \label{eq:A-star-A}
    \bA_{g,n}^{1,\star}(a_{1},\dots,a_{n}) & = \Coeff{b_1^0} \bA_{g,n+1}^{1}(a_1,\dots,a_n,b_1); \\ \notag 
    \bA_{g,n}^{\star}(a_{1},\dots,a_{n}) & = \Coeff{b_1^0} \bA_{g,n+1}(a_1,\dots,a_n,b_1).
\end{align}
In particular, both classes $\bA_{g,n}^{1,\star}(a_{1},\dots,a_{n})$ and $\bA_{g,n}^{\star}(a_{1},\dots,a_{n})$ depend polynomially on $a_1,\dots,a_n$ and Eq.~\eqref{eq:A-star-A} gives equalities of classes in $R^\bullet(\oM_{g,n+2})\otimes_\Q \Q[a_1,\dots,a_n]$ and 
$R^\bullet(\oM_{g,n+1})\otimes_\Q \Q[a_1,\dots,a_n]$, respectively. 
\end{lemma}

\begin{proof} We prove it for $\bA_{g,n+1}^{1}$, and then apply the push-forward to obtain the statement for $\bA_{g,n+1}$. We use the same ideas as in the proof of Lemma~\ref{lem:push-forward-A1}, and in particular, the factor $\lambda_g$ is not necessary for this argument. Recall Eq.~\eqref{eq: split 1 / 1-psi}. Applying it in our situation, we have
\begin{align} \label{eq: coeff b_1 0}
	& \Coeff{b_1^0} s_* \left(\frac{1}{1-t^*\psi_0} \left[\oM_{g}^{\sim}\left(\mathbb{P}^1,a_1,\dots,a_n,b_1,-\suma-b_1 \right)\right]^{\mathrm{vir}}\right) = 
	\\ \notag &  
	\Coeff{b_1^0} s_* \left(\left[\oM_{g}^{\sim}\left(\mathbb{P}^1,a_1,\dots,a_n,b_1,-\suma-b_1 \right)\right]^{\mathrm{vir}}\right) 
	\\ \notag &      
    + \Coeff{b_1^0} s_* \left(\frac{1}{1-t^*\psi_0} D_{n+1}\left[\oM_{g}^{\sim}\left(\mathbb{P}^1,a_1,\dots,a_n,b_1,-\suma-b_1 \right)\right]^{\mathrm{vir}}\right).
\end{align}
Note that 
\begin{align} \label{eq:s-push-forward-b1}
& \Coeff{b_1^0} s_* \left(\left[\oM_{g}^{\sim}\left(\mathbb{P}^1,a_1,\dots,a_n,b_1,-\suma-b_1 \right)\right]^{\mathrm{vir}}\right) 
\\ \notag & = s_*\left[\oM_{g,1}^{\sim}\left(\mathbb{P}^1,a_1,\dots,a_n,-\suma \right)\right]^{\mathrm{vir}},
\end{align}
where we use the polynomiality of the $\DR$ cycle~\cite{Pixton,spelier2024polynomialitydoubleramificationcycle}.

The second summand in Eq.~\eqref{eq: coeff b_1 0} can be rewritten  as 
\begin{align}
	\label{eq: no idea of good name}
    %& \Coeff{b_1^0} D_{n+1}\left[\oM_{g}^{\sim}\left(\mathbb{P}^1,a_1,\dots,a_n,b_1,-\suma-b_1 \right)\right]^{\mathrm{vir}} 
    %\\ \notag & = t^* \Delta_{1}  \left[\oM_{g,1}^{\sim}\left(\mathbb{P}^1,a_1,\dots,a_n,-\suma\right)\right]^{\mathrm{vir}},
    %\\
    & \Coeff{b_1^0} s_* \left(t^* \Big(\frac{1}{1-\psi_0} \Delta_{1}\Big)\left[\oM_{g,1}^{\sim}\left(\mathbb{P}^1,a_1,\dots,a_n,-\suma \right)\right]^{\mathrm{vir}}\right)
\end{align}
where $\Delta_1$ is the divisor in the Losev-Manin space $LM_{2g+n}$ that is composed of two-component curves such that the point $x_1$ lies on the same component as $\infty$. Here the point $x_1$ is distinguished as the image of the free marked point
%in $\oM_{g,1}^{\sim}\left(\mathbb{P}^1,a_1,\dots,a_n,-\suma\right)$
under the map $t$.
Indeed, writing $D_{n+1}$ explicitly, we get
\begin{align}
&\frac{1}{1-t^{*}\psi_{0}} D_{n+1}\left[\oM^{\sim}_{g}\left(\mathbb{P}^{1},a_{1},\dots,a_{n},b_{1},-\suma-b_{1}\right)\right]^{\mathrm{vir}}=\\
 & \notag \sum_{\substack{g_{1}+g_{2}=g-k+1\\
I_{1}\sqcup I_{2}=\left\{ 1,\dots,n\right\} 
}
}\sum_{p\geq0}\sum_{k_{1}+\cdots+k_{p}=|a_{I_{1}}|}
 \frac{k_1\cdots k_p}{p!}\left(\frac{1}{1-t^{*}\psi_{0}}\left[\oM^{\sim,\bullet}_{g_{1}}\left(\mathbb{P}^{1},a_{I_{1}},-k_{1},\dots,-k_{p}\right)\right]^{\mathrm{vir}}\right)\\
& \notag \quad\quad\quad \, \quad\quad\quad\quad\quad\, \quad\quad\quad \, \quad\quad \quad\quad \boxtimes\left[\oM^{\sim}_{g_{2}}\left(\mathbb{P}^{1},a_{I_{2}},k_{1},\dots,k_{p},b_{1},-\suma-b_{1}\right)\right]^{\mathrm{vir}}
\end{align}
where $\boxtimes$ indicates the glueing of the the two moduli of maps at the $p$ marked points, and $\bullet$ indicates that the source curve can be disconnected. After applying $s_*$, we can extract the coefficient of $b_1^0$ using Eq.~\eqref{eq:s-push-forward-b1}, since $b_1$ only appears in the component of genus $g_2$. This gives Eq.~(\ref{eq: no idea of good name}).

Now note that by the topological recursion relation on $LM_{2g+n}$, $\psi_0 = \Delta_1$, the right hand side of~\eqref{eq: coeff b_1 0} is equal to
\begin{align} \label{eq:TRR-LosevManin-psi-0}
    & s_* \left( \left( 1 + \frac{t^*\psi_0}{1-t^*\psi_0}\right)\left[\oM_{g,1}^{\sim}\left(\mathbb{P}^1,a_1,\dots,a_n,-\suma \right)\right]^{\mathrm{vir}} \right)
    \\ \notag & 
    =s_{*}\left(\frac{1}{1-t^{*}\psi_{0}}\left[\overline{\mathcal{M}}_{g,1}^{\sim}\left(\mathbb{P}^{1},a_{1},\dots,a_{n},-\suma\right)\right]^{\mathrm{vir}}\right).
\end{align}
Multiplying this computation by $\lambda_g$, we obtain the desired identity. 
\end{proof}

Using this Lemma, we can reformulate the $1$-st master relation for $\{A_{g,n}\}$ as the following vanishing:
\begin{align}
&\deg_{R^\bullet} \bigg( \mathbb{D}_{g,n+1}\left(a_{1},\dots,a_{n}\right)+\bA^\star_{g,n}(a_1,\dots,a_n)
\\ \notag & \qquad 
  +\sum_{\substack{T\in\SRT(g,n,1)\\ 
\ell\in\mathcal{L}(t),\,\ell(T)=1
}
}\biggl(\prod_{e\in E(T)}a(e)\biggr)(b_{T})_{*}\bigg(\bA^{\star}(v_{r})\otimes\bigotimes_{v\in V_{nr}(T)}\bD(v)\bigg) \bigg) \leq 2g-1,
\end{align}
where $\bA^{\star}(v_{r}) \coloneqq \bA^{\star}_{g(v_{r}),|H_+(v_r)|} (a(h_1),\dots,a(h_{|H_+(v_r)|})$. Note that the degree of this expression is by construction bounded from below by $2g-1$, so in fact we want to prove that its degree is exactly $2g-1$. Note also that this expression is equal to
\begin{align}
& \frac{1}{\suma} \pi_* \bigg( 
    \bA^{1,\star}_{g,n}(a_1,\dots,a_n)
\\ \notag & \qquad 
  +\sum_{\substack{T\in\SRT(g,n,2)\\ 
\ell\in\mathcal{L}(t),\,\ell(T)=1
}
}\biggl(\prod_{e\in E(T)}a(e)\biggr)(b_{T})_{*}\bigg(\bA^{1,\star}(v_{r})\otimes\bigotimes_{v\in V_{nr}(T)}\bD(v)\bigg)\bigg),
\end{align}
where $\bA^{1,\star}(v_{r}) \coloneqq \bA^{1,\star}_{g(v_{r}),|H_+(v_r)|} (a(h_1),\dots,a(h_{|H_+(v_r)|})$, and where $\pi\colon \oM_{g,n+2}\to \oM_{g,n+1}$ forgets the last marked point. The term $\mathbb{D}_{g,n+1}\left(a_{1},\dots,a_{n}\right)$ appears under this push-forward when $g(v_r)=0$ and $|H_+(v_r)|=1$, that is, the component corresponding to the root vertex is contracted by $\pi$. Thus we reduce the $1$-st master relation to the following lemma:

\begin{lemma} For any $g\geq 0$, $n\geq 1$, $a_1,\dots,a_n>0$, the Chow degree of
    \begin{align} \label{eq:A1star-relation}
    & \bA^{1,\star}_{g,n}(a_1,\dots,a_n)
\\ \notag & 
  +\sum_{\substack{T\in\SRT(g,n,2)\\ 
\ell\in\mathcal{L}(t),\,\ell(T)=1
}
}\biggl(\prod_{e\in E(T)}a(e)\biggr)(b_{T})_{*}\bigg(\bA^{1,\star}(v_{r})\otimes\bigotimes_{v\in V_{nr}(T)}\bD(v)\bigg)\bigg)
\end{align}
is bounded from above by $2g$.
\end{lemma}

\begin{remark} Note that the degree of~\eqref{eq:A1star-relation} is also bounded by $2g$ from below, by construction. So, in fact, its degree is just equal to $2g$, and~\eqref{eq:A1star-relation} is equal to 
\begin{align}
    \lambda_g\DR_g(a_1,\dots,a_n,0,-\suma).
\end{align}
\end{remark}

\begin{proof} It is a slight refinement of the argument used in the proof of~\cite[Eq.~(5.12)]{BLS-Omega}. Indeed, following exaclty the same reasoning as in \emph{op.~cit.}, one can show that~\eqref{eq:A1star-relation} is equal to the sum of $\lambda_g\DR_g(a_1,\dots,a_n,0,-\suma)$ and 
\begin{align}
    \lambda_g s_*\Big( t^*Y \left[\overline{\mathcal{M}}_{g,1}^{\sim}\left(\mathbb{P}^{1},a_{1},\dots,a_{n},-\suma\right)\right]^{\mathrm{vir}}\Big),
\end{align}
where $Y$ is the class on the Losev-Manin space defined as
\begin{align}
    Y\coloneqq \frac {\psi_0}{1-\psi_0} + \sum_{i,j\geq 0} (-1)^{i+1}\sum_{D\in \Delta_1} (b_D)_* \psi^i_{\infty'}\psi^j_{0'},
\end{align}
where the sum runs over all irreducible divisors in $\Delta_1$, and the classes $\psi_{\infty'},\psi_{0'}$ are the $\psi$-classes at the node on the components that contain $0$ and $\infty$, respectively. Using the topological recursion relation on the Losev-Manin space, it is immediate to see that $Y=0$ in $A^*(LM_{2g+n})$. Thus the degree of~\eqref{eq:A1star-relation} is equal to $2g$ (and in particular is bounded from above by $2g$). 
\end{proof}

As we have seen above, this lemma implies the $1$-st master relation for $\{\bA_{g,n}\}$, and this completes the proof of Thm.~\ref{thm: A} by Lemma~\ref{lem: LRS - master equivalence}. 

\subsubsection{Proof of Prop.~\ref{prop: A-DR normal F-CohFT}}
\label{sub:sec:prop A-DR proof}
Consider the integrable system associated to $\{\mathfrak{fc}_{g,n}\}$ for the system of integrable observables $\{\bA_{g,n}\}$. Notice that in this case $w^\alpha_{\mathrm{top}} = \d_{t^{1,0}} {}^{\bA}\!X^\alpha$ can be rewritten using 
%Lemma~\ref{lem:A1star-A} twice, with an evaluation of $\psi_0 = \Delta_1$ in between
the string equation in Cor.~\ref{cor:push-forward-A} and subsequently the pull-back to get the class $\bA^1$ and then Lemma~\ref{lem:A1star-A}, as
\begin{align}
      w^\alpha_{\mathrm{top}} & = \sum_{{g\ge 0,n\ge 1}} \frac{\epsilon^{2g}}{n!} \sum_{{d_1,\dots,d_n\geq 0}} \prod_{i=1}^n t^{\gamma_i,d_i} \times 
    \\ \notag & \qquad 
    \int_{\oM_{g,n+2}} \Coeff{\prod_{i=1}^n a_i^{d_i}} \big[(\mathrm{swap}_{n+1,n+2})_*\bA_{g,n}^{1,\star}(a_{1},\dots,a_{n})\big]^\circlearrowright \times
    \\ \notag & \qquad \qquad e^{\alpha}\mathfrak{fc}_{g,n+1}\Big(\Vec{\bigotimes}_{i=1}^n e_{\gamma_i} \otimes e_1\Big),
        & \alpha=1,\dots,N,
\end{align}
where the class $\bA_{g,n}^{1,\star}(a_{1},\dots,a_{n})$ is a variation of $\bA_{g,n}^{1}$ with one extra point defined in Eq.~\eqref{eq: A-1-star} and ${\mathrm{swap}}_{n+1,n+2}$ interchanges the last two marked points.
Note the difference with $u^{\alpha}_{\text{F-str}}$ given by Eq.~\eqref{eq: def: F-str u}, which can also be described using Lemma~\ref{lem:A1star-A} as 
\begin{align} \label{eq: def: F-str u Ver 2}
      u^{\alpha}_{\text{F-str}} & \coloneqq \sum_{{g\ge 0,n\ge 1}} \frac{\epsilon^{2g}}{n!} \sum_{{d_1,\dots,d_n\geq 0}} \prod_{i=1}^n t^{\gamma_i,d_i} \times 
    \\ \notag & \qquad 
    \int_{\oM_{g,n+2}} \bigg[\Coeff{\prod_{i=1}^n a_i^{d_i}} \bA^{1,\star}_{g,n}(a_1,\dots,a_n)\bigg]^\circlearrowright \times
    \\ \notag & \qquad \qquad e^{\alpha} \bigg(\mathfrak{fc}_{g,n+1}\Big(\Vec{\bigotimes}_{i=1}^n e_{\gamma_i} \otimes e_1\Big)\bigg).
        & \alpha=1,\dots,N.
\end{align}

Applying Eq.~\eqref{eq:TRR-LosevManin-psi-0} and using $\psi_0 = \Delta_1$ as in the proof of Lemma~\ref{lem:A1star-A}, one can notice that these two series in $\{t^{\gamma,d}\}$ are connected by the Miura tranformation~\eqref{eq: normal coordinates for F CohFT}:
\begin{align}
    w^\alpha_{\mathrm{top}} & = \sum_{\substack{g,n\ge 0\\ 2g+n>0}}\frac{\epsilon^{2g}}{n!} \sum_{\substack{d_1,\dots,d_n\geq 0 \\ d_1+\cdots+d_n=2g}} \prod_{i=1}^n u^{\gamma_i,d_i}_{\text{F-str}}  \times
\\ \notag & \qquad \int_{\oM_{g,n+2}} \Coeff{\prod_{i=1}^n a_i^{d_i}}\lambda_g \DR_g (0,-\suma, a_1,\dots,a_n) 
\times 
\\ \notag & \qquad \qquad e^{\alpha} \bigg(\mathfrak{fc}_{g,n+1}\Big(e_1\otimes \Vec{\bigotimes}_{i=1}^n e_{\gamma_i}\Big)\bigg).
\end{align}
This automatically implies that the corresponding fluxes are related by the same Miura transformation. Indeed, we have the identity of the formal power series in $\{t^{\gamma,d}\}$ given as
\begin{align}
\d_x R^{\alpha}_{\beta,p}\big|_{w^{\gamma,q} = w^{\gamma,q}_{\mathrm{top}}} & = \d_{t^{\beta,p}} w^\alpha_{\mathrm{top}} 
%\\ \notag & 
= \sum_{s\geq 0}
(\d_{u^{\zeta,s}}u^{\alpha}_{\mathrm{norm}})\big|_{u^{\gamma,q} = u^{\gamma,q}_{\mathrm{F-str}}} \d_x^{q} \d_{t^{\beta,p}} u^{\zeta}_{\mathrm{F-str}}
\\ \notag & = \sum_{s\geq 0}
(\d_{u^{\zeta,s}}u^{\alpha}_{\mathrm{norm}})\big|_{u^{\gamma,q} = u^{\gamma,q}_{\mathrm{F-str}}} \d_x^{q+1} Q^\zeta_{\beta,p}\big|_{u^{\gamma,q} = u^{\gamma,q}_{\mathrm{F-str}}},
\end{align}
where $R^\alpha_{\beta,p}$ are the fluxes of the integrable system associated to $\{\mathfrak{fc}_{g,n}\}$ using the integrable observables $\{\bA_{g,n}\}$, and $Q^\zeta_{\beta,p}$ are the fluxes of the DR hierarchy given by Eq.~\eqref{eq:definition of flux DR}.
Since $w^{\gamma,q} = w^{\gamma,q}_{\mathrm{top}}$ can be alternatively presented as $w^{\gamma,d} = u^{\gamma,d}_{\mathrm{norm}}\big|_{u^{\gamma,q} = u^{\gamma,q}_{\mathrm{F-str}}}$  and the identity of the formal power series in $\{t^{\gamma,d}\}$ implies the identity in the dependent variables (see Remark~\ref{rem:RemarkIdentification} or~\cite[Lemma 4.3]{BS22}), we indeed see that the fluxes of the integrable system associated to $\{\mathfrak{fc}_{g,n}\}$ using the integrable observables $\{\bA_{g,n}\}$ are obtained from the fluxed of the DR hierarchy by the Miura transformation~\eqref{eq: normal coordinates for F CohFT}.

\begin{remark}
    As an alternative proof, one can apply the argument given in~\cite[Proof of Thm.~4.9]{BS22} to ${}^{\bA}\!X^\alpha$, which itself is an adaptation to the setting of F-CohFTs of the arguments given in~\cite{BDGR1,BDGR20}. It still works, with an adjustment of the root vertex in the emerging graphs that indeed amounts to the Miura transformation~\eqref{eq: normal coordinates for F CohFT}.
\end{remark}

\subsection{Properties of \texorpdfstring{$\bOm_{g,n}$}{HHejgn}} The goal of this section is to give a proof of Theorem~\ref{thm: Hej}. 

\subsubsection{Polynomiality, homogeneity, and genus \texorpdfstring{$0$}{0} properties of \texorpdfstring{$\bOm_{g,n}$}{Hejgn}} \label{subsec: Poly Homo g0 Hejgn}
The polynomiality properties of $\bOm_{g,n}$ were established in \cite[Proposition 3.4]{BLRS}. The genus $0$ property follows from \cite[Lemma 3.8]{BLRS} after using Eq.~(\ref{eq:Omega-negative-ai-psi}). The homogeneity property is proved in~\cite[Prop. 3.5]{BLRS}. 

\subsubsection{Push-forward property of \texorpdfstring{$\bOm_{g,n}$}{HHejgn}}
\label{subsubsec: push-forward HHej}
We use  Eq.~(\ref{eq:Omega-negative-ai-psi}) to write $\bOm_{g,n+1}\left(a_{1},\dots,a_{n},b\right)$ as 
\begin{align}
\label{eq:Hej-class-with-b}
\underline{\underset{\,}{\lambda_{g}\left(\suma+b\right)^{1-g}\frac{\Omega_{g,n+1}^{[\suma+b]}(\suma+b,0;\suma+b-a_{1},\dots,\suma+b-a_{n},\suma)}{\prod_{i=1}^{n}(1-a_{i}\psi_{i})}}}\times\frac{1}{(1-b\psi_{n+1})}.	
\end{align}
We will use that the underlined term is a polynomial in $a_1,\dots,a_n,b$ which is clear from the proof of Proposition~3.4 in \cite{BLRS}. In addition, we repeatedly use the pull-back property and the $r$-symmetry of $\Omega$-classes, see \cite[Theorem~4.1]{GLN23}.

\smallskip
\paragraph{\textbf{Constant coefficient in $b$.}}
When $b=0$, the pull-back property of $\Omega$-classes yields
\begin{align}
	\bOm_{g,n+1}\left(a_{1},\dots,a_{n},0\right)=\pi^{*}\left(\lambda_{g}\suma^{1-g}\Omega_{g,n}^{[\suma]}(\suma,0;\suma-a_{1},\dots,\suma-a_{n})\right)\frac{1}{\prod_{i=1}^{n}(1-a_{i}\psi_{i})}.
\end{align}
Pushing forward by $\pi$ this equation yields $\suma\,\bOm_{g,n}\left(a_{1},\dots,a_{n}\right)$, using the formula $\pi_* \bPsi_{g,n+1}\left(a_{1},\dots,a_{n},0\right) = \suma \bPsi_{g,n}\left(a_{1},\dots,a_{n}\right)$. 
\smallskip
\paragraph{\textbf{Linear coefficient in $b$.}}
We first extract the linear coefficient of $b$ in the class $\bOm_{g,n+1}\left(a_{1},\dots,a_{n},b\right)$ using Eq.~(\ref{eq:Hej-class-with-b}). Expanding the geometric series in $b\psi_{n+1}$ at order $0$ and $1$ gives two terms. The term of order $1$ yields
\begin{align}
	\lambda_{g}\suma^{1-g}\frac{\Omega_{g,n+1}^{[\suma]}(\suma,0;\suma-a_{1},\dots,\suma-a_{n},\suma)}{\prod_{i=1}^{n}(1-a_{i}\psi_{i})}\times\psi_{n+1}.
\end{align}
Pushing forward via $\pi$ gives $(2g-2+n)\bOm_{g,n}\left(a_{1},\dots,a_{n}\right)$, using the pull-back property of $\Omega$ and  $\pi_*(\bPsi_{g,n+1}(a_1,\dots,a_n,0)\psi_{n+1})=(2g-2+n)\bPsi_{g,n+1}(a_1,\dots,a_n)$. Therefore, it remains to show that the push-forward of the second term vanishes, explicitly we prove that
\begin{align}
\label{eq:vanishing-term}
	\frac{\partial}{\partial b}\left[\pi_{*}\left(\lambda_{g}\left(\suma+b\right)^{1-g}\frac{\Omega_{g,n+1}^{[\suma+b]}(\suma+b,0;\suma+b-a_{1},\dots,\suma+b-a_{n},\suma)}{\prod_{i=1}^{n}(1-a_{i}\psi_{i})}\right)\right]_{b=0}=0.
\end{align}

To prove this, we apply Chiodo's formula \cite{Chiodo08} for the $\Omega$-class. Since the class is multiplied by $\lambda_g$, it suffices to restrict to the locus of compact type curves. On this locus, Chiodo's formula takes the following form:

\begin{lemma} \label{lem:dilaton:hey} On the moduli space of compact type curves, we have
\begin{align}
	\Omega_{g,n+1}^{[\suma+b]}(\suma+b,0;\suma+b-a_{1},\dots,\suma+b-a_{n},\suma)=\left(\suma+b\right)^{2g-1}e^{f\left(a_{1},\dots,a_{n},b\right)}
	%\exp\left(f\left(a_{1},\dots,a_{n},b\right)\right)
\end{align}
where $f\left(a_{1},\dots,a_{n},b\right)$ is given by
\begin{align}
\sum_{m\geq1}\!\frac{\left(-1\right)^{m}\!\left(\suma+b\right)^{m}}{m\left(m+1\right)}&\Bigg(  B_{m+1}\kappa_{m}-\sum_{i=1}^{n}B_{m+1}\left(\frac{\suma-a_{i}}{\suma+b}\right)\psi_{i}^{m}-B_{m+1}\left(\frac{\suma}{\suma+b}\right)\psi_{n+1}^{m}\nonumber\\
 & +\frac{1}{2}\!\sum_{\substack{g_{1}+g_{2}=g\\
I\sqcup J=\left\{ 1,\dots,n+1\right\} 
}
}\!B_{m+1}\left(\frac{w\left(h_{1}\right)}{\suma+b}\right)\left(\xi_{g_{1},g_{2}}^{I,J}\right)_{*}\frac{\psi_{h_{1}}^{m}-\left(-\psi_{h_{2}}\right)^{m}}{\psi_{h_{1}}+\psi_{h_{2}}}\Bigg).
\end{align}
Here the last sum runs over stable boundary divisors with a separating node:
\[
\xi_{g_{1},g_{2}}^{I,J}:\overline{\mathcal{M}}_{g_{1},\left|I\right|+1}\times\overline{\mathcal{M}}_{g_{2},\left|J\right|+1}\rightarrow\overline{\mathcal{\mathcal{M}}}_{g,n+1}
\]
is the gluing morphism, $h_{1}$ (resp. $h_{2}$) denotes the half-edge on the component of genus $g_{1}$ (resp. $g_{2}$), and the weight is 
\[
w\left(h_{1}\right)=\begin{cases}
\sum_{i\in I}a_{i}, & {\rm if\,}n+1\notin I,\\
b-\sum_{i\in I\backslash\left\{ n+1\right\} }a_{i}, & {\rm if\,}n+1\in I.
\end{cases}
\]
We omit to write in the summation the stability condition which holds. Finally, $B_{m}\left(x\right)$ denotes the $m$-th Bernoulli polynomial defined by $\frac{te^{tx}}{e^{t}-1}=\sum_{m=0}^{\infty}B_{m}\left(x\right)\frac{t^{m}}{m!}$. The $m$-th Bernoulli number is given by $B_{m}=B_{m}\left(0\right)$.
\end{lemma}

\begin{proof}
In general, we have
\begin{align}
	\Omega^{[x]}(r,s;a_{1},\dots,a_{n})=\epsilon_{*}\exp\left(\sum_{m\geq1}\left(-x\right)^{m}\left(m-1\right)!{\rm ch}_{m}\left(r,s;a_{1}\dots,a_{n}\right)\right)
\end{align}
where Chern characters are given in \cite[Theorem 1.1.1]{Chiodo08}. Since we restrict to the compact type locus, only the boundary divisors corresponding to separating nodes appear in the Chern characters. Therefore, the weight between $0$ and $r-1$ assigned to each node is uniquely determined by the weights of the markings. Expanding the exponential yields a sum over trees. Each edge contributes a factor $r$, so a tree with $|E|$ edges contributes a factor $r^{|E|}$. Moreover, for a fixed tree with $|V|$ vertices, the restriction of $\epsilon$ to the associate boundary stratum has degree $r^{2g-|V|}$. Combining both contributions, we obtain an overall $r^{2g-1}$.
\end{proof}

Using Chiodo's formula, the left hand side of Eq.~(\ref{eq:vanishing-term}) writes
\begin{align}
	\underline{\underset{\,}{\pi_{*}\left(\frac{\left(\frac{g}{\suma}+\frac{\partial}{\partial b}f\left(a_{1},\dots,a_{n},b\right)\Big\vert_{b=0}\right)}{\prod_{i=1}^{n}(1-a_{i}\psi_{i})}\right)}} \lambda_{g}\Omega_{g,n}^{[\suma]}(\suma,0;\suma-a_{1},\dots,\suma-a_{n}). 
\end{align}
We now prove that underlined term vanishes.
%push-forward by $\pi$ of the term in parenthesis vanishes. 
Using 
\begin{align}
\bPsi_{g,n+1}(a_1,\dots,a_n,0)=\pi^*\bPsi_{g,n}(a_1,\dots,a_n)(1+\sum_{i=1}^{n}a_{i}\delta_{\left(i,n+1\right)}),
\end{align}
where $\delta_{\left(i,n+1\right)}$ is the boundary divisor of curves such that  the markings $i$ and $n+1$ lie on a genus $0$ component and the remaining markings lie on the other component, this vanishing is reduced to
\begin{align}
\label{eq:final-reduction}
	\frac{\partial}{\partial b}\left[\pi_{*}f\left(a_{1},\dots,a_{n},b\right)+\pi_{*}\left(\sum_{i=1}^{n}a_{i}\delta_{\left(i,n+1\right)}f\left(a_{1},\dots,a_{n},b\right)\right)\right]_{b=0}=-g.
\end{align}

We first compute the push-forwards by $\pi$. We have 
\begin{align}
\pi_{*}f\left(a_{1},\dots,a_{n},b\right) & =\sum_{m\geq1}\frac{\left(-1\right)^{m}\left(\suma+b\right)^{m}}{m\left(m+1\right)}P_{m}\left(a_{1},\dots,a_{n},b\right)
\end{align}
where
\begin{align}
P_{1}&\left(a_{1},\dots,a_{n},b\right) = \left(B_{2}-B_{2}\left(\tfrac{\suma}{\suma+b}\right)\right)\kappa_{0}
		+\sum_{i=1}^{n}\left(B_{2}\left(\tfrac{\suma-a_{i}}{\suma+b}\right)-B_{2}\left(\tfrac{\suma+b-a_{i}}{\suma+b}\right)\right),
\end{align}
and when $m\geq2$: 
\begin{align}
P_{m}&\left(a_{1},\dots,a_{n},b\right) =\\
& \left(B_{m+1}-B_{m+1}\left(\tfrac{\suma}{\suma+b}\right)\right)\kappa_{m-1}+\sum_{i=1}^{n}\left(B_{m+1}\left(\tfrac{\suma-a_{i}}{\suma+b}\right)-B_{m+1}\left(\tfrac{\suma+b-a_{i}}{\suma+b}\right)\right)\psi_{i}^{m-1}\nonumber\\
 & +\frac{1}{2}\sum_{\substack{g_{1}+g_{2}=g\nonumber\\
I\sqcup J=\left\{ 1,\dots,n\right\} 
}
}\left(B_{m}\left(\tfrac{a_{I}+b}{\suma+b}\right)-B_{m+1}\left(\tfrac{a_{I}}{\suma+b}\right)\right)\left(\xi_{g_{1},g_{2}}^{I,J}\right)_{*}\frac{\psi_{h_{1}}^{m-1}-\left(-\psi_{h_{2}}\right)^{m-1}}{\psi_{h_{1}}+\psi_{h_{2}}}.
\end{align}
We used the notation $a_I:=\sum_{i\in I} a_i$ and similarly for $J$. 
On the other hand, we have 
\begin{align}
	\pi_{*}\!\left(\sum_{i=1}^{n}a_{i}\delta_{\left(i,n+1\right)}f\left(a_{1},\dots,a_{n},b\right)\right) \!=\!\sum_{m\geq1}\frac{\left(-1\right)^{m}\left(\suma+b\right)^{m}}{m\left(m+1\right)} Q_{m}\left(a_{1},\dots,a_{n},b\right),
\end{align}
where 
\begin{align}
Q_{m}&\left(a_{1},\dots,a_{n},b\right)=\\
 & \suma B_{m+1}\kappa_{m}-\sum_{i=1}^{n}\left((\suma-a_{i})B_{m+1}\left(\tfrac{\suma+b-a_{i}}{\suma+b}\right)+a_{i}B_{m+1}\left(\tfrac{\suma-a_{i}}{\suma+b}\right)\right)\psi_{i}^{m}\nonumber\\
 & +\frac{1}{2}\sum_{\substack{g_{1}+g_{2}=g\\
I\sqcup J=\left\{ 1,\dots,n\right\} 
}
}\left(a_{I}B_{m}\left(\tfrac{a_{I}+b}{\suma+b}\right)+a_{J}B_{m+1}\left(\tfrac{a_{I}}{\suma+b}\right)\right)\left(\xi_{g_{1},g_{2}}^{I,J}\right)_{*}\frac{\psi_{h_{1}}^{m}-\left(-\psi_{h_{2}}\right)^{m}}{\psi_{h_{1}}+\psi_{h_{2}}}.\nonumber
\end{align}
Finally, Eq.~(\ref{eq:final-reduction}) directly follows from the following lemma. 
\begin{lemma}\label{lem:dilaton:hey:2}
We have 
\begin{align}
	\frac{\partial P_{1}\left(a_{1},\dots,a_{n},b\right)}{\partial b}\Bigg\vert_{b=0}=\frac{2g}{\suma}
\end{align}
and for $m\geq1$ we have 
\begin{align}
P_{m}\left(a_{1},\dots,a_{n},b\right)\Big\vert_{b=0} & =0,\\
\frac{\partial Q_{m}\left(a_{1},\dots,a_{n},b\right)}{\partial b}\Bigg\vert_{b=0} & =0,\\
\frac{\partial P_{m+1}\left(a_{1},\dots,a_{n},b\right)}{\partial b}\Bigg\vert_{b=0} & =\frac{m+2}{\suma^2}Q_{m}\left(a_{1},\dots,a_{n},b\right)\Big\vert_{b=0}.
\end{align}
\end{lemma}
\begin{proof}
This follows from a direct computation using the identities $\d_x B_{m+1}\left(x\right)=\left(m+1\right)B_{m}\left(x\right)$ for $m\geq1$, together with $B_{m}\left(1\right)=B_m$ for $m>1$,  $B_{1}\left(x\right)=x-\frac{1}{2}$, and $\kappa_{0}=2g-2+n$ on $\overline{\mathcal{M}}_{g,n}$.
\end{proof}
This concludes the proof. 

\subsubsection{Leveled rooted tree relations for \texorpdfstring{$\bOm_{g,n}$}{Hejgn}} 

Using Lemmata~\ref{lem: LRS - master equivalence} and~\ref{lem: geometric master}, we see that it is sufficient to prove Prop.~\ref{prop: geometric master} for $\{\bOm_{g,n}\}$. Then it implies the full set of leveled rooted tree relations, and in particular type $1$ and type $2$ integrability of $\bOm_{g,n}$. 

The proof of Prop.~\ref{prop: geometric master} for $\{\bOm_{g,n}\}$ requires localization technique and we give it in the next Section. This completes the proof of Thm.~\ref{thm: Hej}.

%%%%%%%%%%%%%%%%%%%%%%%%%%%%%%%%%%%%%%%%%%%%%%%%%%%%%%%%%%%%%%%%%%%%%%%%%
\vspace{1cm}
%\section{Localization relations/ SRT relation for $\bOm$/Master for $\bOm$}
\section{Geometric master relations for \texorpdfstring{$\{\bOm_{g,n}\}$}{HHejgn} via localization}
\label{sec:localization}
%%%%%%%%%%%%%%%%%%%%%%%%%%%%%%%%%%%%%%%%%%%%%%%%%%%%%%%%%%%%%%%%%%%%%%%%%

%\begin{remark}
%It seems to work for
%\[
%\bOm(v)\coloneqq\begin{cases}
%\bOm_{g(v),|H(v)|}\left(a_{1},\dots,a_{n},b_{1},\dots,b_{m}\right)\prod_{i=1}^{m}\left(1-b_{i}\psi_{i+n}\right) & v=v_{r}\\
%\bOm_{g(v),|H(v)|}\left(a_{1},\dots,a_{n},b_{v}\right)\left(1-b_{v}\psi_{v}\right) & v\neq v_{r}
%\end{cases}
%\]
%where $b_{v}$ is a new variable associated to vertex $v$ and $\psi_{v}$ is the $\psi$-class at $h\in H_{-}\left(v\right)$. Could we say anything about it? In terms of integrability? \todo{investigate}
%\end{remark}

The goal of this section is to prove Prop.~\ref{prop: geometric master}. As we have mentioned there, only the $\{\bOm_{g,n}\}$ part of it is new, and in the latter case once we specialize $b_{1}=\cdots=b_{m}=0$, we recover \cite[Theorem 4.1]{BLS-Omega}.  The proof below adapts the localization setup of Section 4 of \cite{BLS-Omega} to the present setting, where the marked points carry a nontrivial orbifold structure. We indicate only the required modifications.

\subsection{Stable relative maps to orbifold target}

Let $g\geq0$ and $m\geq1$. Let $a_{1},\dots,a_{n},b_{1},\dots,b_{m}$ be positive integers and denote
\begin{equation}
A=\sum_{i=1}^{n}a_{i}+\sum_{j=1}^{m}b_{j},\qquad\suma=\sum_{i=1}^{n}a_{i}.\label{eq:condition-ai-bi}
\end{equation}
Let $\mathbb{P}^{1}\left[A\right]$ be the projective line with a single orbifold point $B\mathbb{Z}_{A}$ at $\infty\in\mathbb{P}^{1}$. We consider the moduli space 
\[
\overline{\mathcal{M}}_{g,(A-b_{1},\dots,A-b_{m})}\left(\mathbb{P}^{1}\left[A\right],a_{1},\dots,a_{n}\right)
\]
of twisted stable maps to the relative pair $\left(\mathbb{P}^{1}\left[A\right],0\right)$ defined in \cite{AGV08}. It parametrizes maps 
\[
f:C\rightarrow P
\]
such that
\begin{itemize}
\item $C$ is a twisted curve of genus $g$ with $n+m$ marked points, the first $n$ are non orbifold points and the remaining $m$ are orbifold points, we do not trivialize the $m$ marked gerbes,
\item $P$ is a glueing of a chain of $l\geq0$ copies of $\mathbb{P}^{1}$ to $0\in\mathbb{P}^{1}\left[a\right]$,
\item $f$ is a relative twisted stable map such that
\begin{itemize}
\item the ramification profile over $0$ is given by the partition $\left(a_{1},\dots,a_{n}\right)$,
\item all the orbifold points (markings and nodes) of $C$ are sent to $B\mathbb{Z}_{A}\in\mathbb{P}^{1}\left[A\right]$, in particular the $m$ orbifold markings are mapped according to
\[
A-b_{1},\dots,A-b_{m},
\]
that is the stabilizer at the $i$th orbifold marking is $\mathbb{Z}_{A}/\mathbb{Z}_{\gcd \left(A,A-b_{i}\right)}\simeq\mathbb{Z}_{\mathrm{ord}\left(A-b_{i}\right)}$, where $\mathrm{ord}\left(A-b_{i}\right)$ is the order of $A-b_{i}$ in $\mathbb{Z}_{A}$, and the map of stabilizer $\mathbb{Z}_{\mathrm{ord}\left(A-b_{i}\right)}\rightarrow\mathbb{Z}_{A}$ is given by $1\rightarrow A-b_{i}$. 
\end{itemize}
\end{itemize}

% \begin{remark}
% We can assume that the stabilizers at the markings and the nodes of $C$ are all equal to $\mathbb{Z}_{A}$. Indeed, in the definition of \cite{AGV08}, the stabilizers at the nodes are not required to be $\mathbb{Z}_{A}$, but only subgroups thereof, to guarantee the injectivity of the map of stabilizers induced by $f$. However, via the identification between the moduli space of twisted maps $\overline{\mathcal{M}}_{g,(A-b_{1},\dots,A-b_{m})}\left(B\mathbb{Z}_{A}\right)$ and the moduli of $A$-spin structures $\overline{\mathcal{M}}_{g,m}^{1/A}\left(A-b_{1},\dots,A-b_{m}\right)$ constructed by Chiodo \cite{Chiodo08-stable}, one can use Chiodo's convention and therefore assume that all stabilizers are $\mathbb{Z}_{A}$. Both conventions yield the same result, see \cite[Remark 2]{JPPZ}.

% Is this only when we apply localization ?

% Say equivalently we can suppose that $A$ is coprime with any ``orbifold structure'' appearing at the nodes and markings. \todo{remove?}

% SO WHAT : could we act like if $A$ is coprime with any ``orbifold structure'' \todo{clarify this}    
% \end{remark}

We refer to \cite{AGV08,JPT} for a detailed definition of the moduli space of twisted stable maps to $\left(\mathbb{P}^{1}\left[A\right],0\right)$. This moduli space has a virtual fundamental class of virtual dimension
\[
\textrm{vdim}\left[\overline{\mathcal{M}}_{g,(A-b_{1},\dots,A-b_{m})}\left(\mathbb{P}^{1}\left[A\right],a_{1},\dots,a_{n}\right)\right]^{{\rm vir}}=2g-1+n.
\]

\subsection{\texorpdfstring{$\mathbb{C}^{*}$}{C-star}-action and fixed loci}

We consider the $\mathbb{C}^{*}$-action on $\mathbb{P}^{1}$ given by
\[
t\cdot\left[z_{0}:z_{1}\right]=\left[z_{0}:tz_{1}\right]
\]
which lifts to $\mathbb{P}^{1}\left[A\right]$ and to $\overline{\mathcal{M}}_{g,(A-b_{1},\dots,A-b_{m})}\left(\mathbb{P}^{1}\left[A\right],a_{1},\dots,a_{n}\right)$.

As in \cite[Section 4.3]{BLS-Omega}, each fixed locus is labelled by a bipartite graph $\Phi$. Each vertex $v\in V\left(\Phi\right)$ is decorated by a genus $g\left(v\right)$ and labelled by $0\in\mathbb{P}^{1}\left[A\right]$ or $\infty\in\mathbb{P}^{1}\left[A\right]$, we denote by $V_{0}\left(\Phi\right)$ and $V_{\infty}\left(\Phi\right)$ the set of vertices over $0$ and $\infty$, respectively. Each edge $e\in E\left(\Phi\right)$ carries a nonzero degree $d_{e}$. Finally, there are $n+m$ legs: the $n$ first, called \emph{regular} legs, correspond to the $n$ ramification points, while the $m$ remaining, called \emph{frozen} legs, correspond to the $m$ orbifold marked points. 

A vertex $v$ is stable if 
\[
2g(v)-2+\left|E_{v}\right|+m(v)+n(v)>0,
\]
where $|E_{v}|$ is the cardinal of the set of edges attached to $v$, $n(v)$ denotes the number of regular legs attached to $v$ and $m(v)$ denotes the number of frozen legs attached to $v$. We denote by $V_{0}^{st}\left(\Phi\right)$ (resp. $V_{\infty}^{st}\left(\Phi\right)$), the set of stable vertices over $0$ (resp. $\infty$).

There is a glueing map
\[
\iota:\overline{\mathcal{M}}_{\Phi}\rightarrow\overline{\mathcal{M}}_{g,(A-b_{1},\dots,A-b_{m})}\left(\mathbb{P}^{1}\left[A\right],a_{1},\dots,a_{n}\right)
\]
of degree $\text{\ensuremath{\left|{\rm Aut}\left(\Phi\right)\right|\prod_{e\in E\left(\Phi\right)}d_{e}} }$ such that
\[
\overline{\mathcal{M}}_{\Phi}=\begin{cases}
\left(\prod_{v\in V_{0}\left(\Phi\right)}\overline{\mathcal{M}}_{v}\right)^{\sim}\times\prod_{v\in V_{\infty}^{st}\left(\Phi\right)}\overline{\mathcal{M}}_{v}^{A}, & {\rm if\,the\,target\,expands}\\
\prod_{v\in V_{\infty}^{st}\left(\Phi\right)}\overline{\mathcal{M}}_{v}^{A}, & {\rm if\,the\,target\,does\,not\,expand}
\end{cases}
\]
where
\begin{itemize}
\item $\overline{\mathcal{M}}_{v}^{A}$, for $v\in V_{\infty}^{st}\left(\Phi\right)$, stands for the moduli space of $A$-spin structures 
\[
\overline{\mathcal{M}}_{g\left(v\right)}^{A,0}\Big(A-d_{e_{1}},\dots,A-d_{e_{\left|E_{v}\right|}},A-b_{i_{1}},\dots,A-b_{i_{m\left(v\right)}}\Big)
\]
where $e_{1},\dots,e_{\left|E_{v}\right|}$ are the indices of the edges attached to $v$, and $i_{1},\dots,i_{m\left(v\right)}$ are the indices of the orbifold marked points attached to $v$, 
\item $\overline{\mathcal{M}}_{v}$, for $v\in V_{0}\left(\Phi\right)$, stands for the moduli space
\[
\overline{\mathcal{M}}_{g\left(v\right),m\left(v\right)}\left(a_{i_{1}^{v}},\dots,a_{i_{\left|L_{r}^{v}\right|}^{v}},-d_{e_{1}},\dots,-d_{e_{\left|E_{v}\right|}}\right),
\]
of relative (not rubber) maps to the pair $\left(\mathbb{P}^{1},0\cup\infty\right)$, where $L_{r}^{v}=\left\{ i_{1}^{v},\dots,i_{\left|L_{r}^{v}\right|}^{v}\right\} \subseteq\left\{ 1,\dots,n\right\} $ is the set of indices of the regular legs attached to $v$, 
\item the tilde indicates that two products of maps are identified if they differ by an element of $\mathbb{C}^{*}$ in each $\mathbb{P}^{1}$ of the target.
\end{itemize}
\begin{lemma}
Each graph $\Phi$ of a fixed locus has a unique (stable or unstable) vertex over $\infty\in\mathbb{P}^{1}\left[A\right]$.
\end{lemma}

\begin{proof}
First, if $v\in V_{\infty}\left(\Phi\right)$ is stable, it must carry all $m$ frozen legs and be adjacent to all edges, it is therefore unique. Indeed, on every moduli space of $A$-spin structure, the sum of twists at the markings should be a multiple of $A$, together with the conditions (\ref{eq:condition-ai-bi}) and the fact that edge degrees add up to the total degree $\suma$ of the map, this forces the result.

Suppose now that all vertices of $V_{\infty}\left(\Phi\right)$ are unstable. There are in particular $m$ vertices of $V_{\infty}\left(\Phi\right)$ with a unique adjacent edge and a unique frozen leg. For such a vertex carrying the $(n+i)$th marking, $i=1,\dots,m$, denote by $d_{i}$ the degree of the edge. The map exists only if $d_{i}=A-b_{i}$. Since the sum of the degrees of the edges is $\suma$, conditions (\ref{eq:condition-ai-bi}) are only satisfied when $m=1$ and $V_{\infty}\left(\Phi\right)$ consists of a single vertex carrying the unique frozen leg and adjacent to the unique edge.
\end{proof}

\subsection{The virtual localization formula}

Let $u=c_{1}\left(\mathcal{O}_{\mathbb{P}^{\infty}}\left(-1\right)\right)$ be a generator of the equivariant cohomology of the point. In
\[
H^{*}\left(\overline{\mathcal{M}}_{g,(A-b_{1},\dots,A-b_{m})}\left(\mathbb{P}^{1}\left[A\right],a_{1},\dots,a_{n}\right)\right)\otimes\mathbb{Q}[u,u^{-1}],
\]
the virtual localization formula \cite{GraPandha99,GraVakil05} is
\begin{equation}
\left[\overline{\mathcal{M}}_{g,(A-b_{1},\dots,A-b_{m})}\left(\mathbb{P}^{1}\left[A\right],a_{1},\dots,a_{n}\right) \right]^{{\rm vir}}=\sum_{\Phi}\frac{1}{\left|{\rm Aut}\Phi\right|}\frac{1}{\prod d_{e}}\iota_{*}\left(\frac{\left[\overline{\mathcal{M}}_{\Phi}\right]^{{\rm vir}}}{e_{\mathbb{C}^{*}}\left(\mathcal{N}_{\Phi}^{{\rm vir}}\right)}\right),\label{eq: loca formula}
\end{equation}
where $\mathcal{N}_{\Phi}^{{\rm vir}}$ is the virtual normal bundle of the fixed locus $\Phi$ and we denote by $e_{\mathbb{C}^{*}}\left(\mathcal{N}_{\Phi}^{{\rm vir}}\right)$ its equivariant Euler class. The term $\iota_{*}\left(\left[\overline{\mathcal{M}}_{\Phi}\right]^{{\rm vir}}/e_{\mathbb{C}^{*}}\left(\mathcal{N}_{\Phi}^{{\rm vir}}\right)\right)$ is computed in the following way. We refer to \cite{JPT,JPPZ} for further details, and note in particular that gcd factors can be avoided using \cite[Remark~2]{JPPZ}. 
\begin{itemize}
\item If the target does not expand it contributes a factor of $1$, otherwise it contributes the factor 
\begin{equation}
-\frac{\prod_{e\in E\left(\Phi\right)}d_{e}}{u+\tilde{\psi}_{\infty}}\left[\left(\prod_{v\in V_{0}\left(\Phi\right)}\overline{\mathcal{M}}_{v}\right)^{\sim}\right]^{{\rm vir}}
\end{equation}
where the class $\tilde{\psi}_{\infty}$ is pull-back by the forgetful map keeping the target curve of the $\psi$-class at the point $\infty$ in the Losev-Manin space.
\item The vertex over $\infty$ contributes in the following way:
\begin{itemize}
\item if it is unstable, it contributes with $1$,
\item if it is stable, it contributes with
\begin{equation}
\prod_{e\in E\left(\Phi\right)}d_{e}\cdot\left(\frac{a}{u}\right)^{2-g-m}\frac{\tilde{\Omega}_{g\left(v\right),|E_v|+m(v)}^{\left[\frac{A}{u}\right]}\Big(A,0;A\!-\!d_{e_{1}},\dots,A\!-\!d_{e_{\left|E_v\right|}},A\!-\!b_{1},\dots,A\!-\!b_{m}\Big)}{\prod_{i=1}^{\left|E_v\right|}\left(1-d_{e_{i}}\psi_{e_{i}}\right)},
\end{equation}
where $\tilde{\Omega}_{g,n}^{[x]}(r,s;a_1,\dots,a_n)
	\!:=\!
	\sum_{i\geq 0} x^i c_i\left( -R^*\pi_* \mathcal L \right) 
	\!\in\!
	R^{*}(\overline{\mathcal{M}}^{r,s}_{g}(a_1, \dots, a_n))$
is such that its push forward by the forgetful map $\epsilon \colon
	\overline{\mathcal{M}}^{r,s}_{g}(a_1, \ldots, a_n)
	\rightarrow
	\overline{\mathcal{M}}_{g,n}$ gives the $\Omega$-class. 
\end{itemize}
\end{itemize}

\subsection{Proof of the geometric master relations}

We prove the $\{\bOm_{g,n}\}$ part of Prop.~\ref{prop: geometric master} with the following sequence of steps.
\begin{itemize}
\item We extract the negative power of $u$ on both sides of the localization formula Eq.~(\ref{eq: loca formula}). Since $\left[\overline{\mathcal{M}}_{g,m}\left(\mathbb{P}^{1}\left[a\right],a_{1},\dots,a_{n}\right)\right]^{{\rm vir}}$ is a class in equivariant cohomology, it has no negative power in $u$, this yields a list of relations. 
\item Then, we push each relation to $\overline{\mathcal{M}}_{g,n+m}$ by the source morphism $$s:\overline{\mathcal{M}}_{g,(A-b_{1},\dots,A-b_{m})}\left(\mathbb{P}^{1}\left[A\right],a_{1},\dots,a_{n}\right)\rightarrow\overline{\mathcal{M}}_{g,n+m}.$$ 
\item Finally, we intersect each relation with $\lambda_{g}$, since this class vanishes on the complement of the moduli space of compact type, the only graphs surviving are trees.
\end{itemize}
\begin{remark}
Unlike in \cite{BLS-Omega}, we do not need to force the $m$ frozen markings to lie over $\infty$, as they are nontrivial orbifold points and are therefore automatically mapped there.
\end{remark}

There are 3 types of graphs, for each graph type we give the contribution of
\[
s_{*}\left(\frac{1}{\left|{\rm Aut}\Phi\right|\prod d_{e}}\iota_{*}\left(\frac{\left[\overline{\mathcal{M}}_{\Phi}\right]^{{\rm vir}}}{e_{\mathbb{C}^{*}}\left(\mathcal{N}_{\Phi}^{{\rm vir}}\right)}\right)\right)\lambda_{g}.
\]
The computation of the class $\tilde{\psi}_{\infty}$ is done using \cite[Lemma 4.5]{BLS-Omega}.
\begin{enumerate}
\item (DR type, only for $m=1$) The DR tree has one unstable vertex over $\infty$ of genus~$0$ with the unique frozen leg attached to it, one stable vertex of genus $g$ over $0$ where the $n$ regular legs are attached, and a unique edge connecting these vertices. Its contribution is 
\[
\delta_{m,1}u^{2g-1}\sum_{d\geq0}\frac{\left(\bD_{g,n+1}(a_{1},\dots,a_{n})\right)_{\deg_{R^\bullet} = d}}{u^{d}}.
\]
\item ($\Omega$ type) The $\Omega$ tree has a unique stable vertex of genus $g$ over $\infty$ with $m$ frozen legs and $n$ edges attached to it, each edge is connected to an unstable vertex over $0$ with one regular leg. The degrees of the edges are $a_{1},\dots,a_{n}$. Its contribution is
\[
A^{1-m}u^{2g-2+m}\sum_{d\geq0}\frac{\left(\bOm_{g,n}^{m}(a_{1},\dots,a_{n},b_{1},\dots,b_{m})\prod_{i=1}^{m}\left(1-b_{i}\psi_{i}\right)\right)_{\deg_{R^\bullet} = d}}{u^{d}}.
\]
where we used Eq.~(\ref{eq:Omega-negative-ai-psi}). 
\item (Mixed type) A mixed type tree has a unique stable vertex over $\infty$ and at least one stable vertex over $0$. Its contribution is 
\begin{align}
& A^{1-m}u^{2g-2+m}
%\!\!\!\!
\sum_{\substack{T\in{\rm SRT}(g,n,m)\\
\ell\in\mathcal{L}(t),\,\ell(T)=1
}
}
%\!
\biggl(\prod_{e\in E(T)}\frac{a(e)}{u}\biggr)
\times
\\ \notag & 
(b_{T})_{*}\bigg(\sum_{d_{1}\geq0}\frac{\left(\bOm(v_{r})\prod_{i=1}^{m}\left(1-b_{i}\psi_{i}\right)\right)_{\deg_{R^\bullet} = d_{1}}}{u^{d_{1}}}\otimes
\bigotimes_{v\in V_{nr}(T)}\sum_{d_{2}\geq0}\frac{\left(\bD(v)\right)_{\deg_{R^\bullet} = d_{2}}}{u^{d_{2}}}\bigg).
\end{align}
\end{enumerate}
Extracting the coefficient of $u^{-d}$ for $d>0$ yields the geometric master relation of %Theorem~\ref{thm: loca-relation}. 
Prop.~\ref{prop: geometric master}.

%%%%%%%%%%%%%%%%%%%%%%%%%%%%%%%%%%%%%%%%%%%%%%%%%%%%%%%%%%%%%%%%%%%%%
\vspace{1cm}
\section{Examples}
\label{sec:examples}
%%%%%%%%%%%%%%%%%%%%%%%%%%%%%%%%%%%%%%%%%%%%%%%%%%%%%%%%%%%%%%%%%%%%%

\subsection{A new proof of the Witten conjecture} \label{sec: new proof WK}The Witten conjecture~\cite{Witten-KdV} stated that the partition function of the trivial CohFT (that is, in our terms we can say that we have a partial CohFT on a one-dimensional vector space $V=\langle e_1\rangle$, $\eta_{1,1}=1$, and $\mathfrak{pc}_{g,n}(e_1^{\otimes n}) = 1\in R^\bullet(\oM_{g,n})$ for all $g$ and $n$) satisfies the Korteweg-de Vries (KdV) equations. 

There are by now many proofs and their variations, where one first uses some geometric idea to compute the intersection numbers of $\psi$-classes on $\oM_{g,n}$, and then typically some technique on the side of integrable systems to connect it to the KdV equations or some other equivalent formulations. See~\cite{yang2026kdvintegrabilityguecorrelators} for the most recent instance of a variation of the argument on the side of integrable systems, and~\cite{AHIS} for the most recent proof that does use a new geometric idea. We also refer to both papers for an overview of other existing proofs.

For us it is important to count the original geometric ideas, since it is a more difficult part of each proof. We list them here in the historical order with a reference to the first proof that employed that idea:
\begin{enumerate}
    \item \cite{Kontsevich} uses the ribbon graph model for moduli spaces;
    \item \cite{KazLan,OkoPan} use the ELSV formula for Hurwitz numbers (\cite{KazLan} is published earlier, but~\cite{OkoPan} has appeared earlier as a preprint);
    \item \cite{Mirz} uses symplectic reduction for Weil-Petersson volumes;
    \item \cite{AHIS} uses reduction of $\psi$-classes on $\DR$ cycles. 
\end{enumerate}
The proof we give below uses, as we shall see, a new geometric ingredient, namely, the master relation of~\cite{BSS25}, which is based on two geometric inputs: (1)~localization of spin Gromov-Witten theory of the projective line, and (2)~the degeneration of the $\DR$ cycles over the divisors in the Losev-Manin space. Although the first input sounds close to the ELSV-type formula, which can itself be obtained by localization of usual Gromov-Witten theory of the projective line, cf.~\cite{OkoPan}, the geometry behind it and the resulting formulas are drastically different. The second geometric input is close to what is used in~\cite{AHIS}, but it is employed in a completely different way in the argument.
% The proof that we give below uses, as we shall see, a new geometric ingredient, namely, the master relation of~\cite{BSS25}, which is based on localization of spin Gromov-Witten theory of the projective line, along with the degeneration of the $\DR$ cycles over the divisors in the Losev-Manin space. Note that though the first geometric insight sounds close to the ELSV-type formula that can be obtained by localization of usual Gromov-Witten theory of the projective line, cf.~\cite{OkoPan}, the geometry behind it\xtodo{it refers to what?}\stodo{That geometry of spin GW theory of $P^1$ and usual GW theory of $P^1$ is different} and the resulting formulas are drastically different. The second geometric ingredient is close to what is used in~\cite{AHIS}, but it is employed in a completely different way in the argument.

Before turning to the proof, we record the following observation:

\begin{lemma} The Dubrovin-Zhang, $\DR$, and $\bOm$ hierarchies coincide for the trivial P-CohFT.
\end{lemma}

\begin{proof} The Dubrovin-Zhang and $\DR$ hierarchies are proved to coincide in this case in~\cite{Bur}. The coincidence of the $\DR$ and the $\bOm$ hierarchies follows from Prop.~\ref{prop:coincide}.
\end{proof}

\begin{proposition}\label{prop:Witten} The partition function of the trivial P-CohFT satisfies the KdV equation.
\end{proposition}

\begin{proof} We have to compute the flux $R^1_{1,1}$ for the time $t^{1,1}$ in the notation of Eq.~\eqref{eq:R-alpha-beta-p}, for 
$\{O_{g,n}\} = \{\bPsi_{g,n}\}$. We have:
\begin{align} 
R^{1}_{1,1} & = \sum_{g=0}^\infty\sum_{n=1}^\infty \frac{\epsilon^{2g}}{n!} \sum_{\substack{q_1,\dots,q_n\geq 0 \\ q_1+\cdots+q_n=2g}} \prod_{i=1}^n w^{1,q_i}  
\int_{\oM_{g,n+2}} \Coeff{b_1^1 b_2^0\prod_{i=1}^n a_i^{q_i}} {}^{\bPsi}\!B^{2}_{g,n}.
\end{align}
For dimensional reasons, the only contributing graphs in ${}^{\bPsi}\!B^{2}_{g,n}$ are the one-vertex graphs for $(g,n)=(0,2)$ with $q_1=q_2=0$ and $(g,n)=(1,1)$ with $q_1=2$. Thus we have:
\begin{align} 
R^{1}_{1,1} & = \frac 12 (w^{1,0})^2 \int_{\oM_{0,4}} \psi_1^0\psi_2^0\psi_3^1\psi_4^0 + \epsilon^2 w^{1,2} \int_{\oM_{1,3}} \psi_1^2\psi_2^1\psi_3^0.  
\end{align}
Both integrals can be computed using the string and the dilaton equations and the initial value $\int_{\oM_{1,1}} \psi_1^1 = \frac 1{24}$. We have
\begin{align} 
R^{1}_{1,1} & = \frac 12 (w^{1,0})^2 + \frac{\epsilon^2}{12} w^{1,2}.
\end{align}
This implies
\begin{align}
    \d_{t^{1,1}} w^{1,1} = w^{1,0}w^{1,1} + \frac{\epsilon^2}{12} w^{1,3},
\end{align}
which is the celebrated Korteweg-de Vries equation, also found in the literature (substituting the $w$ with their definition and dropping the primary field index as the CohFT is one-dimensional) as
\begin{equation}
    \frac{\partial U}{\partial t^1} = U \frac{\partial U}{\partial t^0} + \frac{\epsilon^2}{12} \frac{\partial^3 U}{\partial (t^0)^3}, 
    \qquad\qquad
    U = \frac{\partial^2 (^{\bPsi}F)}{\partial (t^0)^2},
\end{equation}
and the potential $^{\bPsi}F$ evaluated at the trivial partial CohFT $\mathfrak{pc}_{g,n}(e_1^{\otimes n}) = 1.$
\end{proof}

Together with the string and the dilaton equations, the KdV equation uniquely determines the partition function of the trivial P-CohFT (that is, it uniquely determines all intersection numbers of $\psi$-classes), and these three equations together imply the whole KdV hierarchy~\cite{Witten-KdV}.

\begin{remark} Let us also guide the reader through the references that were used to achieve such a simple proof. One has to start with~\cite{BSS25}, which proves the geometric master relations for $\{\bPsi_{g,n}\}$. Then, using the combinatorial analysis of the emerging graphical formulas and factorization ruled for the $\DR$ classes over the divisors in the Losev-Manin space, one goes through the argument in~\cite{BLS-DRDZ,BLS-Omega} in order to prove~\cite[Conj.~1]{BS22}, which implies Thm.~4.7 in \emph{op.~cit.}, that is summarized in our more general setup as parts of Prop.~\ref{prop: integrability of integrable observables for F-CohFT} and~\ref{prop: integrability of integrable observables for P-CohFT} that we use in our proof of the Witten conjecture / Prop.~\ref{prop:Witten}. 
\end{remark}

\subsection{An example: \texorpdfstring{$\bOm$}{HHej}-hierarchy and the DR hierarchy do not coincide for multiple Hodge CohFTs}

It is quite difficult to make explicit computations with the $\bOm$ classes and ${}^\bOm\!B_{g,n}^m$ classes that are needed to study the properties of the $\bOm$-hierarchies. However, in the view of Prop.~\ref{prop:coincide} it is important to show that there are examples when $\bOm$-hierarchy doesn't coincide with the DR hierarchy in the normal coordinates. 

To this end, we consider the logarithm of the tau function ${}^\bOm\!F$ for the $N=1$ CohFT given by the products of Hodge classes
\begin{equation}
    \prod_{i=1}^M \Lambda(x_i), \qquad \qquad \Lambda(x) \coloneqq \sum_{j=0}^\infty x^j \lambda_j,
\end{equation}
For this CohFT we have the following formula:
\begin{align}
{}^\bOm\!F = {}^\bA\!F + R|_{w^{1,d}=w_{\mathrm{top}}^{1,d}},
\end{align}
where %we omit the primary field index for $w$ since $N=1$, and 
\begin{align}
    R &= \frac{\sum_{i<j} (x_i^2x_j+x_ix_j^2) + 2 \sum_{i<j<k} x_ix_jx_k}{362880} \epsilon^6 w^{1,4}  + \text{h.o.t.},
\end{align}

where the higher order terms (h.o.t.) are either of higher degree in $w^{1,d}$ or in $\epsilon$, $B_k$ are Bernoulli numbers and $e_i$ and $p_i$ are elementary symmetric and power sums, respectively. 

In order to compute this, one has to notice that the integrals 
\begin{align}
\sum_{g=0}^\infty\sum_{n=1}^\infty \frac{\epsilon^{2g}}{n!} \sum_{\substack{q_1,\dots,q_n\geq 0 \\ q_1+\cdots+q_n=2g-2}} \prod_{i=1}^n w^{1,q_i}
%\times \\ \notag & \quad \qquad  
\int_{\oM_{g,n}} \Coeff{\prod_{i=1}^n a_i^{q_i}} {}^{\bOm}\!B^{0}_{g,n} 
%\times \\ \notag & \qquad \qquad 
\prod_{i=1}^M \Lambda(x_i)
\end{align}
are non-trivial only if we take the homogenenous $\deg_{R^\bullet} = g-1+n$ part in $\prod_{i=1}^M \Lambda(x_i)$. Moreover, since ${}^{\bOm}\!B^{0}_{g,n}$ contains $\lambda_g$ as a factor, we have vanishing once we take at least one factor of $\lambda_g$ in $\prod_{i=1}^M \Lambda(x_i)$, or at least two factors of $\lambda_{g-1}$, since $\lambda_{g-1}^2 = 2 \lambda_{g}\lambda_{g-2}$. This immediately implies that the non-trivial contributions to $R$ are only in genus $g\geq 3$. 

In the latter case, for $g=3$ and $n=1$, the possible non-trivial contributions from $\prod_{i=1}^M \Lambda(x_i)$ are of the form $\lambda_2 \lambda_1$, multiplied by $x_i^2x_j$, or $\lambda_1^3 = 2\lambda_2\lambda_1$, multiplied by $x_ix_jx_k$. Then the only non-trivially contributing graph ${}^{\bOm}\!B^{0}_{3,1}$ is the one-vertex graph. In fact, for multiple vertices graphs we would have either $\lambda_1^2$ on a vertex of genus $1$, or $\lambda_2^2$ on a vertex of genus $2$. Therefore the integral reduces to 
\begin{align}
&
\int_{\oM_{3,1}} \Coeff{a_1^{4}} {}^{\bOm}\!B^{0}_{3,1} 
%\times \\ \notag & \qquad \qquad 
\prod_{i=1}^M \Lambda(x_i) 
\\ \notag &
= \int_{\oM_{3,1}} \Omega_{g,1}^{[1]}(1,0;-1) \lambda_3 \lambda_2 \lambda_1 \Big(\sum\nolimits_{i<j} (x_i^2x_j+x_ix_j^2) + 2 \sum\nolimits_{i<j<k} x_ix_jx_k\Big).
\end{align}
The integral can be computed using, for instance, that 
\begin{align}
\Omega_{g,1}^{[1]}(1,0;-1)_{\deg_{R^\bullet} = 1} - \Big(-\frac 1{12} \kappa_1 + \frac{13}{12} \psi_1\Big)
\end{align}
is supported on the boundary strata where $\lambda_3 \lambda_2 \lambda_1$ vanishes. Adding and substracting $\frac{\psi_1}{12}$, one obtains that 
$$
\Omega_{g,1}^{[1]}(1,0;-1)_{\deg_{R^\bullet} = 1} - \Big( \psi_1 - \lambda_1 \Big)
$$
is supported on the boundary strata. Since $\lambda_3\lambda_2\lambda_1^2$ vanishes by  $\lambda_1^2 = 2\lambda_2$, $\lambda_2^2 = 2\lambda_3\lambda_1$ and $\lambda_3^2 = 0$, we are left to compute the integral
\begin{align}
\int_{\oM_{3,1}} \lambda_3 \lambda_2 \lambda_1 \psi_1 = 4\int_{\oM_{3}} \lambda_3 \lambda_2 \lambda_1 = 2 \int_{\oM_{3}} \lambda_2^3 = 2 \frac{|B_4||B_6|}{24 \cdot 24} = \frac{1}{362880}
\end{align}
where we employed dilaton equation, again the relation $\lambda_2^2 = 2\lambda_3\lambda_1$ and the formula $\int_{\oM_{g}} \lambda_{g-1}^3 = \frac{|B_{2g}||B_{2g-2}|}{2g(2g-2)(2g-2)!}$ in~\cite{FabPan}.
Notice that for $M=1$, the leading term of $R$ we have just computed vanishes, in agreement with Prop.~\ref{prop:coincide}.

%%%%%%%%%%%%%%%%%%%%%%%%%%%%%%%%%%%%%%%%%%%%
\vspace{1cm}

% \bibliographystyle{alphaurl}
% \bibliography{HejBiblio}

\printbibliography

@article{AGV08,
 author = {Abramovich, D. and Graber, T. and Vistoli, A.},
 title = {Gromov-{Witten} theory of {Deligne}-{Mumford} stacks},
 fjournal = {American Journal of Mathematics},
 journal = {Am. J. Math.},
 issn = {0002-9327},
 volume = {130},
 number = {5},
 pages = {1337--1398},
 year = {2008},
 language = {English},
 doi = {10.1353/ajm.0.0017},
 zbMATH = {5363941},
 Zbl = {1193.14070}
}

@article {ArsLor2,
    AUTHOR = {Arsie, A. and Lorenzoni, P.},
     TITLE = {Flat {$F$}-manifolds, {M}iura invariants, and integrable
              systems of conservation laws},
   JOURNAL = {J. Integrable Syst.},
  FJOURNAL = {Journal of Integrable Systems},
    VOLUME = {3},
      YEAR = {2018},
    NUMBER = {1},
     PAGES = {xyy004, 38},
      ISSN = {2058-5985},
   MRCLASS = {35Q53 (35A30 37K05)},
  MRNUMBER = {3800641},
MRREVIEWER = {Xian\ Guo\ Geng},
       DOI = {10.1093/integr/xyy004},
       URL = {https://doi.org/10.1093/integr/xyy004},
}

@article {ArsLor1,
    AUTHOR = {Arsie, A. and Lorenzoni, P.},
     TITLE = {Complex reflection groups, logarithmic connections and bi-flat
              {F}-manifolds},
   JOURNAL = {Lett. Math. Phys.},
  FJOURNAL = {Letters in Mathematical Physics},
    VOLUME = {107},
      YEAR = {2017},
    NUMBER = {10},
     PAGES = {1919--1961},
      ISSN = {0377-9017,1573-0530},
   MRCLASS = {53D45 (20F55 37K25)},
  MRNUMBER = {3690038},
       DOI = {10.1007/s11005-017-0963-x},
       URL = {https://doi.org/10.1007/s11005-017-0963-x},
}

@article{ABLR,
 author = {Arsie, A. and Buryak, A. and Lorenzoni, P. and Rossi, P.},
 title = {Flat {F}-manifolds, {F}-{CohFTs}, and integrable hierarchies},
 fjournal = {Communications in Mathematical Physics},
 journal = {Commun. Math. Phys.},
 issn = {0010-3616},
 volume = {388},
 number = {1},
 pages = {291--328},
 year = {2021},
 language = {English},
 doi = {10.1007/s00220-021-04109-8},
 zbMATH = {7424951},
 Zbl = {1489.37083}
}

@article {ABLR-2,
    AUTHOR = {Arsie, A. and Buryak, A. and Lorenzoni, P.
              and Rossi, P.},
     TITLE = {Semisimple flat {F}-manifolds in higher genus},
   JOURNAL = {Comm. Math. Phys.},
  FJOURNAL = {Communications in Mathematical Physics},
    VOLUME = {397},
      YEAR = {2023},
    NUMBER = {1},
     PAGES = {141--197},
      ISSN = {0010-3616,1432-0916},
   MRCLASS = {53D45 (14N35)},
  MRNUMBER = {4538285},
MRREVIEWER = {Jie-Zhu\ Lin},
       DOI = {10.1007/s00220-022-04450-6},
       URL = {https://doi.org/10.1007/s00220-022-04450-6},
}

@article {Blot,
    AUTHOR = {Blot, X.},
     TITLE = {The quantum {W}itten-{K}ontsevich series and one-part double
              {H}urwitz numbers},
   JOURNAL = {Geom. Topol.},
  FJOURNAL = {Geometry \& Topology},
    VOLUME = {26},
      YEAR = {2022},
    NUMBER = {4},
     PAGES = {1669--1743},
      ISSN = {1465-3060,1364-0380},
   MRCLASS = {14H10 (05A99 14N35 53D55)},
  MRNUMBER = {4504448},
MRREVIEWER = {William\ Liu},
       DOI = {10.2140/gt.2022.26.1669},
       URL = {https://doi.org/10.2140/gt.2022.26.1669},
}

@misc{BLS-Quantum,
      title={Cohomological representations of quantum tau functions}, 
      author={Xavier Blot and Danilo Lewański and Sergey Shadrin},
      year={2024},
      eprint={2411.03499},
      archivePrefix={arXiv},
      primaryClass={math.AG},
      url={https://arxiv.org/abs/2411.03499}, 
}

@misc{BLS-DRDZ,
      title={On the strong DR/DZ equivalence conjecture}, 
      author={Xavier Blot and Danilo Lewa{\'n}ski and Sergey Shadrin},
      year={2024},
      eprint={2405.12334},
      archivePrefix={arXiv},
      primaryClass={math.AG},
      url={https://arxiv.org/abs/2405.12334}, 
}

@misc{BLS-Omega,
      title={Rooted trees with level structures, $\Omega$-classes and double ramification cycles}, 
      author={Xavier Blot and Danilo Lewa{\'n}ski and Sergey Shadrin},
      year={2024},
      eprint={2406.06205},
      archivePrefix={arXiv},
      primaryClass={math.AG},
      url={https://arxiv.org/abs/2406.06205}, 
}

@article{BLRS,
 author = {Blot, X. and Lewa{\'n}ski, D. and Rossi, P. and Shadrin, S.},
 title = {Stable tree expressions with omega-classes and double ramification cycles},
 fjournal = {Journal of Geometry and Physics},
 journal = {J. Geom. Phys.},
 issn = {0393-0440},
 volume = {209},
 pages = {17},
 note = {Id/No 105391},
 year = {2025},
 language = {English},
 doi = {10.1016/j.geomphys.2024.105391},
 zbMATH = {7972436},
 Zbl = {1566.14061}
}

@article {BSS25,
    AUTHOR = {Blot, Xavier and Sauvaget, Adrien and Shadrin, Sergey},
	TITLE = {The master relation for polynomiality and equivalences of
	integrable systems},
	JOURNAL = {Bull. Lond. Math. Soc.},
	FJOURNAL = {Bulletin of the London Mathematical Society},
	VOLUME = {57},
	YEAR = {2025},
	NUMBER = {2},
	PAGES = {599--604},
	ISSN = {0024-6093,1469-2120},
	MRCLASS = {14H10 (37K10 53D45)},
	MRNUMBER = {4861898},
	MRREVIEWER = {Alessandro\ Giacchetto},
	DOI = {10.1112/blms.13215},
	URL = {https://doi.org/10.1112/blms.13215},
}

@article{Bur,
 author = {Buryak, A.},
 title = {Double ramification cycles and integrable hierarchies},
 fjournal = {Communications in Mathematical Physics},
 journal = {Commun. Math. Phys.},
 issn = {0010-3616},
 volume = {336},
 number = {3},
 pages = {1085--1107},
 year = {2015},
 language = {English},
 doi = {10.1007/s00220-014-2235-2},
 zbMATH = {6424627},
 Zbl = {1329.14103}
}

@article{BDGR1,
 author = {Buryak, A. and Dubrovin, B. and Gu{\'e}r{\'e}, J. and Rossi, P.},
 title = {Tau-structure for the double ramification hierarchies},
 fjournal = {Communications in Mathematical Physics},
 journal = {Commun. Math. Phys.},
 issn = {0010-3616},
 volume = {363},
 number = {1},
 pages = {191--260},
 year = {2018},
 language = {English},
 doi = {10.1007/s00220-018-3235-4},
 zbMATH = {6967377},
 Zbl = {1431.53095}
}

@article{BDGR20,
 author = {Buryak, A. and Dubrovin, B. and Gu{\'e}r{\'e}, J. and Rossi, P.},
 title = {Integrable systems of double ramification type},
 fjournal = {IMRN. International Mathematics Research Notices},
 journal = {Int. Math. Res. Not.},
 issn = {1073-7928},
 volume = {2020},
 number = {24},
 pages = {10381--10446},
 year = {2020},
 language = {English},
 doi = {10.1093/imrn/rnz029},
 zbMATH = {7323449},
 Zbl = {1464.37071}
}

@article{BGR19,
 author = {Buryak, A. and Gu{\'e}r{\'e}, J. and Rossi, P.},
 title = {{DR}/{DZ} equivalence conjecture and tautological relations},
 fjournal = {Geometry \& Topology},
 journal = {Geom. Topol.},
 issn = {1465-3060},
 volume = {23},
 number = {7},
 pages = {3537--3600},
 year = {2019},
 language = {English},
 doi = {10.2140/gt.2019.23.3537},
 zbMATH = {7152164},
 Zbl = {1469.14065}
}

@article{BPS1,
 author = {Buryak, A. and Posthuma, H. and Shadrin, S.},
 title = {A polynomial bracket for the {Dubrovin}-{Zhang} hierarchies},
 fjournal = {Journal of Differential Geometry},
 journal = {J. Differ. Geom.},
 issn = {0022-040X},
 volume = {92},
 number = {1},
 pages = {153--185},
 year = {2012},
 language = {English},
 doi = {10.4310/jdg/1352211225},
 zbMATH = {6130615},
 Zbl = {1259.53079}
}

@article {BPS2,
    AUTHOR = {Buryak, A. and Posthuma, H. and Shadrin, S.},
     TITLE = {On deformations of quasi-{M}iura transformations and the
              {D}ubrovin-{Z}hang bracket},
   JOURNAL = {J. Geom. Phys.},
  FJOURNAL = {Journal of Geometry and Physics},
    VOLUME = {62},
      YEAR = {2012},
    NUMBER = {7},
     PAGES = {1639--1651},
      ISSN = {0393-0440,1879-1662},
   MRCLASS = {37K10 (37K05 37K35)},
  MRNUMBER = {2922026},
MRREVIEWER = {Hao\ Xu},
       DOI = {10.1016/j.geomphys.2012.03.006},
       URL = {https://doi.org/10.1016/j.geomphys.2012.03.006},
}

@article{BR-spin,
 author = {Buryak, A. and Rossi, P.},
 title = {Extended {{\(r\)}}-spin theory in all genera and the discrete {KdV} hierarchy},
 fjournal = {Advances in Mathematics},
 journal = {Adv. Math.},
 issn = {0001-8708},
 volume = {386},
 pages = {48},
 note = {Id/No 107794},
 year = {2021},
 language = {English},
 doi = {10.1016/j.aim.2021.107794},
 zbMATH = {7367632},
 Zbl = {1473.14048}
}

@article{BR-partial,
 author = {Buryak, A. and Rossi, P.},
 title = {Quadratic double ramification integrals and the noncommutative {KdV} hierarchy},
 fjournal = {Bulletin of the London Mathematical Society},
 journal = {Bull. Lond. Math. Soc.},
 issn = {0024-6093},
 volume = {53},
 number = {3},
 pages = {843--854},
 year = {2021},
 language = {English},
 doi = {10.1112/blms.12464},
 zbMATH = {7381913},
 Zbl = {1469.14066}
}

@article{BS22,
 author = {Buryak, A. and Shadrin, S.},
 title = {Tautological relations and integrable systems},
 fjournal = {{\'E}pijournal de G{\'e}om{\'e}trie Alg{\'e}brique. EPIGA},
 journal = {{\'E}pijournal de G{\'e}om. Alg{\'e}br., EPIGA},
 issn = {2491-6765},
 volume = {8},
 pages = {44},
 note = {Id/No 12},
 year = {2024},
 language = {English},
 doi = {10.46298/epiga.2024.10382},
 zbMATH = {7939129},
 Zbl = {1552.14017}
}

@article {BSSZ,
	AUTHOR = {Buryak, A. and Shadrin, S. and Spitz, L. and Zvonkine, D.},
	TITLE = {Integrals of {$\psi$}-classes over double ramification cycles},
	JOURNAL = {Amer. J. Math.},
	FJOURNAL = {American Journal of Mathematics},
	VOLUME = {137},
	YEAR = {2015},
	NUMBER = {3},
	PAGES = {699--737},
	ISSN = {0002-9327,1080-6377},
	MRCLASS = {14H10 (14C25)},
	MRNUMBER = {3357119},
	MRREVIEWER = {Yongnam\ Lee},
	DOI = {10.1353/ajm.2015.0022},
	URL = {https://doi.org/10.1353/ajm.2015.0022},
}

@article{Chiodo08,
 author = {Chiodo, A.},
 title = {Towards an enumerative geometry of the moduli space of twisted curves and {{\(r\)}}th roots},
 fjournal = {Compositio Mathematica},
 journal = {Compos. Math.},
 issn = {0010-437X},
 volume = {144},
 number = {6},
 pages = {1461--1496},
 year = {2008},
 language = {English},
 doi = {10.1112/S0010437X08003709},
 zbMATH = {5382620},
 Zbl = {1166.14018}
}

@article {Dubrovin-Zhang-g1,
    AUTHOR = {Dubrovin, B. and Zhang, Y.},
     TITLE = {Bi-{H}amiltonian hierarchies in {$2$}{D} topological field
              theory at one-loop approximation},
   JOURNAL = {Comm. Math. Phys.},
  FJOURNAL = {Communications in Mathematical Physics},
    VOLUME = {198},
      YEAR = {1998},
    NUMBER = {2},
     PAGES = {311--361},
      ISSN = {0010-3616,1432-0916},
   MRCLASS = {53D45 (37K10 37K30 57R56)},
  MRNUMBER = {1672512},
MRREVIEWER = {Alexandre\ I.\ Kabanov},
       DOI = {10.1007/s002200050480},
       URL = {https://doi.org/10.1007/s002200050480},
}

@incollection {Dub-TFT,
    AUTHOR = {Dubrovin, B.},
     TITLE = {Geometry of {$2$}{D} topological field theories},
 BOOKTITLE = {Integrable systems and quantum groups ({M}ontecatini {T}erme,
              1993)},
    SERIES = {Lecture Notes in Math.},
    VOLUME = {1620},
     PAGES = {120--348},
 PUBLISHER = {Springer, Berlin},
      YEAR = {1996},
      ISBN = {3-540-60542-8},
   MRCLASS = {58D29 (14N10 20F55 32G34 58F07 81T40)},
  MRNUMBER = {1397274},
MRREVIEWER = {N.\ J.\ Hitchin},
       DOI = {10.1007/BFb0094793},
       URL = {https://doi.org/10.1007/BFb0094793},
}

@misc{dubrovin2001normalformshierarchiesintegrable,
      title={Normal forms of hierarchies of integrable {PDE}s, {F}robenius manifolds and {G}romov - {W}itten invariants}, 
      author={Dubrovin, B. and Zhang, Y.},
      year={2001},
      eprint={math/0108160},
      archivePrefix={arXiv},
      primaryClass={math.DG},
    howpublished = {Preprint, {arXiv}:math/0108160 [math.{DG}]},
      url={https://arxiv.org/abs/math/0108160}, 
      note={See also the 2005 updated version.},
}

@incollection {Eliashberg-SFT,
    AUTHOR = {Eliashberg, Y.},
     TITLE = {Symplectic field theory and its applications},
 BOOKTITLE = {International {C}ongress of {M}athematicians. {V}ol. {I}},
     PAGES = {217--246},
 PUBLISHER = {Eur. Math. Soc., Z\"urich},
      YEAR = {2007},
      ISBN = {978-3-03719-022-7},
   MRCLASS = {53D40 (53D35 53D45)},
  MRNUMBER = {2334192},
MRREVIEWER = {Michael\ J.\ Usher},
       DOI = {10.4171/022-1/10},
       URL = {https://doi.org/10.4171/022-1/10},
}

@article {EguchiXiong,
    AUTHOR = {Eguchi, T. and Xiong, C.-S.},
     TITLE = {Quantum cohomology at higher genus: topological recursion
              relations and {V}irasoro conditions},
   JOURNAL = {Adv. Theor. Math. Phys.},
  FJOURNAL = {Advances in Theoretical and Mathematical Physics},
    VOLUME = {2},
      YEAR = {1998},
    NUMBER = {1},
     PAGES = {219--229},
      ISSN = {1095-0761,1095-0753},
   MRCLASS = {14H15 (14C17 14N10 17B68 32G81)},
  MRNUMBER = {1635867},
MRREVIEWER = {Alexandre\ I.\ Kabanov},
       DOI = {10.4310/ATMP.1998.v2.n1.a9},
       URL = {https://doi.org/10.4310/ATMP.1998.v2.n1.a9},
}

@incollection {Getzler,
    AUTHOR = {Getzler, E.},
     TITLE = {Operads and moduli spaces of genus {$0$} {R}iemann surfaces},
 BOOKTITLE = {The moduli space of curves ({T}exel {I}sland, 1994)},
    SERIES = {Progr. Math.},
    VOLUME = {129},
     PAGES = {199--230},
 PUBLISHER = {Birkh\"{a}user Boston, Boston, MA},
      YEAR = {1995},
      ISBN = {0-8176-3784-2},
   MRCLASS = {18C10 (14H10 18D99 18G10 55P99)},
  MRNUMBER = {1363058},
MRREVIEWER = {J.\ Stasheff},
       DOI = {10.1007/978-1-4612-4264-2\{_}8}

@article {Giv,
    AUTHOR = {Givental, A.~B.},
     TITLE = {Gromov-{W}itten invariants and quantization of quadratic
              {H}amiltonians},
      NOTE = {Dedicated to the memory of I. G. Petrovskii on the occasion of
              his 100th anniversary},
   JOURNAL = {Mosc. Math. J.},
  FJOURNAL = {Moscow Mathematical Journal},
    VOLUME = {1},
      YEAR = {2001},
    NUMBER = {4},
     PAGES = {551--568, 645},
      ISSN = {1609-3321,1609-4514},
   MRCLASS = {53D45 (14N35)},
  MRNUMBER = {1901075},
MRREVIEWER = {Domenico\ Fiorenza},
       DOI = {10.17323/1609-4514-2001-1-4-551-568},
       URL = {https://doi.org/10.17323/1609-4514-2001-1-4-551-568},
}

@article{GraVakil05,
 author = {Graber, T. and Vakil, R.},
 title = {Relative virtual localization and vanishing of tautological classes of moduli spaces of curves},
 fjournal = {Duke Mathematical Journal},
 journal = {Duke Math. J.},
 issn = {0012-7094},
 volume = {130},
 number = {1},
 pages = {1--37},
 year = {2005},
 language = {English},
 doi = {10.1215/S0012-7094-05-13011-3},
 zbMATH = {5004321},
 Zbl = {1088.14007}
}

@article{GraPandha99,
 author = {Graber, T. and Pandharipande, R.},
 title = {Localization of virtual classes},
 fjournal = {Inventiones Mathematicae},
 journal = {Invent. Math.},
 issn = {0020-9910},
 volume = {135},
 number = {2},
 pages = {487--518},
 year = {1999},
 language = {English},
 doi = {10.1007/s002220050293},
 zbMATH = {1275767},
 Zbl = {0953.14035}
}

@article{GLN23,
 author = {Giacchetto, A. and Lewa{\'n}ski, D. and Norbury, P.},
 title = {An intersection-theoretic proof of the {Harer}-{Zagier} formula},
 fjournal = {Algebraic Geometry},
 journal = {Algebr. Geom.},
 issn = {2313-1691},
 volume = {10},
 number = {2},
 pages = {130--147},
 year = {2023},
 language = {English},
 doi = {10.14231/AG-2023-004},
 zbMATH = {7656907},
 Zbl = {1506.14109}
}

@incollection{Hain,
 author = {Hain, R.},
 title = {Normal functions and the geometry of moduli spaces of curves},
 booktitle = {Handbook of moduli. Volume I},
 pages = {527--578},
 year = {2015},
 publisher = {Somerville, MA: International Press; Beijing: Higher Education Press},
 language = {English},
 zbMATH = {6492642},
 Zbl = {1322.14049}
}

@article{JPPZ,
 author = {Janda, F. and Pandharipande, R. and Pixton, A. and Zvonkine, D.},
 title = {Double ramification cycles on the moduli spaces of curves},
 fjournal = {Publications Math{\'e}matiques},
 journal = {Publ. Math., Inst. Hautes {\'E}tud. Sci.},
 issn = {0073-8301},
 volume = {125},
 pages = {221--266},
 year = {2017},
 language = {English},
 doi = {10.1007/s10240-017-0088-x},
 zbMATH = {6766883},
 Zbl = {1370.14029}
}

@article{JPT,
 author = {Johnson, P. and Pandharipande, R. and Tseng, H.-H.},
 title = {Abelian {Hurwitz}-{Hodge} integrals},
 fjournal = {Michigan Mathematical Journal},
 journal = {Mich. Math. J.},
 issn = {0026-2285},
 volume = {60},
 number = {1},
 pages = {171--198},
 year = {2011},
 language = {English},
 doi = {10.1307/mmj/1301586310},
 zbMATH = {5899030},
 Zbl = {1222.14119}
}

@article {KontsevichManin,
	AUTHOR = {Kontsevich, M. and Manin, Y.},
	TITLE = {Gromov-{W}itten classes, quantum cohomology, and enumerative
	geometry},
	JOURNAL = {Comm. Math. Phys.},
	FJOURNAL = {Communications in Mathematical Physics},
	VOLUME = {164},
	YEAR = {1994},
	NUMBER = {3},
	PAGES = {525--562},
	ISSN = {0010-3616,1432-0916},
	MRCLASS = {14N10 (53C15 58D10 58F05)},
	MRNUMBER = {1291244},
	MRREVIEWER = {Dietmar\ A.\ Salamon},
	URL = {http://projecteuclid.org/euclid.cmp/1104270948},
}

@article {LiuRuanZhang,
    AUTHOR = {Liu, S.-Q. and Ruan, Y. and Zhang, Y.},
     TITLE = {B{CFG} {D}rinfeld-{S}okolov hierarchies and {FJRW}-theory},
   JOURNAL = {Invent. Math.},
  FJOURNAL = {Inventiones Mathematicae},
    VOLUME = {201},
      YEAR = {2015},
    NUMBER = {2},
     PAGES = {711--772},
      ISSN = {0020-9910,1432-1297},
   MRCLASS = {81T45 (14N35 37K10 53D45)},
  MRNUMBER = {3370624},
MRREVIEWER = {Xiaobin\ Li},
       DOI = {10.1007/s00222-014-0559-3},
       URL = {https://doi.org/10.1007/s00222-014-0559-3},
}

@article {LiuWang,
    AUTHOR = {Liu, X. and Wang, C.},
     TITLE = {On a tautological relation conjectured by {B}uryak and
              {S}hadrin},
   JOURNAL = {Sci. China Math.},
  FJOURNAL = {Science China. Mathematics},
    VOLUME = {69},
      YEAR = {2026},
    NUMBER = {2},
     PAGES = {285--308},
      ISSN = {1674-7283,1869-1862},
   MRCLASS = {14N35 (14H10)},
  MRNUMBER = {5012765},
       DOI = {10.1007/s11425-024-2436-x},
       URL = {https://doi.org/10.1007/s11425-024-2436-x},
}

@article {LiuWangZhang,
    AUTHOR = {Liu, S.-Q. and Wang, Z. and Zhang, Y.},
     TITLE = {Linearization of {V}irasoro symmetries associated with
              semisimple {F}robenius manifolds},
   JOURNAL = {Adv. Math.},
  FJOURNAL = {Advances in Mathematics},
    VOLUME = {460},
      YEAR = {2025},
     PAGES = {Paper No. 110046, 32},
      ISSN = {0001-8708,1090-2082},
   MRCLASS = {53D45 (37K25 37K30)},
  MRNUMBER = {4831844},
MRREVIEWER = {Iskander\ A.\ Taimanov},
       DOI = {10.1016/j.aim.2024.110046},
       URL = {https://doi.org/10.1016/j.aim.2024.110046},
}

@article {LorenzoniPedtoniRaimondo,
    AUTHOR = {Lorenzoni, P. and Pedroni, M. and Raimondo, A.},
     TITLE = {{$F$}-manifolds and integrable systems of hydrodynamic type},
   JOURNAL = {Arch. Math. (Brno)},
  FJOURNAL = {Universitatis Masarykianae Brunensis. Facultas Scientiarum
              Naturalium. Archivum Mathematicum},
    VOLUME = {47},
      YEAR = {2011},
    NUMBER = {3},
     PAGES = {163--180},
      ISSN = {0044-8753,1212-5059},
   MRCLASS = {53D45 (37K25)},
  MRNUMBER = {2852379},
MRREVIEWER = {Dafeng\ Zuo},
}

@misc{lorenzoni2026generalisedbihamiltonianstructureshydrodynamic,
      title={Generalised (bi-){H}amiltonian structures of hydrodynamic type and (bi-)flat {F}-manifolds}, 
      author={Lorenzoni, P. and Wang, Z.},
      year={2026},
      eprint={2604.12819},
      archivePrefix={arXiv},
      primaryClass={math-ph},
      url={https://arxiv.org/abs/2604.12819}, 
    howpublished = {Preprint, {arXiv}:2604.12819 [math-ph]},
}

@misc{ShaWang,
      title={Structure of {D}ubrovin-{Z}hang free energy functions and universal identities}, 
      author={Shadrin, S. and Wang, Z.},
      year={2025},
      eprint={2405.00276},
      archivePrefix={arXiv},
      primaryClass={math-ph},
      url={https://arxiv.org/abs/2405.00276}, 
    howpublished = {Preprint, {arXiv}:2405.00276 [math-ph]},
}

@article {Tel,
    AUTHOR = {Teleman, C.},
     TITLE = {The structure of 2{D} semi-simple field theories},
   JOURNAL = {Invent. Math.},
  FJOURNAL = {Inventiones Mathematicae},
    VOLUME = {188},
      YEAR = {2012},
    NUMBER = {3},
     PAGES = {525--588},
      ISSN = {0020-9910,1432-1297},
   MRCLASS = {57R56 (18D10 53D45)},
  MRNUMBER = {2917177},
MRREVIEWER = {Julia\ Bergner},
       DOI = {10.1007/s00222-011-0352-5},
       URL = {https://doi.org/10.1007/s00222-011-0352-5},
}

@incollection {Witten-KdV,
	AUTHOR = {Witten, E.},
	TITLE = {Two-dimensional gravity and intersection theory on moduli
	space},
	BOOKTITLE = {Surveys in differential geometry ({C}ambridge, {MA}, 1990)},
	PAGES = {243--310},
	PUBLISHER = {Lehigh Univ., Bethlehem, PA},
	YEAR = {1991},
	ISBN = {0-8218-0168-6},
	MRCLASS = {32G15 (14C17 14H15 32G81 58F07 81T40)},
	MRNUMBER = {1144529},
	MRREVIEWER = {Steven\ Rosenberg},
}

@article {LPSZ17,
    AUTHOR = {Lewa\'nski, D. and Popolitov, A. and Shadrin, S.
              and Zvonkine, D.},
     TITLE = {Chiodo formulas for the {$r$}-th roots and topological
              recursion},
   JOURNAL = {Lett. Math. Phys.},
  FJOURNAL = {Letters in Mathematical Physics},
    VOLUME = {107},
      YEAR = {2017},
    NUMBER = {5},
     PAGES = {901--919},
      ISSN = {0377-9017,1573-0530},
   MRCLASS = {14H10 (14N10 14N35)},
  MRNUMBER = {3633029},
       DOI = {10.1007/s11005-016-0928-5},
       URL = {https://doi.org/10.1007/s11005-016-0928-5},
}

@misc{yang2026kdvintegrabilityguecorrelators,
      title={Kd{V} integrability in GUE correlators}, 
      author={Yang, D.},
      year={2026},
      eprint={2603.24956},
      archivePrefix={arXiv},
      primaryClass={math-ph},
      url={https://arxiv.org/abs/2603.24956}, 
    howpublished = {Preprint, {arXiv}:2603.24956 [math-ph]},
}

@article {AHIS,
    AUTHOR = {Alexandrov, A. and Hern\'{a}ndez Iglesias, F.
              and Shadrin, S.},
     TITLE = {Buryak-{O}kounkov formula for the {$n$}-point function and a
              new proof of the {W}itten conjecture},
   JOURNAL = {Int. Math. Res. Not. IMRN},
  FJOURNAL = {International Mathematics Research Notices. IMRN},
      YEAR = {2021},
    NUMBER = {18},
     PAGES = {14296--14315},
      ISSN = {1073-7928,1687-0247},
   MRCLASS = {14N35 (14H10 37K10)},
  MRNUMBER = {4320810},
MRREVIEWER = {Dragos\ Nicolae\ Oprea},
       DOI = {10.1093/imrn/rnaa024},
       URL = {https://doi.org/10.1093/imrn/rnaa024},
}

@article {Kontsevich,
    AUTHOR = {Kontsevich, M.},
     TITLE = {Intersection theory on the moduli space of curves and the
              matrix {A}iry function},
   JOURNAL = {Comm. Math. Phys.},
  FJOURNAL = {Communications in Mathematical Physics},
    VOLUME = {147},
      YEAR = {1992},
    NUMBER = {1},
     PAGES = {1--23},
      ISSN = {0010-3616,1432-0916},
   MRCLASS = {32G15 (14H15 58F07 81T40)},
  MRNUMBER = {1171758},
MRREVIEWER = {Claude\ Itzykson},
       URL = {http://projecteuclid.org/euclid.cmp/1104250524},
}

@incollection {OkoPan,
    AUTHOR = {Okounkov, A. and Pandharipande, R.},
     TITLE = {Gromov-{W}itten theory, {H}urwitz numbers, and matrix models},
 BOOKTITLE = {Algebraic geometry---{S}eattle 2005. {P}art 1},
    SERIES = {Proc. Sympos. Pure Math.},
    VOLUME = {80},
     PAGES = {325--414},
 PUBLISHER = {Amer. Math. Soc., Providence, RI},
      YEAR = {2009},
      ISBN = {978-0-8218-4702-2},
   MRCLASS = {14N35 (14H15)},
  MRNUMBER = {2483941},
MRREVIEWER = {Hsian-Hua\ Tseng},
       DOI = {10.1090/pspum/080.1/2483941},
       URL = {https://doi.org/10.1090/pspum/080.1/2483941},
}

@article {Mirz,
    AUTHOR = {Mirzakhani, M.},
     TITLE = {Weil-{P}etersson volumes and intersection theory on the moduli
              space of curves},
   JOURNAL = {J. Amer. Math. Soc.},
  FJOURNAL = {Journal of the American Mathematical Society},
    VOLUME = {20},
      YEAR = {2007},
    NUMBER = {1},
     PAGES = {1--23},
      ISSN = {0894-0347,1088-6834},
   MRCLASS = {14H15 (14N35 32G15)},
  MRNUMBER = {2257394},
MRREVIEWER = {Hsian-Hua\ Tseng},
       DOI = {10.1090/S0894-0347-06-00526-1},
       URL = {https://doi.org/10.1090/S0894-0347-06-00526-1},
}

@article {KazLan,
    AUTHOR = {Kazarian, M. E. and Lando, S. K.},
     TITLE = {An algebro-geometric proof of {W}itten's conjecture},
   JOURNAL = {J. Amer. Math. Soc.},
  FJOURNAL = {Journal of the American Mathematical Society},
    VOLUME = {20},
      YEAR = {2007},
    NUMBER = {4},
     PAGES = {1079--1089},
      ISSN = {0894-0347,1088-6834},
   MRCLASS = {14H70 (14H10 37K10)},
  MRNUMBER = {2328716},
MRREVIEWER = {Hsian-Hua\ Tseng},
       DOI = {10.1090/S0894-0347-07-00566-8},
       URL = {https://doi.org/10.1090/S0894-0347-07-00566-8},
}

@article {FabPan,
    AUTHOR = {Faber, C. and Pandharipande, R.},
     TITLE = {Hodge integrals and {G}romov-{W}itten theory},
   JOURNAL = {Invent. Math.},
  FJOURNAL = {Inventiones Mathematicae},
    VOLUME = {139},
      YEAR = {2000},
    NUMBER = {1},
     PAGES = {173--199},
      ISSN = {0020-9910,1432-1297},
   MRCLASS = {14N35},
  MRNUMBER = {1728879},
MRREVIEWER = {Jim\ A.\ Bryan},
       DOI = {10.1007/s002229900028},
       URL = {https://doi.org/10.1007/s002229900028},
}

@misc{Pixton,
  author       = {Pixton, A.},
  title        = {DR cycle polynomiality and related results},
  year = {2023},
  howpublished = {Published on A. Pixton personal website, \url{https://websites.umich.edu/~pixton/papers/DRpoly.pdf}},
  url          = {https://websites.umich.edu/~pixton/papers/DRpoly.pdf},
}

@misc{spelier2024polynomialitydoubleramificationcycle,
      title={Polynomiality of the double ramification cycle}, 
      author={Pim Spelier},
      year={2024},
      eprint={2401.17421},
      archivePrefix={arXiv},
      primaryClass={math.AG},
      url={https://arxiv.org/abs/2401.17421}, 
}

%%%%%%%%%%%%%%%%%%%%%%%%%%%%%%%%%%%%%%%%%%%%
\end{document}